\documentclass[11 pt,fullpage]{article}
\usepackage{amsfonts,amssymb}
\usepackage{amsmath,amsthm}
\usepackage{color} 
\usepackage{hyperref}
\usepackage[right = 2.5cm, left=2.5cm, top = 2.5cm, bottom =2.5cm]{geometry}
\theoremstyle{plain}
\newtheorem{theorem}{Theorem}

\newtheorem{proposition}{Proposition}[section]
\newtheorem{lemma}[proposition]{Lemma}

\newtheorem{corollary}{Corollary}

\theoremstyle{definition}
\newtheorem{remark}{Remark}

\allowdisplaybreaks

\numberwithin{equation}{section}

\newcommand\R{{\mathbb R}}

\newcommand\Torus{{\mathbb T}}

\newcommand{\cT}{\mathcal T}

\def\eps{{\varepsilon}}

\renewcommand\eps{\epsilon }

\newcommand{\Real}{\mathbb R}
\newcommand{\Complex}{\mathbb C}
\newcommand{\Integer}{\mathbb Z}
\newcommand{\Integers}{\mathbb Z}
\newcommand{\norm}[1]{\left\lVert#1\right\rVert}
\newcommand{\abs}[1]{\left\vert#1\right\vert}
\newcommand{\set}[1]{\left\{#1\right\}}

\newcommand{\grad}{\nabla}

\newcommand{\Naturals}{\mathbb N}
\newcommand{\jap}[1]{\langle #1 \rangle} 
\newcommand{\brak}[1]{\langle #1 \rangle}

\newcommand{\dss}{\displaystyle}

 \newcommand{\dd}{{\, \mathrm d}}

\begin{document}
 
\title{Nonlinear echoes and Landau damping with insufficient regularity} 
\author{Jacob Bedrossian\footnote{\textit{jacob@cscamm.umd.edu}, University of Maryland, College Park. This work was partially supported by NSF grant DMS-1462029 and an Alfred P. Sloan Fellowship. Additionally, the research was supported in part by NSF RNMS \#1107444 (Ki-Net).}} 

\date{\today}
\maketitle

\begin{abstract} 
We prove that the theorem of Mouhot and Villani on Landau damping near equilibrium for the Vlasov-Poisson equations on $\Torus_x \times \Real_v$ cannot, in general, be extended to high Sobolev spaces in the case of gravitational interactions.
This is done by showing in every Sobolev space, there exists background distributions such that one can construct arbitrarily small perturbations that exhibit arbitrarily many isolated nonlinear oscillations in the density. These oscillations are known as plasma echoes in the physics community.
For the  case of electrostatic interactions, we demonstrate a sequence of small background distributions and asymptotically smaller perturbations in $H^s$ which display similar nonlinear echoes.
This shows that in the electrostatic case, any extension of Mouhot and Villani's theorem to Sobolev spaces would have to depend crucially on some additional non-resonance effect coming from the background -- unlike the case of Gevrey-$\nu$ with $\nu < 3$ regularity, for which results are uniform in the size of small backgrounds.
In particular, the uniform dependence on small background distributions obtained in Mouhot and Villani's theorem in Gevrey class is false in Sobolev spaces. 
\end{abstract}

\setcounter{tocdepth}{1}
{\tiny\tableofcontents}

\section{Introduction} 
In this paper we study the Vlasov-Poisson equations near a homogeneous equilibrium $f^0(v)$, an important problem in both plasma physics \cite{Ryutov99,BoydSanderson,Stix} and stellar mechanics \cite{LyndenBell67}. 
We will consider the phase-space $(x,v) \in \Torus_x \times \Real_v$ (with $\Torus_x$ normalized to length $2\pi$).
If the distribution function $F$ is written as $F(t,x,v) = f^0(v) + h(t,x,v)$, where $h$ is assumed to be a mean-zero fluctuation, then the Vlasov equations for $h$ are
\begin{equation}\label{def:VPE}
  \left\{
\begin{array}{l} \dss 
\partial_t h + v\cdot \grad_x h + E(t,x)\cdot \grad_v \left(f^0 + h\right)  = 0, \\[3mm]
E(t,x) := -(\grad_x W \ast_{x} \rho)(t,x), \\[1mm] 
\dss \rho(t,x) := \int_{\R} h(t,x,v) \dd v, \\[3.5mm] 
h(t=0,x,v) = h_{\mbox{{\scriptsize in}}}(x,v).
\end{array}
\right.
\end{equation}
The potential $W$ describes the mean-field interaction between particles; we will consider:
\begin{align}
\widehat{W}(k) = \zeta \abs{k}^{-2}, \quad k \neq 0,\label{def:W}
\end{align}
with $\zeta \in \set{-1,+1}$; $-1$ corresponds to gravitational interactions in stellar mechanics (with the Jeans swindle; see e.g. \cite{Kiessling2003}) and $+1$ corresponds to electrostatic interactions between electrons in a quasi-neutral plasma (after making an electrostatic approximation and neglecting collisions and ion acceleration). 
We are interested in studying how the asymptotic behavior of small perturbations $h$ is dependent on the regularity of the initial data, in particular, whether or not the evolution
agrees with the linearized Vlasov equations as $t \rightarrow \pm \infty$. 
In this work, we prove that such results are in general false for \eqref{def:VPE} if $h_{in}$ is only taken to be small in Sobolev spaces.   

\subsection{Background}
Denote $\cT_t$ as the free transport group: 
\begin{align}
h \circ \cT_t = h(x+tv,v). 
\end{align}
Then, the $h = h_{in} \circ \cT_{-t}$ solves the free transport equation 
\begin{equation} \label{def:freetrans}
\left\{
\begin{array}{l}
\partial_t h + v\cdot \grad_x h = 0 \\ 
h(0,x,v) = h_{in}(x,v).
\end{array}
\right.
\end{equation}
By direct computation, one verifies that the Fourier transform satisfies $\widehat{h_{in} \circ \cT_{-t}}(t,k,\eta) = \widehat{h_{in}}(k,\eta+kt)$. 
Hence, if $h_{in}$ is analytic and $x \in \Torus^d$, then the density, $\int h(t,x,v) dv$, decays exponentially to its average forward and backward in time. 
 In fact, the density becomes smoother as the radius of analyticity increases linearly with time, as emphasized in \cite{MouhotVillani11}. 
The transfer of information to high frequencies in the distribution function and subsequent decay/smoothing of the density is a simple example of ``phase-mixing'', as pointed out in \cite{VKampen55} (see also \cite{Degond86,CagliotiMaffei98}). 
Phase mixing is a general phenomenon which occurs in a variety of systems, see e.g. \cite{TataronisGrossman73,StrogatzEtAl92,BMT13,BM13} and the references therein.

\subsubsection{Existing work on Landau damping}

In 1946, Landau \cite{Landau46} observed that the linearized Vlasov equations (with $f^0$ Maxwellian) with analytic initial data predict that 
the density fluctuation, $\rho$, decays exponentially fast (again for $x \in \Torus^d$  where the spatial frequencies are bounded below; see \cite{Glassey94,Glassey95,BMM16}). 
For example, Landau's observations imply that for every analytic solution $h(t,x,v)$ to the linearized Vlasov equations, there exists some analytic $h_\infty$ such that the following holds for some $\lambda,c > 0$, which expresses the fact that solutions to the linearized Vlasov equations rapidly converge to free transport:   
\begin{align}   
\norm{e^{\lambda\abs{\grad}}\left(h(t) \circ \cT_{t} - h_\infty\right)}_{L^2}  \lesssim \norm{e^{(\lambda+c)\abs{\grad}}\left(\jap{v}^{d} h_{in}\right)}_{L^2}e^{-\frac{1}{2}ct}.  \label{id:scatter}
\end{align}
The decay of the electric field was experimentally confirmed in \cite{MalmbergWharton64}, and is now known as \emph{Landau damping}. 
It is considered to be fundamentally important to the kinetic theory of weakly collisional plasmas, see e.g. \cite{Ryutov99,BoydSanderson,Stix}, and also is thought to be important in stellar mechanics \cite{LyndenBell67}. 
A number of other works regarding the linearized Vlasov equations have since followed Landau, providing mathematically rigorous treatments and various generalizations \cite{VKampen55,backus,Penrose,Maslov,Degond86}, and for $x \in \Torus^d$, the linear problem is essentially completely understood. 
Similar results to \eqref{id:scatter} also hold in Sobolev spaces (see e.g. \cite{LZ11b} and the references therein); the Landau damping rate is $\jap{t}^{-\sigma}$ in $H^\sigma$, as observed in solutions to \eqref{def:freetrans}. 
One can also draw analogies with scattering in dispersive equations (see e.g. \cite{lax90,TaoTextbook,ReedSimonIII}), as was pointed out in \cite{Degond86,CagliotiMaffei98,MouhotVillani11}. 
We remark that for $x \in \Real^d$, the linear problem is less well understood: in \cite{Glassey94,Glassey95} it is shown in that in general, \eqref{def:freetrans} is not a good approximation for the linearized Vlasov equations.  
However, these pathologies are absent if $W \in L^1(\Real^d)$; see \cite{BMM16}. 

The nonlinear dynamics near equilibrium are far less well-understood. 
It is not at all clear that the linearization should remain a good approximation to \eqref{def:VPE} for long times, even for small data (this has long been recognized in the physics literature \cite{Penrose,Stix,backus}). 
There exist steady states with non-trivial densities, known as BGK waves \cite{BGK57} and it was shown in \cite{LZ11b} that such equillibria are arbitrarily close to the homogeneous Maxwellian in $H^s(\Torus \times \Real)$ with $s < 3/2$, and hence a nonlinear analogue of Landau's results with these regularities is false. 
The existence of solutions to \eqref{def:VPE} which exhibit Landau damping was proved in \cite{CagliotiMaffei98,HwangVelazquez09}.
These works essentially prove: for linearly stable $f^0$, and analytic $h_\infty$, there is a unique solution to \eqref{def:VPE} which satisfies \eqref{id:scatter}. 
That all initial data which is small enough exhibits Landau damping for $(x,v) \in \Torus^d \times \Real^d$ was first proved by Mouhot and Villani \cite{MouhotVillani11}, provided one takes the initial data small in Gevrey-$\nu$ for some $\nu$ close to $1$. 
Moreover, Mouhot and Villani predicted from nonlinear heuristics that something may go wrong due to nonlinear effects if $\nu > 3$. 
The results of \cite{MouhotVillani11} were later extended to cover the predicted range of $\nu \in [1,3)$ in \cite{BMM13} and further extended to a relativistic model in \cite{Young14}. 

Several works have explored the unusually stringent regularity requirement. 
In \cite{MR3437866}, it was shown that if $W$ is compactly supported in frequency, then Landau damping holds for small data in Sobolev spaces. 
Analogous small data Sobolev space results have also been proved for the mean-field Kuramoto model \cite{MR3471147,FGVG}.  
For $W \in L^1$, Landau damping for $(x,v) \in \Real_x^3 \times \Real^3_v$ was proved for small data in Sobolev spaces recently in \cite{BMM16}. 
If one is not interested in quantifying the rapid decay of high frequencies in $\rho(t,x)$, then one can also prove a dispersive decay result in the case $f^0 \equiv 0$ \cite{BardosDegond1985}.

\subsubsection{Plasma echoes}
Mouhot and Villani's weakly nonlinear heuristics are based on a ``resonance'' known as a plasma echo, which was first discovered and isolated experimentally by Malmberg et. al. in 1968 \cite{MalmbergWharton68}. 
Landau damping is due to the transfer of $O(1)$ spatial information to small scales in the velocity distribution. 
This mixing is time-reversible, and hence un-mixing induces a transient growth. 
This growth is essentially the \emph{Orr mechanism} in fluid mechanics, identified by Orr in 1907 \cite{Orr07} (see \cite{Boyd83,BM13} for more discussion). 
More precisely, consider a solution $h(t,x,v)$ to the free transport equation \eqref{def:freetrans} and denote $\rho(t,x) = \int_{\Real} h(t,x,v) dv$. 
Then, 
\begin{align}
\widehat{\rho}(t,k) = \widehat{h_{in}}(k,kt).  \label{eq:rholinkkt}
\end{align}
This implies that the $(k,\eta)$ frequency of $\widehat{h_{in}}$ determines the density near the time $t \sim \frac{\eta}{k}$, which Orr called the \emph{critical time}. 
Orr pointed out that if one concentrates information near $(k,\eta)$ with $\eta \gg k \geq 1$, then this induces a large density fluctuation at $t \sim  \frac{\eta}{k}$ (his work was on the linearized 2D Euler equations near Couette flow -- in that setting, $h$ is the vorticity). 

A plasma echo occurs in \eqref{def:VPE} when a nonlinear effect transfers information to a mode which is un-mixing, 
as this leads to a large response in the future when that mode reaches its critical time (hence `echo').  
These echoes can chain into a repeated cascade, as demonstrated experimentally in both the Vlasov equations \cite{MalmbergWharton68} and 2D Euler near a vortex \cite{YuDriscollONeil,YuDriscoll02}. 
The idea that a transient linear growth can be repeatedly re-excited and amplified by nonlinear effects is by now a classical idea in fluid mechanics; see e.g. \cite{TTRD93,BaggettEtAl,VMW98,Trefethen2005,Vanneste02,SchmidHenningson2001} and the references therein, as well as \cite{BM13,BGM15I}. 
In \cite{MouhotVillani11}, Mouhot and Villani studied this possibility for \eqref{def:VPE} and estimated that an infinite cascade of echoes could potentially transfer so much information to un-mixing modes that one could maybe expect nonlinear instabilities unless the initial data were at least Gevrey-$3$ for $\widehat{W}(k) = \pm \abs{k}^{-2}$ (more generally, for $\widehat{W}(k) = \pm \abs{k}^{-1-\gamma_0}$ with $\gamma_0 \geq 1$, the prediction was Gevrey-$(2+\gamma_0)$). 
However, it was not clear that the heuristics keep enough structure to make an accurate prediction. 

\subsection{Main results}
The existing Landau damping results in Sobolev spaces, \cite{MR3437866,MR3471147,FGVG,BMM16}, are all in settings that either avoid, or suppress in some way, the nonlinear echoes.  
Indeed, the models in \cite{MR3437866,MR3471147,FGVG} do not support infinite echo cascades and in $\Real^3_x \times \Real^3_v$, it turns out there is an additional dispersive mechanism which greatly weakens the effective strength of the echoes \cite{BMM16}. 
In this work, we prove that in the original setting studied by Mouhot and Villani \cite{MouhotVillani11}, small perturbations $h$ in \eqref{def:VPE}, in general, do \emph{not} behave like the linearized Vlasov equations if the initial data is only assumed to be small in a Sobolev space. 
Hence, for long times, the linearization is not valid even for arbitrarily small data and the results of \cite{MouhotVillani11,BMM16} do not extend to finite regularity results on $\Torus_x \times \Real_v$ (in general). 

We will study the following background density: 
\begin{align*}
f^0(v) = \frac{4 \pi \delta}{(1+v^2)},
\end{align*} 
where $0 < \delta \ll 1$ will be chosen small later.
This distribution is chosen since the linear problem can be solved explicitly (see e.g. Lemma \ref{lem:basicvolt} below or \cite{Glassey94,Glassey95}). 
Our main result is the following. 

\begin{theorem}[Nonlinear echoes in Sobolev spaces] \label{thm:main} 
 The following two theorems hold $\forall R \geq 1$ and $\forall p \in (0,1)$:
\begin{itemize} 
\item[(i)] Suppose $\widehat{W}(k) = -\abs{k}^{-1-\gamma_0}$ with $\gamma_0 \geq 1$.  $\exists \sigma_0(R) \gg R$ such that for all $\sigma \geq \sigma_0$, there are constants $\epsilon_0(R,\sigma), \delta_0(\sigma) \ll 1$ such that if  $\epsilon \leq \epsilon_0$ and $0 < \delta  \leq \delta_0$ there exists a real analytic $h_{in}$ with $f^0 + h_{in}$ strictly positive and $h_{in}$ satisfying the quantitative bound
    \begin{subequations} \label{ineq:hinest}
\begin{align}
\norm{\jap{v} h_{in}}_{H^{\sigma}} & \leq \epsilon^{1-p} \\
\norm{\jap{v} h_{in}}_{H^{\sigma-3}} & \lesssim \epsilon,
\end{align}
\end{subequations}
but such that at some finite time $t_\star = t_\star(\epsilon,R)$ satisfying $\epsilon t_\star \rightarrow \infty$ as $\epsilon \rightarrow 0$, the solution to \eqref{def:VPE} satisfies the following for all $z \geq 0$:  
\begin{subequations}
\begin{align}
\norm{h(t_\star) \circ \mathcal{T}_{t_\star}}_{H^{\sigma-R+z}} & \gtrsim t_\star^{z} \gg \epsilon^{-z}, \label{ineq:hcircTexplode0} \\ 
\abs{\widehat{E}(t_\star,\pm 1)}& \gtrsim t_\star^{R-\sigma}. \label{ineq:Etstardefct0}
\end{align} 
\end{subequations} 
\item[(ii)] Suppose $\widehat{W}(k) = \pm \abs{k}^{-1-\gamma_0}$ with $\gamma_0 \geq 1$. Then, $\exists \sigma_0(R) \gg R$ such that for all $\sigma \geq \sigma_0$, there is a constant $\epsilon_0(R,\sigma) \ll 1$ such that for all $\epsilon \leq \epsilon_0$ and $0 < \delta  \leq \epsilon^p$, there exists a real analytic $h_{in}$ with $f^0 + h_{in}$ strictly positive and $h_{in}$ satisfying the quantitative bound
\begin{align}
\norm{\jap{v} h_{in}}_{H^{\sigma}} \leq \epsilon \label{ineq:hinest2}
\end{align}
but such that at some finite time $t_\star = t_\star(\epsilon,R)$ satisfying $\epsilon t_\star \rightarrow \infty$ as $\epsilon \rightarrow 0$, the solution to \eqref{def:VPE} satisfies the following for all $z \geq 0$:  
\begin{subequations}
\begin{align}
\norm{h(t_\star) \circ \mathcal{T}_{t_\star}}_{H^{\sigma-R+z}} & \gtrsim t_\star^{z} \gg \epsilon^{-z}, \label{ineq:hcircTexplode} \\ 
\abs{\widehat{E}(t_\star,\pm 1)} & \gtrsim t_\star^{R-\sigma}. \label{ineq:Etstardefct}
\end{align} 
\end{subequations}
\end{itemize} 
\end{theorem}

\begin{remark}
Theorem \ref{thm:main} (i) shows that in the gravitational case, for all $\sigma$ sufficiently high, there exists fixed $f^0$ such that arbitrarily small Sobolev space perturbations deviate arbitrarily far from linearized predictions. Hence, Mouhot and Villani's theorem \cite{MouhotVillani11} does not extend to Sobolev spaces for gravitational cases.
It is easy to check that the constants in \cite{MouhotVillani11,BMM13} are uniform in $f^0 \rightarrow 0$, that is, one obtains constants which do not depend on $f^0$ for small $f^0$. Theorem \ref{thm:main} (ii) shows that this uniformity in $f^0$ is false in Sobolev spaces.
Hence in the electrostatic case, Theorem \ref{thm:main} shows that if one has any hope of extending their theorems to Sobolev regularity, one needs to depend on some kind of additional non-resonance coming from the linearized Vlasov equations. 
\end{remark} 

\begin{remark}
Note that \eqref{ineq:hcircTexplode0} does not imply \eqref{ineq:Etstardefct0} (but vice-versa is essentially the case for $z=0$), hence, it is important to prove that the rate of decay can really be altered by the regularity class.
\end{remark}

\begin{remark}
Theorem \ref{thm:main} (i) shows that one cannot Landau damp arbitrarily fast with data in Sobolev spaces, however, it does \emph{not} rule out solutions which are nevertheless ``scattering'' to free transport in lower Sobolev norms and hence Landau damping at a slow rate. Such solutions could for example, involve nonlinear oscillations that persist for all times while nevertheless having a decaying force-field.
It is easy to construct distribution functions with behavior like this, but it does not seem clear whether or not this is possible in the Vlasov equations. To our knowledge, there is no a priori reason to believe that there is a dichotomy between non-decay of the force-field and scattering in arbitrarily high regularity as in Mouhot and Villani \cite{MouhotVillani11}. 
\end{remark}

\begin{remark} 
It is easy, using e.g. the methods of \cite{BMM13,MR3437866}, to prove that \eqref{ineq:hinest} implies the solution stays close to a solution of the linearized Vlasov equations for times $t \ll \epsilon^{-1}$. The nonlinear echoes in Theorem \ref{thm:main} occur on the time-scales $\epsilon^{-1} \ll \epsilon^{-1/3}t_\star^{2/3} \lesssim t \lesssim t_\star$.  
\end{remark} 

\begin{remark} \label{rmk:densityosc}
The proof provides an accurate approximation of the solution for long times (see Proposition \ref{prop:stabg} below). For example, at times $\frac{t_\star}{k}$ for $k \in \Naturals$ with $1 \leq k \lesssim (\epsilon t_\star)^{1/3}$, the density fluctuation $\rho(t,x)$ (and hence also the force field $E$) is close to a multiple of $\cos(kx)$, whereas away from these times, the density fluctuation is exponentially small. These nonlinear oscillations in the density are quite distinct from the prediction of the linearized Vlasov equations.
Note, this behavior is very similar to that observed in the classical experiments \cite{MalmbergWharton68}. 
\end{remark}

\begin{remark} 
The fact that $f^0 + h_{in} > 0$ and $h$ is real analytic emphasizes that the Gevrey-$\nu$ regularity requirements appearing in \cite{MouhotVillani11,BMM13,Young14} are quantitative, not qualitative. 
\end{remark}

\begin{remark} 
In terms of $h$ and $E$, the echo instability is a \emph{high-to-low} frequency cascade: the slow decay of the force field is due to information which has been mixed to $O(\epsilon^{-1})$ scales in the velocity distribution returning to $O(1)$ spatial scales (moreover, the length-scale of the density oscillations increases in time as pointed out in Remark \ref{rmk:densityosc}).
However, in terms of $h \circ \mathcal{T}_t$, the echo instability is a low-to-high cascade: a large transfer of information to high frequencies, as suggested by \eqref{ineq:hcircTexplode}. 
In this manner, the results (and to a certain degree, the proof itself), have some analogies with norm growth results for nonlinear Schr\"odinger equations; see e.g. \cite{CKSTT2010,GuardiaKaloshin2015,HaniEtAl2015} and the references therein. 
\end{remark} 

\begin{remark} 
In the non-physical case of $f^0 = 0$, it is likely relatively straightforward to adapt the proof of Theorem \ref{thm:main} to Gevrey-$1/s$ with $s < 1/3$ (more generally Gevrey-$(2+\gamma_0)$). 
An analogue for $s = 1/3$ exactly might also be possible. However, since $f^0 = 0$ is non-physical, we have not followed this further. 
In the case $f^0 = O(1)$, a number of minor additional complications arise and even in the gravitational case we felt that there is insufficient interest to warrant the additional complexity. 
\end{remark}

\begin{remark} 
Gevrey regularity requirements have arisen in a sequence of works on the 2D and 3D Couette flow in the Euler equations \cite{BM13} and Navier-Stokes equations at high Reynolds number \cite{BMV14,BGM15I,BGM15II}. Mixing due to the mean shear flow induces an effect similar to Landau damping, known as \emph{inviscid damping} \cite{BM13,BouchetMorita10,Briggs70,SchecterEtAl00,WeiZhangZhao15}, as well as a variety of other effects, such as enhanced viscous dissipation, vortex stretching, and additional algebraic instabilities (see \cite{BMV14,BGM15I} and the references therein).  
 Theorem \ref{thm:main} and its proof have already provided intuition for better understanding related nonlinear instabilities in fluid mechanics -- see the recently completed preprint \cite{DM18}. The two are related but not the same in proof or results: \cite{DM18}  demonstrates a larger loss of regularity but over shorter time-scales that do not contain the isolated echoes as those here and in the experiments \cite{MalmbergWharton68,YuDriscoll02} and only excites high frequencies in both variables (our solutions excite the first mode in $x$ and high frequencies in $v$). 
\end{remark}

The proof of Theorem \ref{thm:main} is outlined in \S\ref{sec:Proof} below. 
The first step is to find a sufficiently accurate approximate solution to \eqref{def:VPE} (for long times) of the form $f^E(t) \circ \mathcal{T}_{-t}$ for which it is possible to show that $f^E$ exhibits a nonlinear echo cascade and loses large amounts of regularity. 
The second step is to prove that the true solution remains close to such an $f^E$ in norms strong enough and times long enough to clearly see the instability.
This requires a number of ideas, one of the most important being to carefully segregate the low and high frequencies in $f^E$ and to precisely localize the error between $f^E$ and the true solution, $h \circ \mathcal{T}_t$, in frequency.  
One of the techniques adapted for this is a norm built on a time-dependent Fourier multiplier matched to the critical times (see \S\ref{sec:norms}  below). 
A related Fourier multiplier norm was introduced in \cite{BM13} and a number of variations were subsequently used to study the stability of the Couette flow in the Navier-Stokes equations in 2D \cite{BMV14,Zillinger2014,BVW16} and 3D \cite{BGM15I,BGM15II,BGM15III}. 
The norms were used to introduce dissipation-like terms into energy estimates and/or to unbalance the regularity between specific frequencies and/or components of the solution at specific (frequency-dependent) times \cite{BM13,BGM15II}. 
We will use such norms in a different way, adapting them for a bootstrap scheme similar to that employed in \cite{BMM13}. 
In particular, the specific structure of the Fourier norm we employ is only relevant for comparing the density at time $t$ to the density at a previous time $\tau$ -- the crucial step in estimating the effect of the plasma echoes in \cite{MouhotVillani11,BMM13,BMM16}.

As an additional application of the techniques we employ, we prove the following theorem, which is a sharper characterization of the regularity required for Landau damping. 
For simplicity, we only consider $d = 1$ and $f^0$ small; extensions to more general cases should be possible but may not be entirely straightforward. 
The proof of Theorem \ref{thm:main} strongly suggests that Theorem \ref{thm:optimal} is optimal modulo the precise values of $K$ and $\sigma$.  
The proof of Theorem \ref{thm:optimal} is sketched briefly in \S\ref{sec:thmopt}.

\begin{theorem} \label{thm:optimal}
For any $\abs{\widehat{W}(k)} \lesssim \abs{k}^{-1-\gamma_0}$ with $\gamma_0 \geq 1$ and all $\sigma > 9/2$, there exists a large constant $K = K(W)$ (depending only on $W$) and small constants 
$\delta_0 \geq \epsilon_0 > 0$ (depending on $\sigma$ and $W$) with $K\epsilon_0 < 1$ such that the following holds: if $\jap{v}f^0 \in H^{\sigma+2}(\Real)$, and mean-zero $h_{in}$ satisfy
\begin{align*}
\norm{\jap{v} e^{(K\epsilon)^{1/(2+\gamma_0)}\jap{\grad}^{1/(2+\gamma_0)}}h_{in}}_{H^{\sigma}} & \leq \epsilon \leq \epsilon_0, \\ 
\norm{\jap{v}e^{(K\epsilon)^{1/(1+\gamma_0)}\jap{\grad}^{1/(2+\gamma_0)}} f^0}_{H^{\sigma+2}} & = \delta \leq \delta_0, 
\end{align*}
then there exists $h_\infty \in H^{\sigma}$ such that, 
\begin{align*}
\norm{\jap{v} e^{\frac{1}{4}(K\epsilon)^{1/(2+\gamma_0)}\jap{\grad}^{1/(2+\gamma_0)}}\left(h\circ \mathcal{T}_t - h_\infty\right)}_{H^{\sigma-3}} & \lesssim \epsilon \jap{t}^{-3/2 - (\sigma-3)} e^{-\frac{1}{8}(K\epsilon)^{1/(2+\gamma_0)}t^{1/(2+\gamma_0)}}, \\
\abs{\hat{\rho}(t,k)} & \lesssim \epsilon \jap{t}^{-\sigma+3} e^{-\frac{1}{2}(K\epsilon)^{1/(2+\gamma_0)}t^{1/(2+\gamma_0)}}.
\end{align*}
\end{theorem}


\subsubsection*{Basic notations, conventions, and shorthands}
We denote $\Naturals = \set{0,1,2,\dots}$ (including zero) and
$\Integer_\ast = \Integer \setminus \set{0}$.  For $\xi \in \Complex$
we use $\bar{\xi}$ to denote the complex conjugate.  
For a vectors $x = (x_1,x_2,...x_d)$ we use $\abs{x}$  to denote the $\ell^1$ norm,  
We denote $\jap{v} = \left( 1 + \abs{v}^2 \right)^{1/2}$  
and furthermore use the shorthand $\abs{k,\ell} = \abs{(k,\ell)}$ and  $\jap{k,\ell} = \jap{(k,\ell)}$.  
We will use similar notation for $L^2$ inner product: $\jap{f,g}_{2} := \int_{\Torus \times \Real} \overline{f} g \, dx dv$.
Fourier analysis conventions are set in \S\ref{apx:Gev}.
For any locally bounded function $m(t,k,\eta)$ we denote the Fourier multiplier: 
\begin{align*}
mf = m(t,\partial_x,\partial_v)f = m(t,\grad)f := \left(m(t,k,\eta)\hat{f}(t,k,\eta)\right)^{\vee}.
\end{align*}
If $\rho = \rho(t,x)$ is a function only of $x$ then we use the definition: 
\begin{align*}
m \rho = m(t,\partial_x, t\partial_x) \rho.   
\end{align*}
Sobolev norms are given as $\norm{f}^2_{H^\sigma} = \norm{\jap{\grad}^\sigma f}_{L^2}$. We will often use the short-hand $\norm{\cdot}_2$ for
$\norm{\cdot}_{L^2_{z,v}}$ or $\norm{\cdot}_{L^2_v}$ depending on the
context. To deal with moments, we will often abuse notation and write 
\begin{align*}
\norm{\jap{v} f}_{H^\sigma_{x,v}}^2 = \int_{\Torus \times \Real} \abs{\jap{\grad}^{\sigma}(\jap{v}f)}^2 dx dv \approx_{\sigma} \int_{\Torus \times \Real} \jap{v}^2\abs{\jap{\grad}^{\sigma}f}^2 dx dv. 
\end{align*}
We use the notation $f \lesssim g$ when there exists a constant
$C > 0$ such that $f \leq Cg$ (we analogously define $f \gtrsim g$).  Similarly, we use
the notation $f \approx g$ when there exists $C > 0$ such that
$C^{-1}g \leq f \leq Cg$.  We sometimes use the notation
$f \lesssim_{\alpha} g$ if we want to emphasize that the implicit
constant depends on some parameter $\alpha$.

\section{Outline} \label{sec:Proof}
The case of gravitational interactions, $\zeta = -1$ in \eqref{def:W} is slightly easier, for reasons explained in \S\ref{sec:electro}.
Hence, we first carry out the proof in the gravitational case and then explain the technical refinements necessary to extend the proof Theorem \ref{thm:main} (ii) to the electrostatic case in \S\ref{sec:electro}.
Moreover, the case of general $\gamma_0$ is a straightforward variant of the proof for $\gamma_0 = 1$, and henceforth mainly only consider the case $\gamma_0 = 1$ (which is also the most important case). 

As is often the case when studying Landau damping \cite{CagliotiMaffei98,HwangVelazquez09,MouhotVillani11,BMM13,BMM16}, the quantity of interest is $f = h \circ \mathcal{T}_{t}$, and hence we make the coordinate transformation $z = x-tv$ and $f(t,z,v) = h(t,z+tv,v)$. 
The Vlasov equation \eqref{def:VPE} becomes, 
\begin{equation} \label{def:glideVP}
\left\{
\begin{array}{l}\dss
\partial_t f + E(t,z+tv)\partial_v f^0 + E(t,z+tv)(\partial_v - t\partial_z)f = 0, \\ 
\rho(t,x)  = \int_{\Real} f(t,z-tw,w) dw, \\ 
f(0,z,v) = h_{in}. 
\end{array}
\right.
\end{equation}
Note that with this definition we have (compare with \eqref{eq:rholinkkt}), 
\begin{align}
\hat{\rho}(t,k) = \hat{f}(t,k,kt). \label{eq:rhokkt}
\end{align}
The proof is based on finding an approximate solution, $f^E$, exhibiting an echo-driven instability and writing
\begin{align}
f = f^E + g, \label{def:ffEg}
\end{align}
where $g$ is the correction. 
One of the primary difficulties is that $g$ will tend to lose a little more regularity than $f^E$, and hence in order to get sufficient control at the final time, we need to be able to propagate more regularity on $g$ than we actually have on $f^E$. 
Due to this difficulty, and a few others, there are several subtleties to making this scheme work.  

\textbf{Step 1: Accessing an unstable configuration} \\ 
\noindent
The echo instability is driven by the interaction of high frequency perturbations with a larger, low frequency, spatially dependent wave over long time scales.  
The large, low frequency wave produces a strong force field near time zero and hence makes it difficult to easily prescribe initial data which will produce an explicitly computable echo cascade. 
For this reason, we will specify $f(t_{in})$ at time $t_{in} = \epsilon^{-q}$ for some fixed $q$ satisfying $0 < q < \min(1/4,p)$.  
The initial condition is then found by solving the nonlinear final-time problem: 
\begin{equation} \label{def:fHST}
\left\{
\begin{array}{l} \dss
\partial_t f + E(t,z+tv)(\partial_v - t\partial_z)(f^0 + f) = 0, \quad\quad t < t_{in}, \\
\rho(t,x) = \int_{\Real} f(t,z-tw,w) dw \\  
f(t_{in})  = f^L(z,v) + f^H_{in}(z,v),  
\end{array}
\right.
\end{equation}
where we choose
\begin{subequations} 
\begin{align}
f^L (z,v) & = 8\pi\epsilon\frac{\cos(z)}{1+v^2} \\  
f^H_{in}(z,v) & = \frac{\epsilon}{\jap{k_0,\eta_0}^{\sigma}} \frac{\cos(k_0z) \cos(\eta_0 v)}{1 + 4v^2}, \label{def:fHin}
\end{align}
\end{subequations}
for $\eta_0$, $k_0$ large parameters to be chosen later. 
Next, we need to verify that $h_{in} := f(0)$ chosen in this way satisfies the hypotheses of Theorem \ref{thm:main}, supplied by the following proposition proved in \S\ref{sec:Access}.
Notice that since we are starting with well-mixed data and rewinding to un-mixed data, the proposition is \emph{not} a consequence of \cite{MouhotVillani11,BMM13} and is more like \cite{CagliotiMaffei98,HwangVelazquez09}. 
\begin{proposition}[Accessibility of unstable configuration] \label{lem:accessID}
Let $p \in (0,1)$,  $t_{in} = \epsilon^{-q}$ with $0 < q < \min(1/4,p)$, and $\sigma > 6$. Then for all $\epsilon>0$ chosen sufficiently small (depending on $p$, $q$, and $\sigma$),
the following holds:
\begin{itemize}
\item[(i)] Under the hypotheses of Theorem \ref{thm:main} (i), the solution to \eqref{def:fHST} satisfies 
\begin{subequations}
\begin{align}
\sup_{t \in [0,t_{in})} \norm{\jap{v}f(t)}_{H^{\sigma}} & \leq 4\epsilon^{1-p}, \label{ineq:AccessID} \\ 
\sup_{t \in [0,t_{in})} \norm{\jap{v}f(t)}_{H^{\sigma-3}} & \leq 4\epsilon. \label{ineq:AccessID2}
\end{align}
\end{subequations}
\item[(ii)] Under the hypotheses of Theorem \ref{thm:main} (ii), the solution to \eqref{def:fHST} satisfies 
\begin{align}
\sup_{t \in [0,t_{in})} \norm{\jap{v}f(t)}_{H^{\sigma}} & \leq 4\epsilon. \label{ineq:AccessID3}
\end{align}
\end{itemize}
Hence the initial condition satisfies the stated property after slightly re-defining $\eps$. 
\end{proposition}

\textbf{Step 2: Construction of the high frequency approximate solution} \\ 
\noindent
For times $t > t_{in}$, we will set our approximate `echo' solution to 
\begin{align*}
f^E(t,z,v) = f^L(z,v) + f^H(t,z,v), 
\end{align*}
where the high frequencies $f^H$ are chosen to satisfy the linear ``second-iterate'' system 
\begin{equation} \label{eq:seciter}
\left\{
\begin{array}{l} \dss
\partial_t f^H + E^H(t,z+tv)\partial_v f^0 + E^H(t,z+tv)(\partial_v-t\partial_z)f^L = 0, \quad t > t_{in} \\
\rho^H(t,x) = \int_{\Real} f^H(t,z-tw,w) dw \\  
E^H(t,x) = -(\grad_x W \ast_{x} \rho^H)(t,x), \\
f^H(t_{in}) = f^H_{in}. 
\end{array}
\right.
\end{equation} 
We define $\rho^L$ and $E^L$ in an analogous manner. 
The system \eqref{eq:seciter} arises by linearizing \eqref{def:glideVP} around the approximate solution $f^L$ and retaining only the terms expected to remain relevant for long times. 
Note that the third term in \eqref{eq:seciter} corresponds to the term referred to as ``reaction'' in \cite{MouhotVillani11}. 
For a universal constant $K_m' > 0$ (determined by the proof; see Proposition \ref{prop:blowup}), we set 
\begin{align}
k_0 = \textup{Floor}\left( (K_m'\epsilon)^{1/3} \eta_0^{1/3}\right), \label{def:k0}
\end{align}
where we will eventually have $K_m' \epsilon \leq 1$. 
Next, fix $\eta_0 = \eta_0(R,\epsilon)$ such that 
\begin{align}
\epsilon \jap{k_0,\eta_0}^{-R} e^{3(K_m'\epsilon)^{1/3}\eta_0^{1/3}} = 1. \label{def:eta0}
\end{align} 
By $R \geq 1$, we have $\eta_0 \epsilon \rightarrow \infty$ as $\epsilon \rightarrow 0$, however, for all $c > 1$, there holds $\eta_0 = o_{\epsilon \rightarrow 0}(\epsilon^{-c})$.
Finally set $t_\star := \eta_0^{-1}$. 
In \S\ref{sec:fH}, the following is proved regarding the behavior of $f^H$ over long times. 
\begin{proposition}[High frequency instability] \label{lem:AppHi} 
There exists a universal constant $K_m' > 0$ such that for all $\epsilon > 0$ sufficiently small, the solution $f^H$ of \eqref{eq:seciter} with $f^H_{in}$ chosen as in \eqref{def:fHin} satisfies the pointwise-in-frequency lower bounds (recalling \eqref{def:eta0} and \eqref{def:sigmabeta}), 
\begin{subequations} \label{ineq:lowbd}
\begin{align}
\abs{\widehat{f^H}(t_\star,1,\eta)} & \geq  \epsilon \jap{k_0,\eta_0}^{-\sigma} \left(e^{3(K_m'\epsilon)^{1/3}\eta_0^{1/3}} e^{-\abs{\eta-\eta_0}} - e^{-\frac{1}{8}\abs{\eta-\eta_0}}\right), \\ 
\abs{\widehat{f^H}(t_\star,-1,\eta)} & \geq \epsilon \jap{k_0,\eta_0}^{-\sigma} \left(e^{3(K_m'\epsilon)^{1/3}\eta_0^{1/3}} e^{-\abs{\eta+\eta_0}} - e^{-\frac{1}{8}\abs{\eta+\eta_0}}\right). 
\end{align}
\end{subequations}
There also exists a universal constant $K_m > K_m' > 0$ such that $f^H$ satisfies the upper bound, 
\begin{align*}
\abs{\widehat{f^H}(t,k,\eta)} \lesssim \epsilon \jap{k_0,\eta_0}^{-\sigma} e^{3(K_m\epsilon)^{1/3}\eta_0^{1/3}} \left( e^{-\frac{1}{8}\abs{\eta-\eta_0}} +  e^{-\frac{1}{8}\abs{\eta+\eta_0}}\right). 
\end{align*}
More precise estimates can be found in \S\ref{sec:fH}. 
\end{proposition} 
\begin{remark} 
The lower bounds \eqref{ineq:lowbd} in the electrostatic case are slightly different and are significantly harder to obtain; see \S\ref{sec:electro}. 
\end{remark} 

\textbf{Step 3: Stability of approximate solution}\\
\noindent
Now that $f^E$ exhibiting the instability has been constructed, the next major step is to prove that the true solution stays close to $f^E$ in $H^{\sigma-R}$. 
This is done via an energy estimate on $g$ using a very precise norm adapted to control any `secondary' echo instabilities that the solution may undergo.     
From the definition of $f^H$ and $f^L$, the perturbation $g$ in \eqref{def:ffEg} satisfies (denoting $E^E = E^L + E^H$ and $f^E = f^H + f^L$), 
\begin{equation} \label{def:g}
\left\{
\begin{array}{l} \dss
\partial_t g + E_g(z+tv)\partial_v f^0 + E_g(z+tv)(\partial_v - t\partial_z) f^E + E^E(z+tv)(\partial_v - t\partial_z) g  \\ \quad\quad\quad + E_g(z+tv)(\partial_v - t\partial_z) g  = -\mathcal{E}, \quad t > t_{in} \\ 
\rho_g(t,x) = \int_{\Real} g(t,z-tw,w) dw, \\  
E_g(t,x) = -(\partial_x W \ast_{x} \rho_g)(t,x), \\
g(t_{in},z,v) = 0.
\end{array}
\right.
\end{equation}
where $\mathcal{E}$ (the `consistency error' of $f^E$) satisfies 
\begin{align}
\mathcal{E} & = E^L(z+tv)(\partial_v - t\partial_z)(f^L + f^H) + E^L(z+tv)  \partial_v f^0 +  E^H(z+tv)(\partial_v-t\partial_z) f^H.\label{def:cE}
\end{align}
As we will see, due to the restriction $t > t_{in}$, the error $\mathcal{E}$ is very small even in norms significantly stronger than $H^{\sigma-R}$; see \S\ref{sec:ErrorfH}. 

The main difficulty is that $g$ is expected to lose more regularity than $f^H$ itself (or at least, we cannot rule this out) and hence we need to propagate \emph{higher} regularity on $g$ than we have on $f^H$. 
The obvious problems this presents can be surmounted by, among other things, the fact that both $f^H$ and $g$ are very small in weaker norms -- this quantifies that both $f^H$ and $g$ are concentrated at high frequencies.  
With paraproduct Fourier decompositions, this small-ness can be used to overcome the large size of $f^H$ in higher norms. 
In order for this scheme work, however, we cannot measure $g$ in a norm which is too strong, or equivalently, we cannot allow $g$ to lose much more regularity than $f^H$. 
Proposition \ref{lem:AppHi} suggests that we need to quantify the regularity loss of $g$ in exactly Gevrey-$3$ with an $O(\epsilon^{1/3})$ radius of regularity so that the norm will be close to $H^{\sigma-R}$ at frequencies comparable to $\eta_0$; see \S\ref{sec:norms}. 
However, a standard Gevrey-$3$ norm does not capture the dynamics accurately enough -- for example, the energy method of \cite{BMM13} breaks down in Gevrey-$3$.  
To overcome this, we will use ideas from the related works in fluid mechanics \cite{BM13,BMV14,BGM15I,BGM15II}, especially the inviscid damping work \cite{BM13}, which employ norms built on time-dependent Fourier multipliers designed to match the loss of regularity precisely.
As discussed above, these techniques require some adaptation to the current setting, moreover, some technical refinements are necessary in order to treat the borderline case of Gevrey-$3$ with exactly $O(\epsilon^{1/3})$ radius of regularity (note that none of the previous works obtain results in the borderline regularity). 

The energy estimate on \eqref{def:g} involves the interplay of two norms, a low and a high norm, both built on Fourier multipliers,
\begin{align*} 
\norm{A(t,\grad) g}_{L^2}, \quad \quad \norm{B(t,\grad) g}_{L^2}. 
\end{align*}
The multiplier $B$ defines the low norm, a Gevrey-3 norm with a carefully tuned radius of regularity: for parameters $\gamma$, $\nu(t)$, and $K$ determined by the proof below (see \S\ref{sec:norms}), 
\begin{align*} 
B(t,\grad) = \jap{\grad}^{\gamma} e^{\nu(t)(K\epsilon)^{1/3}\jap{\grad}^{1/3}}. 
\end{align*} 
The parameters will be tuned such that $\norm{B(t,\grad)f^H}_{L^2}$ is small (see Lemma \ref{lem:AfHBfH} below) and hence we will be able to deduce something similar for $g$. 
The high norm, defined via $A(t,\grad)$, uses the ideas introduced in \cite{BM13}.  
See \S\ref{sec:norms} for details on the definition.  
Among a variety of other properties, $A$ satisfies $A(t, \pm 1, \eta_0) \gg \jap{k_0,\eta_0}^{\sigma-R+\alpha} $, for some range of $\alpha$ with $0 \leq \alpha \lesssim R$. 
Ultimately, this will ensure that sufficient control on $A(t,\grad)g$ implies that $\norm{f(t_\star)}_{H^{\sigma-R}} \gtrsim 1$. 
However, it also implies that $\norm{Af^H}_{2}$ is large. 
To compensate for this, the multipliers $A$ and $B$ are tuned so that a product rule-type inequality roughly of the following form holds (see \S\ref{sec:norms} below):
\begin{align*}
\norm{A(q_1 q_2)}_{2} & \lesssim \norm{Aq_1}_2 \norm{\jap{\grad}^{-1}Bq_2}_2 + \norm{Aq_2}_2 \norm{\jap{\grad}^{-1}Bq_1}_2;  
\end{align*} 
this ensures that small-ness when measured with $B$ can balance large-ness when measured with $A$. 

The requisite energy estimate on $g$ is summarized by the following proposition. 
\begin{proposition}[Stability of approximate solution] \label{prop:stabg}
Let $g$ satisfy \eqref{def:g} and assume that $f^E$ is chosen as in Proposition \ref{lem:AppHi}. Then we have the estimate 
\begin{subequations}
\begin{align}
\sup_{t \in (t_{in},t_\star)}\norm{A(t)\left(\jap{v} g\right)}_{L^2} & \lesssim \epsilon^2,  \label{ineq:FinalAg} \\ 
\sup_{t \in (t_{in}, t_\star)}\norm{B(t)\left(\jap{v} g\right)}_{L^2} & \lesssim \epsilon^{\sigma/5}. 
\end{align}
\end{subequations}
\end{proposition} 
\begin{remark} 
In fact, the $\epsilon^2$ in \eqref{ineq:FinalAg} is arbitrary. We can choose $\epsilon^{\alpha}$ for any $\alpha$ fixed (not dependent on $\sigma$).   
This shows that $f^E$ is a very accurate approximate solution of \eqref{def:glideVP} in $H^{\sigma-R}$ for long times. 
\end{remark}

\begin{proposition} \label{prop:pfend}
Propositions \ref{lem:accessID}, \ref{lem:AppHi}, and \ref{prop:stabg} imply Theorem \ref{thm:main}.  
\end{proposition}
\begin{proof} 
For future convenience, define 
\begin{align}
\beta & := \sigma-R  \label{def:sigmabeta}
\end{align}
For $z \geq 0$ arbitrary, 
\begin{align}
\norm{f(t_\star)}_{H^{\beta+z}} 
& \geq \left(\int_{\abs{\eta-\eta_0} < 1}\jap{1,\eta}^{2(\beta+z)} \abs{\widehat{f^E}(t_\star,1,\eta)}^2 d\eta \right)^{1/2} \nonumber \\ & \quad - \left(\int_{\abs{\eta-\eta_0} < 1}\jap{1,\eta}^{2(\beta+z)} \abs{\hat{g}(t_\star,1,\eta)}^2 d\eta\right)^{1/2}. \label{ineq:flowbd}
\end{align}
By Proposition \ref{lem:AppHi} we have (recall \eqref{def:sigmabeta}), 
\begin{align}
\int_{\abs{\eta-\eta_0} < 1}\jap{1,\eta}^{2(\beta+z)} \abs{\hat{f^E}(t_\star,1,\eta)}^2 d\eta & \gtrsim \epsilon^2 \jap{\eta_0}^{2(\beta+z)} \jap{k_0,\eta_0}^{-2\sigma} e^{6(K_m'\epsilon)^{1/3}\eta_0^{1/3}} \gtrsim \jap{\eta_0}^{2z}.\label{ineq:fElowbd}
\end{align}
By the definition of $A$ in \eqref{def:AG}, and $\mu$ in \eqref{def:mu}, there holds:
\begin{align*}
\int_{\abs{\eta-\eta_0} < 1}\jap{1,\eta}^{2(\beta+z)} \abs{\hat{g}(t_\star,1,\eta)}^2 d\eta & \lesssim \jap{\eta_0}^{2z} e^{-2\left(\mu(t_\star)+r\right)(K\epsilon)^{1/3} \eta_0^{1/3}} \norm{Ag}_2^2 \\ 
& \lesssim \epsilon^4 \jap{\eta_0}^{2z} e^{-2\left(\mu_\infty + \frac{3}{2}r\right)(K\epsilon)^{1/3} \eta_0^{1/3}} \\ 
& =\epsilon^{6} \jap{\eta_0}^{2z-2R} e^{6(K_m' \epsilon)^{1/3}\eta_0^{1/3} - 2\left(\mu_\infty + \frac{3}{2}r\right)(K\epsilon)^{1/3} \eta_0^{1/3}}. 
\end{align*}
The constants satisfy $6(K_m')^{1/3} - \left(2\mu_\infty + 3r\right)K^{1/3} \leq 0$, and hence we have
\begin{align*}
\int_{\abs{\eta-\eta_0} < 1}\jap{1,\eta}^{2(\beta+z)} \abs{\hat{g}(t_\star,1,\eta)}^2 d\eta & \lesssim \epsilon^{6} \jap{\eta_0}^{2z-2R}. 
\end{align*}
Putting this inequality together with \eqref{ineq:flowbd}, \eqref{ineq:fElowbd}, \eqref{eq:rhokkt}, and Proposition \ref{lem:accessID} completes the proof of Theorem \ref{thm:main}  (for gravitational interactions). 
\end{proof} 

\section{Linearized Vlasov equations} \label{sec:LinVlas}
In this section we discuss the forced linearized Vlasov problem (written as in \eqref{def:glideVP}), 
\begin{equation}\label{def:LinVPE}
  \left\{
\begin{array}{l} \dss 
\partial_t f + E(t,z+tv)\cdot \grad_v f^0 = \mathcal{S} \\ 
\rho(t,x) = \int_{\Real} f(t,x,v) dv \\  
E(t,x) = -\partial_x W \ast \rho \\ 
h(0,x,v) = h_{in}(x,v). 
\end{array} 
\right.
\end{equation} 
As in \eqref{eq:rhokkt}, $\hat{\rho}(t,k) = \hat{f}(t,k,kt)$. 
Using this, we derive from \eqref{def:VPE} the following Volterra equation, 
\begin{align}
\hat{\rho}(t,k) = H(t,k) - \frac{1}{2\pi}\int_0^t\abs{k}^2 \widehat{W}(k) (t-\tau) \widehat{f^0}(k(t-\tau)) \hat{\rho}(\tau,k) d\tau, \label{eq:Voltlin}
\end{align}
where
\begin{align*}
H(t,k)& = \widehat{h_{in}}(k,kt) + \int_0^t \widehat{\mathcal{S}}(\tau,k,kt) d\tau. 
\end{align*}
Recall, (see \S\ref{apx:Gev} for our Fourier analysis conventions), 
\begin{align}
\widehat{f^0}(\xi) =  2 \pi e^{-\abs{\xi}}. \label{eq:Fourierf0}
\end{align}
For this choice of $f^0$, \eqref{eq:Voltlin} admits a simple, explicit solution; see e.g. Glassey and Schaeffer \cite{Glassey94}. 
\begin{lemma} \label{lem:basicvolt}
Let $f^0(v) = 4\pi \delta (1 + v^2)^{-1}$ and $\widehat{W}(k) = \zeta\abs{k}^{-1-\gamma_0}$ for some $\gamma_0 \geq 1$ and $\zeta \in \set{-1,1}$, 
then for $\delta < 1$, the solution to \eqref{eq:Voltlin} can be written in the following form, for some kernel $R \in L^1(\Real_+)$, 
\begin{align*}
\hat{\rho}(t,k) = H(t,k) + \int_0^t R(t-\tau,k)H(\tau,k) d\tau, 
\end{align*}
which satisfies, 
\begin{align}
\sup_{k \in \Integers_\ast }\int_0^\infty e^{(1-\sqrt{\delta})\abs{kt}}\abs{R(\tau,k)} d\tau \lesssim \sqrt{\delta}.
\end{align}
More specifically: 
\begin{itemize} 
\item if $\zeta = -1$, the solution to \eqref{eq:Voltlin} is given by 
\begin{align*}
\hat{\rho}(t,k) = H(t,k) + \int_0^t \sqrt{\delta} \abs{k}^{\frac{1-\gamma_0}{2}} \sinh(\sqrt{\delta}\abs{k}^{\frac{1-\gamma_0}{2}}(t-s)) e^{-\abs{k}(t-s)} H(s,k) ds; 
\end{align*}
\item if $\zeta = +1$, the solution to \eqref{eq:Voltlin} is given by 
\begin{align*}
\hat{\rho}(t,k) = H(t,k) - \int_0^t \sqrt{\delta} \abs{k}^{\frac{1-\gamma_0}{2}} \sin(\sqrt{\delta}\abs{k}^{\frac{1-\gamma_0}{2}}(t-s)) e^{-\abs{k}(t-s)} H(s,k) ds. 
\end{align*}
\end{itemize}
\end{lemma} 

We will need the linearized dynamics with initial data specified at an arbitrary time, however the convolution structure of \eqref{eq:Voltlin} indicates that the problem is translation invariant in time. 
\begin{corollary} \label{cor:linsolve} 
Let $\hat{\rho}(t,k)$ solve the following for $t_0$ fixed and arbitrary, 
\begin{align}
\hat{\rho}(t,k) = H(t,k) - \frac{1}{2\pi}\int_{t_0}^t\abs{k}^2 \widehat{W}(k) (t-\tau) \widehat{f^0}(k(t-\tau)) \hat{\rho}(\tau,k) d\tau. \label{eq:Voltlin0}
\end{align} 
Then the solution is given by the following
\begin{align*}
\hat{\rho}(t,k) = H(t,k) + \int_{t_0}^t R(t-\tau,k) H(\tau,k) d\tau, 
\end{align*}
where $R$ is given as in Lemma \ref{lem:basicvolt}. 
\end{corollary}

\section{Accessibility of an unstable configuration} \label{sec:Access}
In this section we prove Proposition \ref{lem:accessID}.
Let us only prove part (i); part (ii) is easier and can be proved via a simpler variant so is omitted. 
We solve \eqref{def:fHST} backwards in time from $t = t_{in} = \epsilon^{-q}$ for some fixed $0 < q < \min\left(\frac{1}{4},p\right)$ back to $t = 0$.
Let $T_0$ be the smallest time such that on $[T_0,t_{in}]$, the following estimates hold for $t \in [T_0, t_{in}]$: 
\begin{subequations} \label{boot:init}
\begin{align}
 \norm{\jap{v}f(t)}_{H^\sigma} & \leq 4\eps(1 +  4t_{in} \abs{t_{in} - t}^{1/2}) \label{boot:init:hi} \\ 
\norm{\brak{\grad_x, t\grad_x}^{\sigma-1} \rho}_{L^2_t(T_0,t_{in};L^2_x)} & \leq 4\eps \label{boot:init:mid} \\
\norm{\jap{v}f(t)}_{H^{\sigma-3}} &\leq 4\eps. \label{boot:init:lo}
\end{align} 
\end{subequations}
Notice that because the evolution is going backwards in time, the norm on $\rho$ is getting weaker rather than stronger as usual in \cite{BMM13,MouhotVillani11}. 
By well-posedness of the Vlasov equations, we have $T_0 < t_{in}$, and moreover, the norms on the left-hand side of \eqref{boot:init} takes values continuously in time.  
 Proposition \ref{lem:accessID} follows from the following. 
\begin{lemma} \label{lem:bootinit} 
Under the hypotheses of Theorem \ref{thm:main}, for $\eps$ sufficiently small, the inequalities in \eqref{boot:init} hold with `4' replaced with `2'. 
\end{lemma}
\begin{proof}
Define $I_0= [T_0,t_{in}]$. 
By time-reversibility, one derives from Lemma \ref{lem:basicvolt} for $t \leq t_{in}$,
\begin{align*}
\hat{\rho}(t,k) = \widehat{\rho_0}(t,k) + \int_{t}^{t_{in}}R(s-t) \widehat{\rho_0}(s,k) ds,
\end{align*}
where,
\begin{align*}
\widehat{\rho_0}(t,k) = \widehat{f}_{in}(k,kt) -\sum_{\ell \in \Integers} \int_t^{t_{in}} \hat{\rho}(\tau,\ell) \widehat{W}(\ell) \ell \cdot k(\tau-t) \widehat{f}(\tau,k-\ell, \ell \tau - kt) d\tau.
\end{align*}
By Lemma \ref{lem:basicvolt} it follows that (for some constant depending only on $\sigma$),
\begin{align*} 
\norm{\brak{\grad_x,t\grad_x }^{\sigma-1} \rho}_{L^2_t(I_0;L^2_x)} & \leq (1+C(\sigma)\sqrt{\delta})\norm{\brak{\grad_x, t\grad_x}^{\sigma-1} \rho_0}_{L^2_t(I_0;L^2_x)}. 
\end{align*} 
By an easier variation of the methods in \cite{BMM13}, we have, 
\begin{align*} 
\norm{\brak{\grad_x, t\grad_x}^{\sigma-1} \rho_0}_{L^2_t(I_0;L^2_x)} & \leq \eps + C(\sigma)\norm{\brak{\grad_x, t\grad_x}^{\sigma-1} \rho}_{L^2_t(I_0; L^2_x)} \norm{\brak{v} f}_{L^\infty_t(I_0;H^{\sigma-3})} \\
& \quad + C(\sigma)\norm{\brak{v} f}_{L^\infty_t(I_0;H^{\sigma-3})} \norm{\brak{v} f}_{L^\infty_t (I_0;H^{\sigma})}. 
\end{align*}
By the bootstrap hypotheses \eqref{boot:init} and the definition of $t_{in}$, there holds
\begin{align*}
\norm{\brak{\grad_x, t\grad_x}^{\sigma-1} \rho_0}_{L^2_t(I_0;L^2_x)} & \leq \eps + C(\sigma) 16 \eps^2 + 128 C(\sigma) \eps^{2-\frac{3q}{2}}
\end{align*}
As $q < 1/4$, for $\eps$ small, we improve \eqref{boot:init:mid}.  
Next, consider \eqref{boot:init:hi}. 
Let $\alpha \in \set{0,1}$. Computing from \eqref{def:fHST}, we have (note we have used $\gamma_0 \geq 1$) 
\begin{align*}
\frac{1}{2}\frac{d}{dt}\norm{\jap{\grad}^\sigma (v^\alpha f)}_{2}^2 & = -\jap{\jap{\grad}^\sigma (v^\alpha f), \jap{\grad}^\sigma v^\alpha \left(E(z+tv) \partial_v f^0\right)}_2 \\ & \quad 
-\jap{\jap{\grad}^\sigma (v^\alpha f), \jap{\grad}^\sigma v^\alpha \left(E(z+tv)(\partial_v - t\partial_z) f \right)}_2 \\ 
& = L + NL. 
\end{align*}
By easier variants of methods in \cite{BMM13}, we have
\begin{align*} 
\abs{L} & \lesssim_\sigma \delta \brak{t} \norm{v^\alpha f(t)}_{H^\sigma}\norm{\brak{\grad_x,t\grad_x}^{\sigma-1}\rho}_{L^2}\\ 
\abs{NL} & \lesssim_\sigma    \brak{t} \norm{\brak{v} f(t)}_{H^\sigma}^2\norm{v^\alpha f(t)}_{H^\sigma}^2 + \brak{t}\norm{\brak{v}f(t)}_{H^{\sigma-3}}\norm{\brak{v} f(t)}_{H^\sigma} \norm{v^\alpha f(t)}_{H^\sigma}.
\end{align*}
Hence, by \eqref{boot:init} and Cauchy-Schwarz in time, it follows that (with different $C$'s each line)
\begin{align*}
\norm{\jap{\grad}^\sigma \brak{v} f(t)}_{2} & \leq \eps + C\int_{t}^{t_{in}} \delta \brak{\tau} \norm{\brak{\grad_x,\tau\grad_x}^{\sigma-1}\rho(\tau)}_{L^2} d\tau + C\int_t^{t_{in}} \eps^2 \brak{s} \left(1 + t_{in}\abs{s - t_{in}}^{1/2}\right)^2 ds \\
& \leq \eps + C \delta \eps t_{in} \abs{t-t_{in}}^{1/2} + C\eps^2 t_{in}^4 \abs{t-t_{in}}. 
\end{align*}
The improvement to \eqref{boot:init:hi} follows by choosing $\delta$ and $\eps t_{in}^4 = \eps^{1-4q}$ sufficiently small.  
The improvement to \eqref{boot:init:lo} is a straightforward variant and is hence omitted for brevity. 
\end{proof} 

\section{High frequency approximate solution} \label{sec:fH}
On the Fourier side, the high frequency initial condition \eqref{def:fHin} is given by
\begin{align*}
\widehat{f^H_{in}}(k_0,\eta) & = \epsilon\jap{k_0,\eta_0}^{-\sigma} \left(e^{-\frac{1}{2}\abs{\eta - \eta_0}} + e^{-\frac{1}{2}\abs{\eta + \eta_0}}\right) \\ 
\widehat{f^H_{in}}(-k_0,\eta) & = \epsilon\jap{k_0,\eta_0}^{-\sigma} \left(e^{-\frac{1}{2}\abs{\eta - \eta_0}} + e^{-\frac{1}{2}\abs{\eta + \eta_0}}\right) \\ 
\widehat{f^H_{in}}(k\neq k_0,\eta) & = 0. 
\end{align*} 
We will take $f^H$, the approximate high frequency solution, to solve \eqref{eq:seciter}. 
As \eqref{eq:seciter} is linear, it is convenient to sub-divide $f^H$ into four separate components based on the initial data and solve for them separately: 
\begin{subequations} \label{def:fHpm}
\begin{align}
f^H_{++}(t_{in}) & = \epsilon\jap{k_0,\eta_0}^{-\sigma} e^{-\frac{1}{2}\abs{\eta - \eta_0}} \delta_{k = k_0} \\ 
f^H_{+-}(t_{in}) & = \epsilon\jap{k_0,\eta_0}^{-\sigma} e^{-\frac{1}{2}\abs{\eta - \eta_0}} \delta_{k = -k_0} \\ 
f^H_{-+}(t_{in}) & = \epsilon\jap{k_0,\eta_0}^{-\sigma} e^{-\frac{1}{2}\abs{\eta + \eta_0}} \delta_{k = k_0} \\ 
f^H_{--}(t_{in}) & = \epsilon\jap{k_0,\eta_0}^{-\sigma} e^{-\frac{1}{2}\abs{\eta + \eta_0}} \delta_{k = -k_0}. 
\end{align}
\end{subequations} 
We analogously define $\rho^{H}_{++}$, $\rho^{H}_{+-}$, $\rho^{H}_{-+}$ and $\rho^{H}_{--}$ to be the densities associated with each of the four corresponding solutions to \eqref{eq:seciter}. 
Moreover, the amplitude of $f^H$ is irrelevant. Hence, for future convenience we define the following to suppress it: 
\begin{align}
\epsilon' = \epsilon \jap{k_0,\eta_0}^{-\sigma}. \label{def:epsprime}
\end{align}
We define the \emph{critical intervals}, corresponding to Orr's critical times,  
\begin{subequations} \label{def:ICk}
\begin{align}
I_{k,\eta} & = \left[\frac{\eta}{k} - \frac{\abs{\eta}}{2\abs{k}(\abs{k}+1)}, \frac{\eta}{k} + \frac{\abs{\eta}}{2\abs{k}(\abs{k}-1)} \right] := [t_{k,\eta}, t_{k-1,\eta}], \quad 2 \leq \abs{k}, \\ 
I_{1,\eta} & = \left[\frac{3\abs{\eta}}{4}, 2\abs{\eta}\right] := [t_{1,\eta},t_{0,\eta}].  
\end{align}
\end{subequations} 
For notational convenience we use the shorthand $t_{k} := t_{k,\eta_0}$.  
We record the following lemma regarding various growth factors that will arise below. 
\begin{lemma} \label{lem:GrowthFact}
Let $\tilde{K} > 0$ be arbitrary and let $\epsilon \eta$ be sufficiently large relative to $\tilde{K}$. 
Fix  
\begin{align}
\tilde{N} = \textup{Floor}((\tilde{K}\epsilon)^{1/3} \eta^{1/3}), \label{def:tildek0}
\end{align}
and define the following growth factor for $k \geq 1$,  
\begin{subequations} \label{def:Yketa0}
\begin{align}
Y_k(\eta) & := 1 & \quad\quad k \geq \tilde{N} \\ 
Y_k(\eta) & :=   \left(\frac{\eps \tilde{K} \eta}{(k+1)^3}\right)\cdots \left(\frac{\eps \tilde{K} \eta}{(\tilde{N}-1)^3} \right)\left(\frac{\eps \tilde{K} \eta}{\tilde{N}^3} \right)  & \quad\quad 1 \leq k < \tilde{N}. 
\end{align}
\end{subequations}
Then there holds the following, with implicit constant independent of $\epsilon, \eta_0$, and $\tilde{K}$,  
\begin{align*}
Y_k(\eta) \leq Y_1(\eta) \approx \frac{1}{(\tilde{K}\eps \eta)^{1/2}}e^{3(\tilde{K} \epsilon)^{1/3}\eta^{1/3}}. 
\end{align*}
\end{lemma}  
\begin{proof}
The proof is essentially from [Lemma 3.1, \cite{BM13}]. By definition, 
\begin{align*}
Y_1(\eta) & = \left(   \frac{\tilde{K}\epsilon\eta}{(\tilde{N})^3}  \right)  \left(    \frac{\tilde{K} \epsilon \eta}{(\tilde{N}-1)^3}  \right) ... \left(    \frac{\tilde{K} \epsilon \eta}{1^3}  \right) = \frac{(\tilde{K}\epsilon\eta)^{\tilde{N}}} { (\tilde{N}!)^3 }.    
\end{align*}
Using Stirling's formula, $N !   \sim \sqrt{2 \pi  N}   (N/e)^N$, we have 
\begin{align*}
Y_1(\eta) \approx \frac{(\tilde{K}\epsilon\eta)^{\tilde{N}}} { ( 2\pi \tilde{N} )^{3/2} (\tilde{N}/e)^{3\tilde{N}}} \approx  
   \frac{1} { (\tilde{K}\epsilon \eta)^{1/2}   } e^{3(\tilde{K}\epsilon \eta)^{1/3}}  
     \left[ e^{3\tilde{N} - 3(\tilde{K}\epsilon \eta)^{1/3}} \Big( \frac{ K\epsilon\eta}{\tilde{N}^3}  \Big)^{\tilde{N}+\frac{1}{2}}   \right] 
\end{align*}
 and hence the result follows, since the term between $[..]$ is $\approx 1$ by $\abs{\tilde{N}^3 - \tilde{K}\epsilon\eta} \leq 1$ and $\epsilon \eta$ large. 
\end{proof}

\subsection{Upper bounds}
In this section we deduce upper bounds on $f^H$ and $vf^H$ pointwise in frequency. 
The first observation is that, analogous to the linearized Vlasov equations,  \eqref{eq:seciter} can be reformulated as a closed system of Volterra equations only involving the density $\rho$: 
\begin{align}
\widehat{\rho^H}(t,k) & = f^H_{in}(t_{in},k,kt) -  \frac{1}{2\pi}\int_{t_{in}}^t \abs{k}^2 \widehat{W}(k)(t-\tau)\widehat{f^0}(k(t-\tau))\widehat{\rho^H}(\tau,k) d\tau \nonumber \\ & \quad - \epsilon\sum_{\ell = k \pm 1} \int_{t_{in}}^t \widehat{\rho^H}(\tau,\ell) \ell \widehat{W}(\ell) k(t-\tau) e^{-\abs{kt - \ell \tau}} d\tau. \label{eq:rhoseciter}
\end{align}
By a standard contraction mapping principle, it is straightforward to prove that there exists a unique (global) solution to \eqref{eq:rhoseciter}. 
Define the following growth factor, for some constant $K_m \geq K_m'$ to be specified later, analogously to the definition used in Lemma \ref{lem:GrowthFact}: 
\begin{subequations} \label{def:Cketa0}
\begin{align}
N_m & := \textup{Floor}\left((K_m\epsilon)^{1/3} \eta_0^{1/3}\right), \\ 
C_k(\eta_0) & := 1 & \quad\quad k \geq N_m,  \\ 
C_k(\eta_0) & :=   \left(\frac{\eps K_m \eta_0}{(k+1)^3}\right)\cdots \left(\frac{\eps K_m \eta_0}{(N_m-1)^3} \right)\left(\frac{\eps K_m \eta_0}{N_m^3} \right)  & \quad\quad 1 \leq k < N_m.
\end{align}
\end{subequations}
Note that since $K_m \geq K_m'$, $N_m \geq k_0$. 
Fix $r(t)$ as follows, for $b \in (0,\frac{1}{6})$ chosen below: 
\begin{align}
r(t) & = \frac{1}{4} + \frac{1}{4} \left(\frac{t_{in}}{t}\right)^b.  \label{eq:rtrhoup}
\end{align}
The following lemma provides the desired upper bounds on the density. 
Note that \eqref{ineq:rhoHupps} localizes the density very close to the critical times.  
\begin{lemma} \label{lem:AprxRhoUp}
The solutions to \eqref{eq:rhoseciter} with initial distributions given by \eqref{def:fHpm} satisfy the following for $\delta = \eps^p$ and $\eps$ chosen sufficiently small (recall the definition \eqref{def:epsprime}): for all $k \geq 1$, there holds for all $\alpha \geq 0$,    
\begin{subequations} \label{ineq:rhoHupps}
\begin{align}
\abs{\widehat{\rho^H_{++}}(t,k)} & \leq 4\eps' \left(C_k(\eta_0)\mathbf{1}_{k \geq 1} + \eps^\alpha \mathbf{1}_{k \leq -1}\right) e^{-r(t)\abs{\eta_0-kt}}e^{-k_0^{-1} \abs{k}} \label{ineq:rhoH++} \\ 
\abs{\widehat{\rho^H_{--}}(t,k)} & \leq 4\eps' \left(C_{\abs{k}}(\eta_0)\mathbf{1}_{k \leq -1} + \eps^{\alpha} \mathbf{1}_{k \geq 1}\right) e^{-r(t)\abs{\eta_0 + kt}}e^{-k_0^{-1} \abs{k}} \label{ineq:rhoH--} \\ 
\abs{\widehat{\rho^H_{+-}}(t,k)} & \lesssim \eps' \eps^{\alpha} e^{-r(t)\abs{\eta_0-kt}}e^{-k_0^{-1} \abs{k}} \label{ineq:rhoH+-} \\  
\abs{\widehat{\rho^H_{-+}}(t,k)} & \lesssim \eps' \eps^{\alpha}  e^{-r(t)\abs{\eta_0+kt}}e^{-k_0^{-1} \abs{k}}. \label{ineq:rhoH-+}
\end{align}
\end{subequations}
\end{lemma} 
\begin{proof}
We will consider simply the proof for \eqref{ineq:rhoH++}; \eqref{ineq:rhoH--} follows by symmetry whereas \eqref{ineq:rhoH+-} and \eqref{ineq:rhoH-+} are simpler variants (as there are not resonances for positive times). Similarly, we will only consider $k \geq 1$, as for $\rho_{++}^H$, the $k \leq -1$ modes are essentially treated the same as $\rho^{H}_{+-}$.  

For notational simplicity, for the duration of the proof of Lemma \ref{lem:AprxRhoUp}, denote 
\begin{align*}
\rho := \rho^H_{++}. 
\end{align*}
Let $T$ be the largest time such that the following holds for all $t \in [t_{in},T]$ 
\begin{align}
\abs{\widehat{\rho}(t,k)} & \leq 8\eps' C_k(\eta_0) e^{-r(t)\abs{\eta_0-kt}} e^{-k_0^{-1}\abs{k}}.  \label{ineq:AprxRhoBoot}
\end{align}
We prove that on $[t_{in},T]$, \eqref{ineq:AprxRhoBoot} holds with `8' replaced with `4' for $\eps$ chosen sufficiently small (by continuity this is sufficient and the assumptions on the initial distribution imply $T > t_{in}$).

Using \eqref{eq:rhoseciter}, Lemma \ref{lem:basicvolt}, 
\begin{align}
\abs{\hat{\rho}(t,k)}\frac{e^{r(t)\abs{\eta_0-kt} +k_0^{-1}\abs{k}}}{C_k(\eta_0)}  & \leq \frac{e^{r(t)\abs{\eta_0-kt}+k_0^{-1}\abs{k}}}{C_k(\eta_0)} \abs{\widehat{f^H_{++}}(t_{in},k,kt)} \nonumber \\   
& \quad  \hspace{-2.5cm}+ \epsilon\sum_{\ell = k \pm 1} \int_{t_{in}}^t \frac{e^{r(t)\abs{\eta_0-kt}+k_0^{-1}\abs{k}}}{C_k(\eta_0)} \abs{\hat{\rho}(\tau,\ell)} \abs{\ell \widehat{W}(\ell) k(t-\tau)} e^{-\abs{kt - \ell \tau}} d\tau \nonumber \\ 
& \quad \hspace{-2.5cm} + \sqrt{\delta} \int_{t_{in}}^t e^{-(\frac{1}{4}-\sqrt{\delta})\abs{k(t-\tau)}} \frac{e^{r(\tau)\abs{\eta_0-k\tau}+k_0^{-1}\abs{k}}}{C_k(\eta_0)} \abs{\widehat{f^H_{in}}(t_{in},k,k\tau)} d\tau \nonumber \\  
& \quad \hspace{-2.5cm} + \eps \sqrt{\delta} \sum_{\ell = k \pm 1} \int_{t_{in}}^t e^{-(\frac{1}{4}-\sqrt{\delta})\abs{k(t-\tau)}} \int_{t_{in}}^\tau \frac{e^{r(\tau)\abs{\eta_0-k\tau}+k_0^{-1}\abs{k}}}{C_k(\eta_0)}\abs{ \hat{\rho}(s,\ell) \ell \widehat{W}(\ell) k(\tau-s)} e^{-\abs{k\tau - \ell s}} ds d\tau \nonumber \\  
& = \sum_{i=1}^4 R_i. \label{def:RiAprxRho} 
\end{align}
From the definition of $f^H_{++}$ in \eqref{def:fHpm}, we have, 
\begin{align*}
R_1 \leq \eps' e^{\left(r(t)-\frac{1}{2}\right)\abs{\eta_0 - kt}} \leq \eps'. 
\end{align*} 
For the term $R_2$ we have by the bootstrap hypothesis \eqref{ineq:AprxRhoBoot}, 
\begin{align*}
R_2 & \lesssim 8\epsilon \epsilon' \sum_{\ell = k \pm 1} \int_{t_{in}}^t \frac{C_{\ell}(\eta_0)}{C_k(\eta_0)} e^{r(t)\abs{\eta_0-kt} - r(\tau)\abs{\eta_0-\ell \tau}} \abs{\ell \widehat{W}(\ell) k(t-\tau)} e^{-\abs{kt - \ell \tau}} d\tau \\ 
& \lesssim \epsilon \epsilon' \sum_{\ell = k \pm 1} \int_{t_{in}}^t \frac{C_{\ell}(\eta_0)}{C_k(\eta_0)} e^{\left(r(t) - r(\tau)\right)\abs{\eta_0-\ell \tau}} \frac{\jap{\tau}}{\abs{\ell}} e^{-\frac{1}{2}\abs{kt - \ell \tau}} d\tau.
\end{align*}
To treat this integral, we divide into resonant and non-resonant regions: 
\begin{align*}
R_2 & \lesssim \epsilon  \epsilon' \sum_{\ell = k \pm 1} \int_{t_{in}}^t \left(\mathbf{1}_{\abs{kt-\ell \tau} < \frac{t}{2}} + \mathbf{1}_{\abs{kt-\ell \tau} \geq \frac{t}{2}}\right) \frac{C_{\ell}(\eta_0)}{C_k(\eta_0)} e^{\left(r(t) - r(\tau)\right)\abs{\eta_0-\ell \tau}} \frac{\jap{\tau}}{\abs{\ell}} e^{-\frac{1}{2}\abs{kt - \ell \tau}} d\tau \\ 
& = R_{2;R} + R_{2;NR}. 
\end{align*}
The non-resonant region is straightforward, despite the potential loss from the ratio of $C_k$ and $C_\ell$. Indeed, using $t > \eps^{-q}$ for some $q \in (0,1)$ and the definition of $\eta_0$ (see \eqref{def:eta0}), 
\begin{align*}
R_{2;NR} &  \lesssim \epsilon \epsilon' \sum_{\ell = k \pm 1} \int_{t_{in}}^t \mathbf{1}_{\abs{kt-\ell \tau} \geq \frac{t}{2}} \frac{\tau}{\abs{k}} \jap{\frac{\eps \eta_0}{\abs{k}^3}} e^{-\frac{1}{8}t} e^{-\frac{1}{4}\abs{kt - \ell \tau}} d\tau \lesssim_q \epsilon^2 \epsilon',
\end{align*}
which suffices to prove Lemma \ref{lem:AprxRhoUp} provided $\epsilon$ is chosen sufficiently small. 
Turn next to the resonant region.
First, observe that over the resonant region, necessarily $\ell = k+1$. Therefore, 
\begin{align*}
R_{2;R} & \lesssim \epsilon \epsilon' \int_{t_{in}}^t \mathbf{1}_{\abs{kt-\ell \tau} < \frac{t}{2}} \frac{C_{k+1}(\eta_0)}{C_k(\eta_0)} e^{\left(r(t) - r(\tau)\right)\abs{\eta_0-(k+1)\tau}} \frac{\tau}{\abs{k+1}} e^{-\frac{1}{4}\abs{kt - (k+1)\tau}} d\tau \\ 
& = \epsilon \epsilon' \int_{t_{in}}^t \mathbf{1}_{\abs{kt-\ell \tau} < \frac{t}{2}}\left(\mathbf{1}_{\abs{\eta_0 - (k+1)\tau} < \frac{(k+1)\tau}{2}} + \mathbf{1}_{\abs{\eta_0 - (k+1)\tau} \geq \frac{(k+1)\tau}{2}} \right) \\ & \quad\quad \times \frac{C_{k+1}(\eta_0)}{C_k(\eta_0)} e^{\left(r(t) - r(\tau)\right)\abs{\eta_0-(k+1)\tau}} \frac{(k+1)\tau}{\abs{k+1}^2} e^{-\frac{1}{4}\abs{kt - (k+1)\tau}} d\tau \\ 
& = R_{2;R,0} + R_{2;R,1}. 
\end{align*}
For the first term we use that $\eta_0 \approx (k+1)\tau$ on the support of the integrand, hence by the definition of $C_k$ \eqref{def:Cketa0} and that $N_m \geq k_0$, there holds (on the support of the integrand), 
\begin{align*}
\frac{C_{k+1}(\eta_0)}{C_k(\eta_0)} \frac{(k+1)\tau}{(k+1)^3} \approx \frac{C_{k+1}(\eta_0)}{C_k(\eta_0)} \frac{\eta_0}{(k+1)^3}  \lesssim \frac{1}{\eps K_m}, 
\end{align*}
which implies 
\begin{align*}
R_{2;R,0} & \lesssim \epsilon' \int_{t_{in}}^t \mathbf{1}_{\abs{kt-\ell \tau} < \frac{t}{2}} \mathbf{1}_{\abs{\eta_0 - (k+1)\tau} < \frac{(k+1)\tau}{2}} \frac{(k+1)}{K_m} e^{-\frac{1}{4}\abs{kt - (k+1)\tau}} d\tau  \lesssim \frac{\eps'}{K_m}, 
\end{align*}
which is sufficient for the proof of Lemma \ref{lem:AprxRhoUp} for $K_m$ chosen sufficiently large (depending only on universal constants). 
For $R_{2;R,1}$, note that on the support of the integrand there holds, (resonance implies $\tau \approx \frac{k}{k+1}t$ on the support), 
\begin{align}
r(t) - r(\tau) 
& = \frac{t_{in}^b}{4}\left(\frac{\tau^b - t^b}{t^b \tau^b}\right)  \lesssim - t_{in}^b\left(\frac{t-\tau}{t^b \tau}\right)  \approx -\frac{t_{in}^b}{t^{b}(k+1)}, \label{ineq:rtrtau}
\end{align}
which implies, using also that $\abs{\eta_0 - (k+1)\tau} \gtrsim (k+1)\tau$, we have, 
\begin{align*}
R_{2;R,1} & \lesssim \epsilon^2 \jap{k_0,\eta_0}^{-\sigma} \int_{t_{in}}^t \mathbf{1}_{\abs{kt-\ell \tau} < \frac{t}{2}} \mathbf{1}_{\abs{\eta_0 - (k+1)\tau} \geq \frac{(k+1)\tau}{2}} \frac{C_{k+1}(\eta_0) \tau}{C_k(\eta_0)(k+1)} \left(\frac{1}{t_{in}^b\tau^{1-b}}\right)^{\frac{1}{1-b}} e^{-\frac{1}{4}\abs{kt - (k+1)\tau}} d\tau. 
\end{align*}
If $k \leq N_m$, then, there holds, by definition of $C_k$ \eqref{def:Cketa0}, $K_m \eps\eta_0 \approx N_m^3$, 
\begin{align*}
R_{2;R,1}\mathbf{1}_{k \leq k_0} & \lesssim_\alpha \epsilon' \int_{t_{in}}^t \mathbf{1}_{\abs{kt-\ell \tau} < \frac{t}{2}} \mathbf{1}_{\abs{\eta_0 - (k+1)\tau} \geq \frac{(k+1)\tau}{2}} \frac{\tau (k+1)^2 }{K_m \eta_0} \left(\frac{1}{t_{in}^b\tau^{1-b}}\right)^{\frac{1}{1-b}} e^{-\frac{1}{4}\abs{kt - (k+1)\tau}} d\tau \\ 
& \lesssim \epsilon' \int_{t_{in}}^t \mathbf{1}_{\abs{kt-\ell \tau} < \frac{t}{2}} \mathbf{1}_{\abs{\eta_0 - (k+1)\tau} \geq \frac{(k+1)\tau}{2}} \frac{\tau \eps^{1/3}(k+1)}{K_m^{2/3} \eta_0^{2/3}} \left(\frac{1}{t_{in}^b\tau^{1-b}}\right)^{\frac{1}{1-b}} e^{-\frac{1}{4}\abs{kt - (k+1)\tau}} d\tau \\ 
& \lesssim \epsilon \epsilon',  
\end{align*}
which is sufficient for Lemma \ref{lem:AprxRhoUp} by choosing $\epsilon$ sufficiently small. 
If, $k > N_m$, then there holds, 
\begin{align*}
R_{2;R,1}\mathbf{1}_{k \leq k_0} & \lesssim \epsilon \epsilon' \int_{t_{in}}^t \mathbf{1}_{\abs{kt-\ell \tau} < \frac{t}{2}} \mathbf{1}_{\abs{\eta_0 - (k+1)\tau} \geq \frac{(k+1)\tau}{2}} \frac{\tau}{k+1} \left(\frac{1}{t_{in}^b\tau^{1-b}}\right)^{\frac{1}{1-b}} e^{-\frac{1}{4}\abs{kt - (k+1)\tau}} d\tau \\ 
& \lesssim \epsilon \epsilon',  
\end{align*}
which is sufficient for Lemma \ref{lem:AprxRhoUp} by choosing $\epsilon$ sufficiently small. 
This completes the treatment of $R_2$ from \eqref{def:RiAprxRho}. 
The treatments of $R_3$ and $R_4$ are similar to $R_1$ and $R_2$ by applying the same arguments with $t\mapsto \tau$ and $\tau \mapsto s$; note that $\delta$ is a small parameter. The details are omitted. 
\end{proof} 

Given the bounds on $\rho^H$ provided by Lemma \ref{lem:AprxRhoUp}, it is straightforward to derive the following. 
\begin{lemma} \label{lem:distup} 
Consider the solution to \eqref{eq:seciter} $f^H_{++}$ with initial data $f^H(t_{in}) = f^H_{++}(t_{in})$ (see \eqref{def:fHpm}). 
Then the solution $f_{++}^H$ satisfies the following, 
\begin{align*}
\abs{\widehat{f_{++}^H}(t,k,\eta)} \lesssim \epsilon' C_k(\eta_0) e^{-\frac{1}{8}\abs{\eta-\eta_0}} e^{-k_0^{-1}\abs{k}}. 
\end{align*}
Analogous estimates hold for $f_{--}^H$, $f_{+-}^H$, and $f^H_{-+}$ from Lemma \ref{lem:AprxRhoUp}. 
\end{lemma} 
\begin{proof} 
Integrating \eqref{eq:seciter}, we have
\begin{align*}
\widehat{f^{H}_{++}}(t,k,\eta) & = \widehat{f^H_{++}}(t_{in},k,\eta) - \frac{1}{2\pi}\int_{t_{in}}^t \hat{\rho}(\tau,k) \widehat{W}(k)k (\eta-k\tau) \widehat{f^0}(\eta-k\tau) d\tau \\ 
& \quad - \epsilon\sum_{\ell = k\pm 1} \int_{t_{in}}^t \hat{\rho}(\tau,\ell) \ell \widehat{W}(\ell)(\eta-k\tau) e^{-\abs{\eta-\ell \tau}} d\tau. 
\end{align*}
From here, the result follows by integration using the estimates in Lemma \ref{lem:AprxRhoUp} and the definition of the growth factor in \eqref{def:Cketa0}. We omit the details for brevity.  
\end{proof} 
The following variant of the upper bound quantifies the fact that before the critical time $\frac{\eta_0}{k+1}$, the $k$-th spatial mode of the distribution function must be small. 
\begin{lemma} \label{lem:distup_preres} 
Consider the solution to \eqref{eq:seciter} $f^H_{++}$ with initial data $f^H(t_{in}) = f^H_{++}(t_{in})$ (see \eqref{def:fHpm}). 
Then the solution $f_{++}^H$ satisfies the following for all $\alpha \geq 1$, all $\eps$ chosen sufficiently small (depending on $\alpha$), and all $k < k_0$,
\begin{align}
\abs{\widehat{f_{++}^H}(t_{k+1},k,\eta)} \leq \eps^\alpha \epsilon' e^{-\frac{1}{8}\abs{\eta-\eta_0}}. \label{ineq:distup_preres}
\end{align}
Analogous estimates hold for $f_{--}^H$, $f_{+-}^H$, and $f^H_{-+}$. 
\end{lemma} 
\begin{proof} 
Set $1 \leq k < k_0$. 
Integrating \eqref{eq:seciter} implies
\begin{align*}
e^{\frac{1}{8}\abs{\eta-\eta_0}} \widehat{f^{H}_{++}}(t_{k+1},k,\eta) & = e^{\frac{1}{8}\abs{\eta-\eta_0}}\widehat{f^H_{++}}(t_{in},k,\eta) - \frac{1}{2\pi}\int_{t_{in}}^{t_{k+1}} e^{\frac{1}{8}\abs{\eta-\eta_0}} \hat{\rho}(\tau,k) \widehat{W}(k)k (\eta-k\tau) \widehat{f^0}(\eta-k\tau) d\tau \\ 
& \quad - \sum_{\ell = k\pm 1} \int_{t_{in}}^{t_{k+1}} e^{\frac{1}{8}\abs{\eta-\eta_0}} \hat{\rho}(\tau,\ell) \ell \widehat{W}(\ell)(\eta-k\tau) e^{-\abs{\eta-\ell \tau}} d\tau \\ 
& = e^{\frac{1}{8}\abs{\eta-\eta_0}} \widehat{f^H_{++}}(t_{in},k,\eta) + L + NL. 
\end{align*}
The first term is in fact zero as $k \neq k_0$.  
For $L$, we apply Lemma \ref{lem:AprxRhoUp} and that since $\abs{\eta_0-k\tau} \gtrsim \tau$ on the support of the integral, we have the following for some universal constant $c > 0$,  
\begin{align*}
\abs{L} &  \lesssim \epsilon' \int_{t_{in}}^{t_{k+1}} C_k(\eta_0)  e^{\frac{1}{8}\abs{\eta-\eta_0}}  e^{-\frac{1}{4}\abs{\eta_0-k\tau}} e^{-\abs{\eta-k\tau}}  d\tau
  \lesssim \epsilon' \int_{t_{in}}^{t_{k+1}} C_k(\eta_0)  e^{-c\tau} e^{-\frac{7}{8}\abs{\eta-k\tau}}  d\tau. 
\end{align*}
Even though $C_k(\eta_0) \lesssim e^{(K_m \epsilon)^{1/3}\eta_0^{1/3}}$ by Lemma \ref{lem:GrowthFact}, the exponential decay still dominates  (note $\tau \geq t_{in} = \epsilon^{-q}$) 
and we deduce that for any $\alpha$ we can derive 
\begin{align*}
\abs{L} \lesssim \epsilon' \epsilon^{\alpha+1}, 
\end{align*}
which is consistent with \eqref{ineq:distup_preres} by choosing $\epsilon$ sufficiently small. 
The nonlinear term $NL$ is treated similarly and is hence omitted for brevity. 
\end{proof} 

We will also need upper bounds on the first moment of $f$. 
\begin{lemma} \label{lem:fHmoment}
Consider the solution to \eqref{eq:rhoseciter} with initial data $f^H(t_{in}) = f^H_{++}(t_{in})$ (see \eqref{def:fHpm}). 
The solution $f_{++}^H$ satisfies the following for $\eps$ chosen sufficiently small
\begin{align*}
\abs{\partial_\eta \widehat{f_{++}^H}(t,k,\eta)} \leq 4 \eps' C_k(\eta_0) e^{-\frac{1}{8}\abs{\eta-\eta_0}}. 
\end{align*}
Analogous estimates hold for $f_{--}^H$, $f_{+-}^H$, and $f^H_{-+}$. 
\end{lemma}
\begin{proof} 
Integrate \eqref{eq:seciter} and then differentiate, yielding: 
\begin{align*}
\partial_\eta \widehat{f^{H}_{++}}(t,k,\eta) & = \partial_\eta \widehat{f^H_{++}}(t_{in},k,\eta) - \frac{1}{2\pi}\int_{t_{in}}^t \hat{\rho}(\tau,k) \widehat{W}(k)k \partial_\eta\left( (\eta-k\tau) \widehat{f^0}(\eta-k\tau) \right) d\tau \\ 
& \quad - \sum_{\ell = k\pm 1} \int_{t_{in}}^t \hat{\rho}(\tau,\ell) \ell \widehat{W}(\ell) \partial_\eta \left( (\eta-k\tau) e^{-\abs{\eta-\ell \tau}} \right) d\tau. 
\end{align*}
The result now follows in a manner analogous to Lemma \ref{lem:distup}. The details are omitted for brevity. 
\end{proof} 

\subsection{Lower bounds}\label{sec:lowbd} 
In this section we focus on deriving growth of solutions to the system \eqref{eq:seciter}. 
After looking carefully at \eqref{eq:seciter}, it becomes clear that only a specific subset of interactions are relevant. 
If one only retains ``resonant'' interactions, we are left with the sub-system: 
\begin{subequations} \label{def:resonantsub}
\begin{align}
\partial_t \widehat{f}(t,k,\eta) & = -\delta\widehat{f}(t,k,kt) \widehat{W}(k) k(\eta-kt) e^{-\abs{\eta-kt}}, \quad  t \in I_{k,\eta}, \\ 
\partial_t \widehat{f}(t,k-1,\eta) & = -\epsilon\widehat{f}(t,k,k t) \widehat{W}(k) k \left(\eta-(k-1)t\right)e^{-\abs{\eta-tk}}, \quad  t \in I_{k,\eta}, \\ 
\partial_t \widehat{f}(t,k+1,\eta) & = -\epsilon\widehat{f}(t,k,k t) \widehat{W}(k) k \left(\eta-(k+1)t\right)e^{-\abs{\eta-tk}}, \quad  t \in I_{k,\eta}. \label{eq:resirrelev}  
\end{align}
\end{subequations}
We are only interested in the cascade $k \mapsto k-1 \mapsto k-2 \mapsto \cdots \mapsto 1$ so that the information propagates from one critical time to the next. 
Hence, removing this \eqref{eq:resirrelev}, which only involves modes which have passed the associated critical time, leaves us with the \emph{resonant sub-system}: 
\begin{subequations} \label{def:resonantsub2}
\begin{align}
\partial_t \widehat{f}(t,k,\eta) & = -\delta\widehat{f}(t,k,kt) \widehat{W}(k) k(\eta-kt) e^{-\abs{\eta-kt}}, \quad  t \in I_{k,\eta}, \\
\partial_t \widehat{f}(t,k-1,\eta) & = -\epsilon\widehat{f}(t,k,k t) \widehat{W}(k) k \left(\eta-(k-1)t\right)e^{-\abs{\eta-tk}}, \quad  t \in I_{k,\eta}
\end{align}
\end{subequations}
In this section, we will essentially treat the full second iterate system \eqref{eq:seciter} as a small perturbation of \eqref{def:resonantsub2} near the critical times $k^{-1}\eta_0$. 
The sub-system \eqref{def:resonantsub2} should be compared with the `toy models' of \cite{BM13,BGM15I,BGM15II}. 
\begin{remark} 
In \eqref{def:resonantsub2}, the evolution of $f(t,k,\eta)$ is the linearized Vlasov evolution, de-coupled from $f(t,k-1,\eta)$.
Aside from the obvious difference between the density and the Biot-Savart law of fluid mechanics, this decoupling is the main difference between the resonances in Vlasov \cite{MouhotVillani11,BMM13} and 2D Euler/Navier-Stokes near Couette flow \cite{BM13,BMV14}. 
In particular, this coupling appears to be the origin of the Gevrey-2 regularity requirement in \cite{BM13,BMV14}, as opposed to Gevrey-3 as it is in \cite{MouhotVillani11,BMM13}. 
\end{remark} 

\begin{proposition}[Instability of second iterate system in the gravitational case] \label{prop:blowup}
Let $\delta \ll 1$. Then there exists a constant $K_m'$ such that the solutions $f_{++}$ and $f_{--}$ to \eqref{eq:seciter} satisfy the following: 
\begin{align*}
\epsilon' e^{3(K_m' \epsilon)^{1/3} \eta_0^{1/3}} e^{-\abs{\eta-\eta_0}} - \epsilon' e^{-\frac{1}{8}\abs{\eta-\eta_0}} & \lesssim \abs{\widehat{f^H_{++}}(\eta_0,1,\eta)} \\ 
\epsilon' e^{3(K_m' \epsilon)^{1/3} \eta_0^{1/3}} e^{-\abs{\eta+\eta_0}} - \epsilon' e^{-\frac{1}{8}\abs{\eta+\eta_0}} & \lesssim \abs{\widehat{f^H_{--}}(\eta_0,-1,\eta)}. 
\end{align*}  
\end{proposition}
\begin{proof} 
By reality, it suffices to treat only $f^H_{++}$, here denoted $f^H$ for simplicity. 
For simplicity, we denote  $\rho^H = \rho^H_{++}$. 
Let $\gamma$ be a large, fixed multiple of $R$ with constant $C_\gamma$ to be chosen below: 
\begin{align*}
\gamma = C_\gamma R. 
\end{align*}
We will proceed iteratively over the echo times. 
First, we propagate the lower bound for early times before the first significant resonance. 
\begin{lemma}[Short time] \label{lem:lowst} 
For all $\gamma > 0$ and $t_{in} < t < t_{k_0}$, there holds for all $k,\eta$, 
\begin{align*}
\widehat{f^H}(t,k_0,\eta) \geq \epsilon' e^{-\frac{1}{2}\abs{\eta-\eta_0}} - \epsilon^\gamma \epsilon' e^{-\frac{1}{8}\abs{\eta-\eta_0}}. 
\end{align*}
\end{lemma} 
\begin{proof} 
We have
\begin{align*}
\widehat{f^H}(t,k,\eta) - \widehat{f^H}(t_{in},k,\eta) & = -\delta\int_{t_{in}}^t \widehat{\rho^H}(\tau,k) \widehat{W}(k) k (\eta-\tau k) e^{-\abs{\eta-\tau k}} d\tau \\ & \quad - \sum_{\ell = k\pm 1} \int_{t_{in}}^t \widehat{\rho^H}(\tau,\ell) \widehat{W}(\ell) \ell (\eta-t k) e^{-\abs{\eta-t\ell}} d\tau \\ 
& = L + NL. 
\end{align*}
The lower bound is satisfied by $f(t_{in})$ by definition.  
By Lemma \ref{lem:distup_preres}, there holds 
\begin{align*}
\abs{L} & \lesssim \epsilon' \epsilon^{\gamma} \delta \int_{t_{in}}^{t} e^{-\frac{1}{8}\abs{\eta_0 - kt}} e^{-\abs{\eta-\tau k}}\jap{\eta-\tau k} d\tau 
 \lesssim \epsilon' \epsilon^{\gamma} \delta e^{-\frac{1}{8}\abs{\eta-\eta_0}}, 
\end{align*}
The nonlinear term follows similarly and is omitted for brevity. 
\end{proof} 

We next need to propagate lower bounds through all of the critical times. 
For universal constants $K_m'$, $K_m''$ to be fixed  below, define the growth factors
\begin{align*}
k_0(\eta_0) & = \textup{Floor}\left( (K_m'\epsilon)^{1/3}\eta_0^{1/3}\right), \\ 
D_k(\eta_0) & = 1 &\quad k \geq k_0, \\ 
D_k(\eta_0) & = \left(\frac{\eps K_m' \eta_0}{(k+1)^3}\right)\cdots \left(\frac{\eps K_m' \eta_0}{(k_0-1)^3} \right)\left(\frac{\eps K_m' \eta_0}{k_0^3} \right),  & \quad\quad 1 \leq k < k_0, \\ 
k_0'(\eta_0) & = \textup{Floor}\left( (K_m''\epsilon)^{1/3}\eta_0^{1/3}\right), \\ 
D_k'(\eta_0) & = 1 &\quad k \geq k_0', \\ 
D_k'(\eta_0) & = \left(\frac{\eps K_m'' \eta_0}{(k+1)^3}\right)\cdots \left(\frac{\eps K_m'' \eta_0}{(k_0'-1)^3} \right)\left(\frac{\eps K_m'' \eta_0}{(k_0')^3} \right)  & \quad\quad 1 \leq k < k_0'.
\end{align*}
We proceed inductively over critical times. 
Let $1 \leq k \leq k_0$ and assume that the following holds: 
\begin{align}
\abs{\widehat{f^H}(t_{k},k,\eta)} \geq D_{k}(\eta_0)\epsilon' e^{-\abs{\eta-\eta_0}} - D_{k}'(\eta_0)\epsilon^\gamma \epsilon' e^{-\frac{1}{8}\abs{\eta-\eta_0}}. \label{ineq:blowlwbdinduct}
\end{align}
Lemma \ref{lem:lowst} implies that this holds for $k = k_0$. 
Proposition \ref{prop:blowup} then reduces to proving that \eqref{ineq:blowlwbdinduct} implies (for suitably chosen $K_m'$ and $K_m''$),  
\begin{align}
\abs{\widehat{f^H}(t_{k-1},k-1,\eta)} \geq D_{k-1}(\eta_0)\epsilon' e^{-\abs{\eta-\eta_0}} - D_{k-1}'(\eta_0)\epsilon^{\gamma} \epsilon' e^{-\frac{1}{8}\abs{\eta-\eta_0}}. \label{ineq:ftkm1Low}
\end{align}
The first step is proving that the critical density is large. 
By Lemma \ref{lem:basicvolt} and Corollary \ref{cor:linsolve}, over $I_{k,\eta_0}$ the critical density mode is given by
\begin{align*}
\widehat{\rho^H}(t,k) & = \widehat{f^H}(t_k,k,kt) + \sqrt{\delta} \int_{t_k}^t  \sinh(\sqrt{\delta}(t-\tau)) e^{-\abs{k}(t-\tau)} \widehat{f^H}(t_k,k,k\tau) d\tau \\ 
& \quad - \epsilon\sum_{\ell = k\pm 1}\int_{t_k}^t \widehat{\rho^H}(\tau,\ell) \ell \widehat{W}(\ell) k(t-\tau) e^{-\abs{kt - \ell \tau}} d\tau  \\ 
& \quad - \epsilon\sqrt{\delta}\sum_{\ell = k\pm 1}\int_{t_k}^t\sinh(\sqrt{\delta}(t-\tau)) e^{-\abs{k}(t-\tau)}  \int_{t_k}^\tau \widehat{\rho^H}(s,\ell) \ell \widehat{W}(\ell) k(\tau-s) e^{-\abs{k\tau - \ell s}} ds d\tau \\ 
&=\sum_{j=1}^4 I_j. 
\end{align*} 
The linear term is large, and hence must be controlled carefully. 
Here we crucially use the gravitational interaction: as the integral kernel is \emph{positive}, there holds from \eqref{ineq:blowlwbdinduct},  
\begin{align}
I_2 & \geq - D_k'(\eta_0)\epsilon^\gamma \epsilon' \sqrt{\delta} \int_{t_k}^t  \sinh(\sqrt{\delta}(t-\tau)) e^{-\abs{k}(t-\tau)} e^{-\frac{1}{8}\abs{k\tau-\eta_0}} d\tau \nonumber \\
& \gtrsim -\sqrt{\delta}D_k'(\eta_0)\epsilon^\gamma \epsilon' e^{-\frac{1}{8}\abs{\eta_0 - kt}}. \label{ineq:I2ctrlgrav}
\end{align}
The remaining terms are non-critical error terms. In particular, as $\ell = k \pm 1$ is not critical, we can gain arbitrary powers of $t$ as the $\ell$-th density mode is hence very small.
For $I_3$, using that $\abs{\eta_0 - \tau \ell} \gtrsim \tau \geq t_{in} = \epsilon^{-q}$ on the support of the integrand and Lemma \ref{lem:AprxRhoUp}, we have   
\begin{align*}
\abs{I_3} &  \lesssim \epsilon \epsilon' \sum_{\ell = k \pm 1} C_{\ell}(\eta_0) \int_{t_k}^te^{-\frac{1}{4}\abs{\eta_0 - \tau \ell}} \jap{\tau} e^{-\abs{kt-\ell \tau}} d\tau \\
& \lesssim \epsilon \epsilon' e^{-\frac{1}{8}\abs{\eta_0 - kt}} \sum_{\ell = k \pm 1} C_{\ell}(\eta_0) \int_{t_k}^t \jap{\tau} e^{-\frac{\tau}{8}} e^{-\frac{3}{4}\abs{kt-\ell \tau}} d\tau \\ 
& \lesssim_{q,\gamma} \epsilon^{2\gamma} \epsilon' e^{-\frac{1}{8}\abs{\eta_0 - kt}}. 
\end{align*}
The $I_4$ term follows similarly.
Therefore, there is a large constant $C > 0$, such that the following lower bound on $\hat{\rho}(t,k)$ holds: 
\begin{align}
\hat{\rho}(t,k) \geq  D_k(\eta_0)\epsilon' e^{-\abs{kt-\eta_0}} - \left(1 + C\sqrt{\delta}\right) D_k'(\eta_0) \epsilon^{\gamma} \epsilon' e^{-\frac{1}{8}\abs{kt-\eta_0}}. \label{ineq:rhoClow}
\end{align}
Turn next to the distribution function $\widehat{f}(t_{k-1},k-1,\eta)$: 
\begin{align}
\widehat{f}(t_{k-1},k-1,\eta) & = -\epsilon \int_{t_k}^{t_{k-1}} \hat{\rho}(\tau,k) k \widehat{W}(k)(\eta-(k-1)\tau) e^{-\abs{\eta-k\tau}} d\tau \nonumber \\ 
& \quad + \hat{f}(t_k,k-1,\eta) + \delta \int_{t_k}^{t_{k-1}} \hat{\rho}(\tau,k-1)(k-1) \widehat{W}(k-1)(\eta-(k-1)\tau)e^{-\abs{\eta-(k-1)\tau}} d\tau \nonumber \\ 
& \quad + \int_{t_k}^{t_{k-1}} \hat{\rho}(\tau,k-2)(k-2) \widehat{W}(k-2)(\eta-(k-1)\tau) e^{-\abs{\eta-(k-2)\tau}} d\tau \nonumber \\ 
& = \mathcal{I}_C  + \sum_{j=1}^3\mathcal{E}_j. \label{def:gravlbftkm1}
\end{align}
The growth comes from the leading term and the others are error.  
Sub-divide again: 
\begin{align*}
\mathcal{I}_C = \left(\mathbf{1}_{\abs{\eta- \eta_0} < \eta_0/2} + \mathbf{1}_{\abs{\eta- \eta_0} \geq \eta_0/2}\right)\mathcal{I}_C = \mathcal{I}_{C;R} + \mathcal{I}_{C;NR}. 
\end{align*}
Using that Lemma \ref{lem:AprxRhoUp} implies that $\hat{\rho}(t,k)$ is localized near $t \approx \eta_0/k$ we have for any $\gamma$,  
\begin{align}
\mathcal{I}_{C;NR} & \lesssim \epsilon \epsilon' \mathbf{1}_{\abs{\eta- \eta_0} \geq \eta_0/2}  \int_{t_k}^{t_{k-1}} C_k(\eta_0)e^{-\frac{1}{4}\abs{\eta_0 - k\tau}} \jap{\tau} e^{-\abs{\eta-k\tau}} d\tau \lesssim_\gamma \epsilon' \epsilon^{\gamma+1} e^{-\frac{1}{8}\abs{\eta-\eta_0}}. \label{ineq:ICNR}
\end{align}  
By positivity of the kernel on the support of the integrand (due to the restriction $\abs{\eta-\eta_0} \leq \eta_0/2$), we have by \eqref{ineq:rhoClow}, 
\begin{align*} 
\mathcal{I}_{C;R} & \geq -\epsilon D_k(\eta_0)\epsilon' \mathbf{1}_{\abs{\eta- \eta_0} \geq \eta_0/2} \int_{t_k}^{t_{k-1}} k \widehat{W}(k)(\eta-(k-1)\tau) e^{-\abs{\eta-k\tau}} e^{-\abs{k\tau-\eta_0}} d\tau \\  
& \quad + \left(1 + C\sqrt{\delta}\right)\epsilon^{1+\gamma} D_k'(\eta_0) \epsilon' \mathbf{1}_{\abs{\eta- \eta_0} \geq \eta_0/2} \int_{t_k}^{t_{k-1}} k \widehat{W}(k)(\eta-(k-1)\tau) e^{-\abs{\eta-k\tau}} e^{-\frac{1}{8}\abs{k\tau-\eta_0}} d\tau \\ 
& \gtrsim \epsilon D_k(\eta_0)\epsilon' e^{-\abs{\eta_0-\eta}} \frac{\eta_0}{k^2}\int_{t_k}^{t_{k-1}} e^{-2\abs{k\tau-\eta_0}} d\tau \\  
& \quad - \epsilon^{1+\gamma} (1+C\sqrt{\delta}) D_k'(\eta_0) \epsilon' e^{-\frac{1}{8}\abs{\eta - \eta_0}} \mathbf{1}_{\abs{\eta- \eta_0} \geq \eta_0/2} \int_{t_k}^{t_{k-1}} (-\widehat{W}(k))k (\eta-(k-1)\tau) e^{-\frac{7}{8}\abs{\eta-k\tau}} d\tau. 
\end{align*} 
Note that for $k \leq k_0$, the integral in the first term satisfies  
\begin{align*}
\int_{t_k}^{t_{k-1}} e^{-2\abs{k\tau-\eta_0}} d\tau = \frac{1}{k}\int_{\eta_0 - \frac{\eta_0}{(k+1)}}^{\eta_0 + \frac{\eta_0}{k-1}} e^{-2\abs{s-\eta_0}} ds \gtrsim \frac{1}{k}. 
\end{align*}
The integral in the error term can be bounded similarly, $k\tau \approx \eta_0 \approx \eta$.  
Hence we have the following lower bound for universal constants $K_m'$ and $K_m''$ (if need be, choosing $\delta$ small),  
\begin{align}
\mathcal{I}_{C;R} & \geq \frac{K_m' \epsilon \eta_0}{\abs{k}^3} D_k(\eta_0) \epsilon' e^{-\abs{\eta-\eta_0}} - \frac{1}{2}\left(\frac{K_m''\epsilon \eta_0}{\abs{k}^3}\right) \epsilon^\gamma  D'_k(\eta_0) \epsilon' e^{-\frac{1}{8}\abs{\eta-\eta_0}}. \label{ineq:ICR}
\end{align}
Note that none of the constants depend on $\gamma$.
Using non-resonance, the other terms in \eqref{def:gravlbftkm1} $\mathcal{E}$ are all bounded above in absolute value via  
\begin{align*}
\sum_{j=1}^3 \mathcal{E}_j \lesssim \epsilon^{\gamma+1} \epsilon' e^{-\frac{1}{8}\abs{\eta- \eta_0}}. 
\end{align*}
Putting this together with \eqref{ineq:ICNR} and \eqref{ineq:ICR}, for $\epsilon$ sufficiently small, we deduce \eqref{ineq:ftkm1Low}.
 
By iterating over $k$, we have 
\begin{align}
\abs{\widehat{f}(t_{1},1,\eta)} \geq D_{1}(\eta_0)\epsilon' e^{-\abs{\eta-\eta_0}} - \epsilon^{\gamma} D_1'(\eta_0) \epsilon' e^{-\frac{1}{8}\abs{\eta-\eta_0}}. \label{ineq:lastbd}
\end{align}
Via an easy variation of the preceding arguments, we further deduce the following for some large constant $C>0$, 
\begin{align}
\widehat{\rho}(t,1) \geq D_{1}(\eta_0)\epsilon' e^{-\abs{kt-\eta_0}} - \left(1 + C\sqrt{\delta}\right)\epsilon^{\gamma} \epsilon' D_1'(\eta_0) e^{-\frac{1}{8}\abs{kt-\eta_0}}, \label{ineq:lastbdrho}
\end{align}
and similarly 
\begin{align}
\widehat{f}(\eta_0,1,\eta) \gtrsim D_{1}(\eta_0)\epsilon' e^{-\abs{\eta-\eta_0}} - \epsilon^{\gamma} D_1'(\eta_0) \epsilon' e^{-\frac{1}{8}\abs{\eta-\eta_0}}. \label{ineq:lastbd2}
\end{align}
By Lemma \ref{lem:GrowthFact} and \eqref{def:eta0}, it follows that by choosing $\gamma = C_\gamma R$ for $C_\gamma$ large relative to a universal constant and then choosing $\epsilon$ small, we have  
\begin{align*}
\epsilon^\gamma D_1'(\eta_0) \approx \frac{\epsilon^\gamma}{(K_m'' \epsilon \eta_0)^{1/2}} \left[ \frac{\jap{k_0,\eta_0}^R}{\epsilon} \right]^{\left(\frac{K_m''}{K_m'}\right)^{1/3}} \ll 1. 
\end{align*}
Therefore, the result follows by Lemma \ref{ineq:lastbd2} and a small adjustment to $K_m'$ to deal with the $(\epsilon \eta_0)^{-1/2}$. 
This completes the proof of Proposition \ref{prop:blowup}. 
\end{proof} 

\begin{remark} 
In the electrostatic case, the issue is the lack of positivity in the solution to the linear problem, which makes it difficult to propagate such a lower bound such as \eqref{ineq:ftkm1Low}. 
\end{remark}

\section{Norms and related preliminaries for stability estimates} \label{sec:norms} 
As discussed in \S\ref{sec:Proof}, one of the main steps of the proof is to design a precise norm with which to measure $g$. 
We will use the following Fourier multiplier to build the high norm to measure $g$, for constants $K$, $r$, (to be chosen later) and a time-dependent index $\mu(t)$, 
\begin{subequations} \label{def:AG}
\begin{align}
A(t,\grad) & = \jap{\grad}^{\beta} e^{\mu(t)(K\epsilon)^{1/3}\jap{\grad}^{1/3}} G(t,\grad), \\ 
G(t,\grad) & = \left(\frac{e^{r(K\epsilon)^{1/3}\jap{\partial_v}^{1/3}}}{w(t,\partial_v)} + e^{r(K\epsilon)^{1/3}\jap{\partial_z}^{1/3}}\right). \label{def:G}
\end{align}
\end{subequations}
The multiplier $w$ is essentially a continuous-time, general-$\eta$ analogue of the $C_k(\eta_0)$ weights (defined in \eqref{def:Cketa0}) used to obtain the upper bounds on the approximate solution in Lemma \ref{lem:AprxRhoUp}. 
This multiplier is discussed further below in \S\ref{sec:G}. Among other properties, there holds $e^{-\frac{1}{2}r(K\epsilon)^{1/3 \abs{\eta}^{1/3}}} \leq w(t,\partial_v) \leq 1$. 
For constants $\mu_\infty > 12$ and $b \in (0,1/6)$ fixed by the proof  (compare with \eqref{eq:rtrhoup}), we define,  
\begin{align}
\mu(t) = \mu_\infty\left(1 + \left(\frac{t_{in}}{t}\right)^b\right). \label{def:mu}
\end{align}
For the low norm, we use a standard Gevrey norm: 
\begin{align}
B(t,\grad) = \jap{\grad}^{\gamma} e^{\nu(t)(K\epsilon)^{1/3}\jap{\grad}^{1/3}}. \label{def:B}
\end{align} 
We let $\gamma \in \left(\frac{3\beta}{4}+3, \frac{3\beta}{4}+4\right)$ and $\nu(t) = (1-c_r)\mu(t)$ for some $c_{r} \ll 1$ small to be chosen later.  

\subsection{Definition and basic properties of $G$} \label{sec:G}
Recall the definition of the critical intervals from \eqref{def:ICk}.
Fix a constant $K > 0$ to be chosen later depending only universal constants and define the following non-negative integer,  
\begin{align}
N(\eta) = \textup{floor}\left((K\epsilon \abs{\eta})^{1/3}\right),\label{def:N}
\end{align}
Note that the definition implies $N = 0$ if $\abs{\eta} < (K\epsilon)^{-1}$. 
Recall the definition of critical intervals in \eqref{def:ICk}. We further write, for $\abs{k} \geq 2$ (with analogous definition for $\abs{k}=1$),   
\begin{align*}
I_{k,\eta}^R  = \left[\frac{\eta}{k}, \frac{\eta}{k} + \frac{\abs{\eta}}{2\abs{k}(\abs{k}-1)} \right], \quad\quad I_{k,\eta}^L = \left[\frac{\eta}{k} - \frac{\abs{\eta}}{2\abs{k}(\abs{k}+1)}, \frac{\eta}{k} \right]. 
\end{align*}
We then define $\tilde{w}$ as the following (the definition is recursive \emph{backwards} in time, as in \cite{BM13,BGM15I}), 
\begin{subequations} \label{def:tildw}
\begin{align}
\tilde{w}(t,\eta) & = 0 \quad t \geq t_{0,\eta}, \\ 
\tilde{w}(t,\eta) & = \frac{k^3}{K\epsilon \eta}\left(1 + a_{k,\eta}\frac{K\epsilon}{\abs{k}}\abs{t-\frac{\eta}{k}}\right)\tilde{w}(t_{k-1},\eta) \quad t \in I_{k,\eta}^R, \\  
\tilde{w}(t,\eta) & = \left(1 + b_{k,\eta}\frac{K\epsilon}{\abs{k}}\abs{t-\frac{\eta}{k}}\right)^{-1}\tilde{w}(\frac{\eta}{k},\eta) \quad t \in I_{k,\eta}^L, \\  
\tilde{w}(t,\eta) & = \tilde{w}(t_{N,\eta},\eta) \quad t \leq t_{N,\eta},  
\end{align}
\end{subequations}
where $a_{k,\eta},b_{k,\eta}$ are defined so that the regularity loss each half interval is exactly $(K\epsilon \eta) k^{-3}$. 
Therefore, for $\abs{k} > 1$ we have,  
\begin{align*}
a_{k,\eta} & = \frac{2(\abs{k}-1)}{\abs{k}}\left[1 - \frac{k^3}{K\epsilon \eta}\right], \quad\quad \abs{k} > 1\\ 
a_{k,\eta} & = \frac{1}{2}\left(1 - \frac{1}{K\epsilon\abs{\eta}}\right), \quad\quad \abs{k} = 1,\\ 
b_{k,\eta} & = \frac{2(\abs{k}+1)}{\abs{k}}\left(1 - \frac{k^3}{K\epsilon \eta}\right), \quad\quad \abs{k} \geq 1.  
\end{align*}
Note that $b_{k,\eta},a_{k,\eta} \in [0,4]$ in the range of $k$ and $\eta$ which are possible in the definition \eqref{def:tildw} and also depend mildly on $K$ and $\epsilon$.  
Naturally, we take the convention that if $N = 0$, then
\begin{align*}
\tilde{w}(t,\eta) & = 1, \forall t. 
\end{align*}
Notice that this implies $\tilde{w}$ is constant unless $\abs{\eta} \geq (K\epsilon)^{-1}$.
\begin{remark} \label{rmk:wsimilarity}
We remark that the similarity between the definition in \eqref{def:tildw} and that used in \cite{BM13} was specifically motivated by an interest in easily deriving useful properties of $w(t,\eta)$ by adapting the ideas from \cite{BM13}. Indeed, there seems to be more flexibility here than in \cite{BM13} and there are a variety of choices which could serve our purposes, however, this seems to be easiest. See Appendix \ref{sec:ApxNorms}. 
\end{remark} 
We will need to differentiate $w$ with respect to $\eta$ in order to take a moment in velocity. 
Therefore, let $\varphi \in C_c^\infty((-1,1))$ be non-negative with $\int \varphi dx = 1$ and define $w$ as the mollified version of $\tilde{w}$: 
\begin{align}
w(t,\eta) = \int_{-\infty}^\infty \varphi(\eta-\xi) \tilde{w}(t,\xi) d\xi. \label{def:wconv}
\end{align}
Hence, $w(t,\eta)$ is a smooth function of $\eta$ and $\partial_\eta w$ can be easily compared to $w$. 

For the remainder of the section, we outline some properties of $w$ and $G$ which we will need going forward. 
The proofs are tedious and are reserved for Appendix \ref{sec:ApxNorms}; moreover, several proofs are variants of proofs found in \cite{BM13}. 
The first lemma determines the precise growth of $w$. 
\begin{lemma}[Total growth of $w$] \label{lem:growthw} 
If $K\epsilon\abs{\eta} \geq 1$ (otherwise $w \equiv 1$), then 
\begin{align*}
\frac{w(t_0,\eta)}{w(1,\eta)} \approx \left(K\epsilon\abs{\eta} \right)^{-1} e^{6(K\epsilon)^{1/3}\abs{\eta}^{1/3}}. 
\end{align*}
\end{lemma} 
It will then suffice to fix $r = 12$ in the definition of $G$ in \eqref{def:AG}. 

The following lemma emphasizes that ratios of $w$ account for the growth due to the echo resonances. 
In particular, this lemma represents the primary use of $w$ (which is found in \S\ref{sec:LRg}) and the motivation for the definition.    
\begin{lemma} \label{lem:resw}
Let $\tau \in I_{k+1,kt}$ and $1 \leq k \leq N(\eta)-1$. Then, 
\begin{align*}
\frac{w(\tau,kt)}{w(t,kt)} & \lesssim \left(\frac{\abs{k}^2}{\epsilon K t}\right)^2 \left(1 + \frac{K\epsilon}{k^2}\abs{(k+1)\tau-kt}\right) \lesssim \frac{\abs{k}^2}{\epsilon K t}, 
\end{align*}
where the implicit constant does not depend on $\epsilon$, $K$, or $t$. 
\end{lemma} 
\begin{remark} \label{rmk:genwres}
If one takes $\tau \in I_{j,kt}$ for any $j > k+1$, Lemma \ref{lem:resw} still holds. This is relevant to the proof of Theorem \ref{thm:optimal} but not Theorem \ref{thm:main}. 
\end{remark} 
In order for $A$ to make a reasonable norm,  we  need the following lemma (see Appendix \ref{sec:ApxNorms}).    
\begin{lemma} \label{lem:Gcomp}
There exists a universal $\tilde{r} > 0$ such that the followings holds (with constant independent of $\epsilon$, $K$, and $t$), 
\begin{align*}
\frac{G(t,k,\eta)}{G(t,\ell,\xi)} & \lesssim e^{\tilde{r}(K\epsilon)^{1/3}\jap{k-\ell,\eta-\xi}^{1/3}}. 
\end{align*}
\end{lemma}

We also need the following commutator-like estimate in order to take advantage of the transport structure. 
As in \cite{BM13}, this estimate motivates the $+ e^{r(K\epsilon)^{1/3}\jap{k}^{1/3}}$ in \eqref{def:G}.  
See Appendix \ref{sec:ApxNorms}.  

\begin{lemma} \label{lem:commG}
Suppose that either $t < \frac{1}{2}(K\epsilon)^{-1/3}\min(\abs{\eta}^{2/3},\abs{\xi}^{2/3})$ holds \emph{or} $\abs{k,\ell} \geq 10\abs{\eta,\xi}$ holds (not exclusive). Then there is some $\tilde{r} > 0$ universal such that  
\begin{align*}
\abs{\frac{G(t,k,\eta)}{G(t,\ell,\xi)}- 1} \lesssim \left(\frac{1}{(K\epsilon)^{2/3}}\right)\frac{\jap{k-\ell,\eta-\xi}}{\jap{k,\eta}^{2/3} + \jap{\ell,\xi}^{2/3}} e^{\tilde{r}(K\epsilon)^{1/3}\jap{k-\ell,\eta-\xi}^{1/3}}, 
\end{align*}
where the implicit constant does not depend on $K$, $\epsilon$, or $t$. 
\end{lemma} 
Next, we need the following lemma for deducing moment controls. 
\begin{lemma} \label{lem:MomentG}
For all $t,\eta$, there holds 
\begin{align*}
\abs{\partial_\eta G(t,k,\eta)} & \lesssim G(t,k,\eta). 
\end{align*}
\end{lemma} 

Lemma \ref{lem:MomentG} ensures the following property of $A$. 
\begin{lemma} \label{lem:AMomentEquiv}
There holds the following for an arbitrary function $q$,  
\begin{align*}  
\norm{A q} + \norm{\partial_\eta \left( A q \right)}_2 \approx \norm{Aq} + \norm{A \partial_\eta q}. 
\end{align*} 
\end{lemma}

\subsection{Paraproduct decompositions and further properties of $A$ and $B$} \label{sec:paranote}
We will need to estimate terms of the general form: 
\begin{align*}
\Sigma(t,z,v) = H(t,z+tv) F(t,z,v),  
\end{align*}
where $\int_{\Real} H(t,z) dz = 0$. 
On the frequency-side this becomes
\begin{align*}
\widehat{\Sigma}(t,k,\eta) = \frac{1}{2\pi}\sum_{\ell \in \Integers_\ast}\widehat{H}(t,\ell) \widehat{F}(t,k-\ell,\eta-t\ell). 
\end{align*}
The Littlewood-Paley projection of functions depending only on $z$ is defined as follows, for $M \in 2^{\Integers}$, 
\begin{align*}
\widehat{H_M}(t,k) = \widehat{H}(t,k) \chi\left(\frac{\abs{k,k t}}{M}\right), \quad\quad H_{<M} = \sum_{N \in 2^{\Integers}: N < M} H_N, 
\end{align*}
where $\chi$ is a smooth cut-off supported on $(1/2,3/2)$ and equal to $1$ on $(3/4,1)$ (see Appendix \ref{apx:Gev}). 
Due to \eqref{eq:rhokkt}, this is consistent with the corresponding definition for functions of $z$ and $v$ in \eqref{def:fM}. 
We use the following paraproduct, introduced by Bony \cite{Bony81}: 
\begin{align} 
\Sigma  & = \sum_{M \in 2^\Integers} (H \circ \mathcal{T}_{t})_{M} F_{<M/8} + \sum_{M \in 2^{\Integers}} (H \circ \mathcal{T}_{t})_{<M/8} F_{M} \nonumber \\ & \quad + \sum_{M \in 2^{\Integers}} \sum_{M/8 \leq M^\prime \leq 8M} (H \circ \mathcal{T}_{t})_{M} F_{M^\prime} \nonumber \\
& = \Sigma_{HL} + \Sigma_{LH} + \Sigma_{\mathcal{R}}  \label{def:parapp}
\end{align}  
For the majority of the proof, we will not need these decompositions. Instead, we mostly rely on some ``black-box'' product-type inequalities relating $A$ and $B$, outlined in Lemmas \ref{lem:AB}  and \ref{lem:GenCEst} below.
The proofs are similar to some appearing in \cite{BM13,BMM13}; they are straightforward applications of basic properties of $A$, $B$, and a few tricks for paradifferential calculus in Gevrey regularity.  

\begin{lemma} \label{lem:AB} 
Let $\Sigma_{HL}$ and $\Sigma_{LH}$ be defined as in \eqref{def:parapp} above. For $\beta$ large relative to a universal constant and $\nu(t) \geq c\mu(t) + \tilde{r}$, where $c \in (0,1)$ is a fixed constant, we have the following $\forall\,\delta' > 0$, 
\begin{subequations}
\begin{align}
\norm{A(t)\Sigma_{HL}}_2 & \lesssim_{\beta,\gamma,\delta'} \norm{AH}_2\norm{\jap{\grad}^{-\gamma+1+\delta'}BF}_2, \label{ineq:ABprodHL}\\
\norm{A(t) \Sigma_{LH}}_2 & \lesssim_{\beta,\gamma,\delta'} \norm{AF}_2 \norm{\jap{\partial_x,t\partial_x}^{-\gamma+1+\delta'} B H}_2. \label{ineq:ABprodLH}
\end{align}
\end{subequations} 
Let $\Sigma_{\mathcal{R}}$ be defined as in \eqref{def:parapp} above. For $\beta$ large relative to a universal constant, and $\nu(t) \geq \tilde{c}\mu(t) + \frac{3}{2}r$, for $\tilde{c}$ a universal constant and $\forall\, \delta' > 0$, 
\begin{align}
\norm{A(t)\Sigma_{\mathcal{R}}}_2 & \lesssim_{\beta,\gamma,\delta'} \norm{\jap{\grad}^{\beta/2-\gamma} BH}_2\norm{\jap{\grad}^{\beta/2+1+\delta'-\gamma} BF}_2. \label{ineq:ABRem}
\end{align}
\end{lemma}
\begin{corollary} \label{cor:prod}
For $\beta$ large relative to a universal constant and $\nu$, $\mu$ chosen such that 
\begin{align} 
\nu(t) & \geq \max\left(c\mu(t) + \tilde{r},\tilde{c}\mu(t) + \frac{3}{2}r\right),\label{ineq:numurest}
\end{align}
where $c,\tilde{c}$ are defined in Lemma \ref{lem:AB}, then there holds: 
\begin{align}
\norm{A(t)\left(H\circ \mathcal{T}_t F\right)}_{2} & \lesssim \norm{\jap{\grad}^{\beta/2 - \gamma + 2}BH}_2\norm{AF}_2 + \norm{\jap{\grad}^{\beta/2-\gamma+2}BF}_2\norm{AH}_2. \label{ineq:prod}
\end{align}
There similarly holds 
\begin{align}
\norm{B(t)\left(H\circ \mathcal{T}_t F\right)}_{2} & \lesssim \norm{\jap{\grad}^{-\gamma/2 + 2}BH}_2\norm{BF}_2 + \norm{\jap{\grad}^{-\gamma/2+2}BF}_2\norm{BH}_2. \label{ineq:Bprod}
\end{align}
\end{corollary}
\begin{proof}[\textbf{Proof of Lemma \ref{lem:AB}}]
We only verify \eqref{ineq:ABprodHL}; inequalities \eqref{ineq:ABprodLH} and \eqref{ineq:ABRem} are simple variants.  
We project to a frequency shell $N \in 2^\Integers$, and hence by the frequency restrictions imposed by the Littlewood-Paley projections on the support of the integrand, we have by Lemma \ref{lem:Gcomp} and \eqref{lem:scon},  that there is a constant $c \in (0,1)$ such that 
\begin{align*}
\abs{\left(A(t)\Sigma_{HL}\right)_N} & \lesssim_\beta \sum_{M \in 2^\Integers: M \sim N} \sum_{\ell \in \Integers_\ast}  A(t,\ell,t\ell) \abs{\widehat{H}_{M}(t,\ell)}  \\ & \quad \quad \times e^{\left(c\mu(t) + \tilde{r}\right)(K\epsilon)^{1/3}\jap{k-\ell,\eta-t\ell}^{1/3}} \abs{\widehat{F}_{<M/8}(t,k-\ell,\eta-t\ell)}. 
\end{align*}
For $\nu(t)$ and $\mu(t)$ chosen such that 
\begin{align*}
\nu(t) & \geq c\mu(t) + \tilde{r}, 
\end{align*}
we have by the definition of \eqref{def:B} followed by \eqref{ineq:L2L1},  
\begin{align*}
\abs{\left(A(t)\Sigma_{HL}\right)_N}  & \lesssim_\beta \sum_{M \in 2^\Integers : M \sim N} \sum_{\ell \in \Integers_\ast}  A(t,\ell,t\ell) \abs{\widehat{H}_{M}(t,\ell)} \jap{k-\ell,\eta-t\ell}^{-\gamma} \abs{B\widehat{F}_{<M/8}(t,k-\ell,\eta-t\ell)} \\ 
& \lesssim \sum_{M \in 2^\Integers: M \sim N} \norm{A H_M}_2\norm{\jap{\grad}^{-\gamma+1+\delta'}BF_{<M/8}}_2. 
\end{align*}
Then \eqref{ineq:ABprodHL} follows from \eqref{ineq:GeneralOrtho}. 
\end{proof}

To control the density, we will also need a version of Corollary \ref{cor:prod} for which the nonlinearity is integrated in time and restricted in frequency. 
The following lemma will be sufficient to treat most terms in the density estimates. 
\begin{lemma} \label{lem:GenCEst}
For $r = r(t,x)$ and $q = q(t,x,v)$, define
\begin{align*}
\mathcal{C}(t,k) = \sum_{\ell \in \Integers_\ast}\int_{t_{in}}^t \hat{r}(\tau,\ell) \widehat{W}(\ell) \ell k(t-\tau) \widehat{q}(\tau,k-\ell,kt-\ell \tau) d\tau. 
\end{align*}
Let $\nu$ and $\mu$ satisfy \eqref{ineq:numurest}. 
Then, for any time interval $I = [t_{in},T]$, for $m,\alpha \geq 0$ arbitrary,  
\begin{align}
\norm{\jap{\grad}^m A\mathcal{C}}_{L^2_t(I;L^2_k)} & \lesssim_{\sigma,\gamma,m} \norm{\jap{\partial_x,\partial_x t}^{(m+\beta)/2-\gamma'+\alpha+1}Br}_{L^2_t(I;L^2_k)} \sup_{\tau \in I}\left( \jap{\tau}^{-\alpha} \norm{\jap{v}\jap{\grad}^{m+1}Aq(\tau)}_2\right) \nonumber \\ &   \quad + \left(\sup_{t \in I} \sup_{k \in \Integers_\ast} \int_{t_{in}}^T \sum_{\ell \neq 0} \bar{K}(t,\tau,k,\ell) d\tau\right)^{1/2}\left(\sup_{\tau \in I} \sup_{\ell \in \Integers_\ast} \int_{\tau}^{T} \sum_{k \neq 0} \bar{K}(t,\tau,k,\ell) d t\right)^{1/2} \nonumber \\ & \quad\quad \times \norm{\jap{\grad}^m Ar}_{L^2_t(I;L^2_k)} \left(\sup_{\tau \in I} \norm{\jap{v}\jap{\grad}^{-\gamma+2}Bq(\tau)}_{2}\right), \label{ineq:GenRhointbd}
\end{align}
where, for some universal $c > 0$, there holds
\begin{align*}
\bar{K}(t,\tau,k,\ell) = \abs{\widehat{W}(\ell) \ell k (t-\tau)} e^{\left(\mu(t)-\mu(\tau)\right)(K\epsilon)^{1/3}\jap{k,kt}^{1/3}} e^{-c(K\epsilon)^{1/3}\jap{k-\ell,kt-\ell \tau}^{1/3}} \jap{k-\ell,kt-\ell \tau}^{-\gamma}.
\end{align*}
The analogue of \eqref{ineq:Bprod} holds as well. 
\end{lemma} 
\begin{proof} 
Consider the $m = 0$ case; the other cases are similar. 
Expand via the paraproduct, 
\begin{align}
\mathcal{C} & = \sum_{M \in 2^\Integers}\sum_{\ell \in \Integers_\ast}\int_{t_{in}}^t \hat{r}_M(\tau,\ell)\widehat{W}(\ell) \ell k(t-\tau) \widehat{q}_{<M/8}(\tau,k-\ell,kt-\ell \tau) d\tau \nonumber \\ 
& \quad + \sum_{M \in 2^\Integers}\sum_{\ell \in \Integers_\ast}\int_{t_{in}}^t \hat{r}_{<M/8}(\tau,\ell) \widehat{W}(\ell) \ell k(t-\tau) \widehat{q}_{M}(\tau,k-\ell,kt-\ell \tau) d\tau \nonumber \\ 
& \quad + \sum_{M \in 2^\Integers}\sum_{\ell \in \Integers_\ast}\int_{t_{in}}^t \hat{r}_{M}(\tau,\ell)\widehat{W}(\ell) \ell k(t-\tau) \widehat{q}_{\sim M}(\tau,k-\ell,kt-\ell \tau) d\tau \nonumber \\ 
& = \mathcal{C}_{HL}+ \mathcal{C}_{LH} + \mathcal{C}_{\mathcal{R}}. \label{def:Cpara}
\end{align}
The $\mathcal{C}_{HL}$ term is treated in a manner similar to the ``reaction'' term in [\S5.1.1. \cite{BMM13}], the $\mathcal{C}_{LH}$ term is treated in a manner similar to the ``transport'' term in [\S5.1.2 \cite{BMM13}], and the $\mathcal{C}_{\mathcal{R}}$ term is treated in a manner similar to the ``remainder'' term in [\S5.1.3 \cite{BMM13}]. 
Hence, we omit the proof for brevity. 
\end{proof} 

We will also need versions of Lemmas \ref{lem:AB} and \ref{lem:GenCEst} for $F(t,x,v) = \partial_v f^0(v)$. 
We will simply state the version we need; the proof is omitted as it is an easier version of the above. 
\begin{lemma} \label{lem:Abckgr}
There holds the following, 
\begin{align}
\norm{A(t)\jap{v}\left((H\circ \mathcal{T}_t) \partial_v f^0\right)}_2 & \lesssim_{\beta} \delta \norm{AH}_2, \label{ineq:Abck}
\end{align}
and if we define 
\begin{align*}
\mathcal{C} = \int_{t_{in}}^t \hat{r}(\tau,k) \widehat{W}(k) \abs{k}^2 (t-\tau) \widehat{f^0}(k(t-\tau)) d\tau,  
\end{align*}
there holds for any $I = (t_{in},T)$, 
\begin{align*}
\norm{A\mathcal{C}}_{L^2_t(I;L^2)} \lesssim_\beta \delta \norm{Ar}_{L^2_t(I;L^2)}. 
\end{align*}
\end{lemma} 

\begin{remark} 
As $c,\tilde{c}$, $r$, and $\tilde{r}$ are all universal constants, we may choose $\mu_\infty \geq 12$ large enough and $c_r$ small enough such that \eqref{ineq:numurest} is satisfied if we set $\mu$ as in \eqref{def:mu} and $\nu(t) = (1-c_r)\mu(t)$. 
\end{remark}

\subsection{Estimates on the approximate solution and the consistency error} \label{sec:ErrorfH}
First, we want to estimate the size of $f^L$ in the norms defined by $A$ and $B$. 
\begin{lemma} \label{lem:AfLBfL}
For $t \in (t_{in},t_\star)$ there holds the following for $K\epsilon \leq 1$ and $\alpha \in \Real$, 
\begin{subequations} 
\begin{align}
\norm{\jap{\grad}^\alpha \jap{v}Af^L(t)}_2 & \lesssim_{\beta,\alpha} \epsilon, \label{ineq:AfL} \\  
\norm{\jap{\grad}^\alpha Af^L(t)}_2 & \lesssim_{\beta,\alpha} \epsilon e^{-\frac{1}{2}\epsilon^{-q}}. \label{ineq:BfL} 
\end{align}
\end{subequations}
\end{lemma} 
\begin{proof} 
Follows immediately from \eqref{ineq:IncExp}, and Lemma \ref{lem:growthw} (recall \eqref{def:AG} for the definition of $A$). The details are omitted for the sake of brevity. 
\end{proof} 

Next, we estimate $f^H$ in the norms defined by $A$ and $B$. 
This lemma, in particular that one gains in the low norms and loses in the high norms in \eqref{ineq:fHctrls} (which is a way of measuring that $f^H$ exists only at high frequencies), is one of the crucial ideas behind the proof of Theorem \ref{thm:main}. 
\begin{lemma} \label{lem:AfHBfH}
For $t \in (t_{in},t_\star)$, there holds the following for any $\alpha \geq -\beta/4$ and $\sigma$ chosen sufficiently large relative to $R,K,\mu$, and $K_m$,
\begin{subequations} \label{ineq:fHctrls}
\begin{align}
\norm{\jap{\grad}^\alpha \jap{v}Af^H}_2 & \lesssim_{\sigma,\alpha} \epsilon^{-a+1}\eta_0^{R(a-1)+\alpha}, \label{ineq:AfH} \\  
\norm{\jap{\grad}^\alpha \jap{v}Bf^H}_2 & \lesssim_{\sigma,\alpha} \eta_0^{-\sigma/5+\alpha}, \label{ineq:BfH} 
\end{align}
\end{subequations}
where $a > 1$ is given by 
\begin{align*}
a = K_m^{1/3}(K_m')^{-1/3} + \frac{2}{3}(2\mu_\infty + r)K^{1/3}(K_m')^{-1/3} + \frac{1}{3}, 
\end{align*}
where $K_m$ is defined via the growth factor \eqref{def:Cketa0}, appearing in Lemmas \ref{lem:distup} and \ref{lem:AprxRhoUp}, $K_m'$ is the lower bound growth factor appearing in Proposition \ref{prop:blowup}, and  $K$ is the constant in the definition of $G$ arising in \S\ref{sec:G} (see \eqref{def:AG} for the rest of the constants). 
In particular,  as both of these constants, along with $r$ and $\mu$, are fixed independent of $\sigma$, $a$ also does not depend on $\sigma$.  
\end{lemma} 
\begin{remark} 
The fact that $a$ is independent of $\sigma$ is crucial to the proof of Theorem  \ref{thm:main}.
\end{remark}
\begin{corollary} \label{cor:RhoHctrls}
 Lemma \ref{lem:AfHBfH} implies 
\begin{subequations} \label{ineq:RhoHctrls}
\begin{align}
\eta_0^{-1/2}\norm{\jap{\partial_x,\partial_x t}^\alpha A\rho^H}_{L^2_t(t_{in},t_\star;L^2)} + \norm{\jap{\partial_x,\partial_x t}^\alpha A\rho^H}_{L^\infty_t L^2} & \lesssim_{\sigma,\alpha} \epsilon^{-a+1}\eta_0^{R(a-1)+m}, \label{ineq:ArhoH} \\  
\eta_0^{-1/2}\norm{\jap{\partial_x, \partial_xt}^m B\rho^H}_{L^2_t(t_{in},t_\star; L^2)} + \norm{\jap{\partial_x, \partial_x t}^\alpha B\rho^H}_{L^\infty_t L^2} & \lesssim_{\sigma,\alpha} \eta_0^{-\sigma/5+m}. \label{ineq:BrhoH} 
\end{align}
\end{subequations}
\end{corollary} 
\begin{proof} 
Similar upper bounds hold for $vf^H$ as for $f^H$ (see Lemma \ref{lem:fHmoment}), and hence it suffices to show the proof without $\jap{v}$ (recall also Lemma \ref{lem:AMomentEquiv}).  
Moreover, it suffices to prove the result for $f^H_{++}$ (recall \eqref{def:fHpm}) as the other contributions are similar. 
For simplicity, we prove the result for $m = 0$; the more general case follows similarly. 

By Lemma \ref{lem:distup} and \eqref{def:B}  (recall also Lemma \ref{lem:GrowthFact}), and \eqref{def:eta0}, there holds the following, 
\begin{align*}
\abs{B \widehat{f^H_{++}}(t,k,\eta)} & \lesssim \epsilon \jap{k_0,\eta_0}^{-\sigma} \jap{k,\eta}^{\gamma} e^{\nu(t)(K\epsilon)^{1/3}\jap{k,\eta}^{1/3}} e^{3(K_m\epsilon)^{1/3}\eta_0^{1/3}} e^{-\frac{1}{8}\abs{\eta-\eta_0}}e^{-k_0^{-1}\abs{k}} \\  
& \lesssim \epsilon \jap{k_0,\eta_0}^{-\sigma} \left(\frac{\jap{k_0,\eta_0}^{R}}{\epsilon}\right)^{K_m^{1/3}(K_m')^{-1/3}}  \jap{k,\eta}^{\gamma} e^{\nu(t)(K\epsilon)^{1/3}\jap{k,\eta}^{1/3}}  e^{-\frac{1}{8}\abs{\eta-\eta_0}} e^{-k_0^{-1}\abs{k}}. 
\end{align*}
By \eqref{def:k0}, \eqref{ineq:IncExp}, and \eqref{ineq:SobExp} we have 
\begin{align*}
\abs{B \widehat{f^H}_{++}(t,k,\eta)} & \lesssim_{\gamma}  \epsilon \jap{k_0,\eta_0}^{-\sigma + \gamma} \left(\frac{\jap{k_0,\eta_0}^{R}}{\epsilon}\right)^{K_m^{1/3}(K_m')^{-1/3}} e^{2\nu(t)(K\epsilon)^{1/3}\jap{\eta_0}^{1/3}}  e^{-\frac{1}{10}\abs{\eta-\eta_0}} e^{-\frac{1}{2}k_0^{-1}\abs{k}}. 
\end{align*}
Using the definition of $\gamma$ (see \eqref{def:B}) and again \eqref{def:eta0}, this implies
\begin{align*}
\abs{B \widehat{f^H_{++}}(t,k,\eta)} & \lesssim \epsilon \jap{k_0,\eta_0}^{-\beta/4 + 4} \left[\frac{\jap{k_0,\eta_0}^{R}}{\epsilon} \right]^{K_m^{1/3}(K_m')^{-1/3} + \frac{2}{3}\nu(0)K^{1/3}(K_m')^{-1/3}} e^{-\frac{1}{10}\abs{\eta-\eta_0}} e^{-\frac{1}{2}k_0^{-1}\abs{k}}. 
\end{align*}
Therefore, recalling that $\beta = \sigma- R$, by choosing $\sigma$ large relative only to the constants $R,K,\nu(0),K_m$, and $K_m'$, we have the stated result \eqref{ineq:BfH} (without the $v$ moment; see above) using that $\eta_0 \gtrsim \epsilon^{-1}$ by \eqref{def:eta0} (also \eqref{def:k0}).

The high norm estimate \eqref{ineq:AfH} is similar; by Lemma \ref{lem:distup} (recall Lemma \ref{lem:GrowthFact}) and \eqref{def:AG} we have, 
\begin{align*}
\abs{A \widehat{f^H_{++}}(t,k,\eta)} & \lesssim \epsilon \jap{k,\eta}^{\beta-\sigma} e^{\mu(t)(K\epsilon)^{1/3} \jap{k,\eta}^{1/3}} G(t,k,\eta) e^{3(K_m \epsilon)^{1/3}\eta_0^{1/3}}e^{-\frac{1}{8}\abs{\eta-\eta_0}} e^{-k_0^{-1}\abs{k}}. 
\end{align*}
By Lemma \ref{lem:growthw}, \eqref{def:k0}, \eqref{ineq:IncExp}, and \eqref{ineq:SobExp}, we have 
\begin{align*}
\abs{A \widehat{f^H_{++}}(t,k,\eta)} & \lesssim_{\sigma} \epsilon\jap{k_0,\eta_0}^{\beta-\sigma} e^{2\left(\mu(t)+r\right)(K\epsilon)^{1/3}\eta_0^{1/3}} e^{3(K_m \epsilon)^{1/3}\eta_0^{1/3}}e^{-\frac{1}{10}\abs{\eta-\eta_0}} e^{-\frac{1}{2}k_0^{-1}\abs{k}}. 
\end{align*}
Therefore, by \eqref{def:eta0} there holds 
\begin{align*}
\abs{A \widehat{f^H_{++}}(t,k,\eta)} & \lesssim_{\sigma} \epsilon\jap{k_0,\eta_0}^{\beta-\sigma} \left[\frac{\jap{k_0,\eta_0}^{R}}{\epsilon}\right]^{K_m^{1/3}(K_m')^{-1/3} + \frac{2}{3}\left(2\mu_{\infty}+r\right)K^{1/3}(K_m')^{-1/3}} e^{-\frac{1}{10}\abs{\eta-\eta_0}} e^{-\frac{1}{2}k_0^{-1}\abs{k}}.
\end{align*}
By integration, this implies the stated result \eqref{ineq:AfH}.  
\end{proof} 

Next, we want to estimate the error $\mathcal{E}$ arising in \eqref{def:g}. We divide as follows 
\begin{align}
\mathcal{E} & = E^L(z+tv)(\partial_v - t\partial_z)f^L + E^L(z+tv)(\partial_v - t\partial_z)f^H \nonumber \\ & \quad + E^L(z+tv)  \partial_v f^0 +  E^H(z+tv) \cdot (\grad_v-t\grad_x) f^H \nonumber \\ 
& = \mathcal{E}_{LL} + \mathcal{E}_{LH} + \mathcal{E}_{L0} + \mathcal{E}_{HH}. \label{def:Esubdiv}
\end{align}

The following lemmas are straightforward consequences of the definition of $t_{in}$, the product rule-type estimates in Corollary \ref{cor:prod} and Lemma \ref{lem:Abckgr}, together with the estimates on $f^H$ stated in Lemma \ref{lem:AfHBfH}. Hence, we omit the proofs for the sake of brevity. 

\begin{lemma} 
For $t \in (t_{in},t_\star)$, there holds for any $\alpha \in \Real$, 
\begin{subequations} \label{ineq:ELow} 
\begin{align}
\norm{\jap{\grad}^\alpha A\mathcal{E}_{LL}}_2 & \lesssim_{\alpha,\sigma} \epsilon^2 e^{-\frac{1}{2}\epsilon^{-q}} \\ 
\norm{\jap{\grad}^\alpha B\mathcal{E}_{LL}}_2 & \lesssim_{\alpha,\sigma} \epsilon^2 e^{-\frac{1}{2}\epsilon^{-q}} \\ 
\norm{\jap{\grad}^\alpha A\mathcal{E}_{L0}}_2 & \lesssim_{\alpha,\sigma} \delta \epsilon  e^{-\frac{1}{2}\epsilon^{-q}} \\ 
\norm{\jap{\grad}^\alpha B\mathcal{E}_{L0}}_2 & \lesssim_{\alpha,\sigma} \delta \epsilon e^{-\frac{1}{2}\epsilon^{-q}}.
\end{align}
\end{subequations}
\end{lemma}

\begin{lemma} 
For $t \in (t_{in},t_\star)$, there holds 
\begin{subequations} \label{ineq:ELH} 
\begin{align}
\norm{A\mathcal{E}_{LH}}_2 & \lesssim \epsilon e^{-\frac{1}{2}\epsilon^{-q}} \epsilon^{-a+1} \eta_0^{R(a-1)+1}, \\ 
\norm{B\mathcal{E}_{LH}}_2 & \lesssim \epsilon e^{-\frac{1}{2}\epsilon^{-q}} \eta_0^{-\sigma/5 + 1}. 
\end{align}
\end{subequations} 
\end{lemma} 

\begin{lemma} 
For $t \in (t_{in},t_\star)$, there holds 
\begin{align*}
\norm{A\mathcal{E}_{HH}}_2 & \lesssim \epsilon^{1-a} \eta_0^{a(R-1)+2 - \sigma/5} \\ 
\norm{B\mathcal{E}_{HH}}_2 & \lesssim \eta_0^{-2\sigma/5+2}. 
\end{align*}
\end{lemma} 

\section{Stability of the approximate solution} \label{sec:g}
Now we are ready to make our estimates on $g$. We will use a variation of the method employed in \cite{BMM13}, however, the norms here are significantly more complicated and we need to localize  $g$ to high frequencies in the sense that $g$ must be very small when measured in lower norms. 
For convenience, we will make use of the following short-hand to denote small adjustments to $A$ and $B$: 
\begin{align*}
A_{\alpha}(t,\grad) = \jap{\grad}^\alpha A(t,\grad), \quad\quad B_{\alpha}(t,\grad) = \jap{\grad}^\alpha B(t,\grad). 
\end{align*} 
Moreover, for notational convenience, for the duration of \S\ref{sec:g}, we will use 
\begin{align*}
\rho & := \rho_g, \quad\quad E := E_g.
\end{align*}
Let $T_\star$ be the largest time $T_\star \leq \eta_0 = t_\star$ such that the following holds for all  $t_{in} < t < T_\star$ (denoting $I = [t_{in},T_\star]$):    
\begin{subequations}\label{def:bootg} 
\begin{align}
 \sup_{t \in I} \left(\jap{t}^{-5/2}\norm{\jap{v}A_3g(t)}_{L^2}\right) = & \leq 8\epsilon^2, \label{ineq:Hig} \\ 
\norm{\jap{v} A_2 \rho}_{L^2_t(I;L^2)} & \leq 8\epsilon^2, \label{ineq:HiRho} \\ 
\sup_{t \in I} \norm{\jap{v}Ag(t)}_{L^2} & \leq 8\epsilon^2, \label{ineq:Midg} \\ 
 \norm{\jap{v}B_2 \rho}_{L^2_t(I;L^2)} & \leq 8\epsilon^{\sigma/5}, \label{ineq:LoRho} \\ 
\sup_{t \in I}\norm{\jap{v}Bg(t)}_{L^2} & \leq 8\epsilon^{\sigma/5}. \label{ineq:Log}
\end{align}
\end{subequations}
By well-posedness, we have $t_{in} < T_\star$ and in since the quantities on the left-hand side all take values continuously in time, it suffices to prove the following proposition. 
Proposition \ref{prop:stabg} then follows. 
\begin{proposition} \label{prop:boot}
For $K$ chosen large relative to a universal constant, $\sigma$ chosen large relative to $R$ and $K$, $\epsilon$ chosen sufficiently small relative to $K$, $R$, and $\sigma$, and $T_\star \leq \eta_0$, the inequalities \eqref{def:bootg} hold with the `8' replaced with a `4' and as a result, $T_\star = \eta_0$.
\end{proposition} 

\subsection{$L^2_t$ high norm estimate on the density} \label{sec:dense} 
In this section we improve the estimate \eqref{ineq:HiRho}.
This is the key estimate in the proof of Proposition \ref{prop:boot}. 
Define:
\begin{align}
\widehat{\rho_0}(t,k) & = -\frac{1}{2\pi}\sum_{\ell \in \Integers_\ast} \int_{t_{in}}^t \hat{\rho}(\tau,\ell) \widehat{W}(\ell) \ell k(t-\tau) \widehat{f^E}(\tau,k-\ell,kt-\ell \tau) d\tau \nonumber \\ 
& \quad - \frac{1}{2\pi}\sum_{\ell \in \Integers_\ast}\int_{t_{in}}^t \widehat{\rho^E}(\tau,\ell) \widehat{W}(\ell) \ell k(t-\tau) \widehat{g}(\tau,k-\ell,kt-\ell \tau) d\tau  \nonumber \\ 
& \quad - \frac{1}{2\pi}\sum_{\ell \in \Integers_\ast} \int_{t_{in}}^t \hat{\rho}(\tau,\ell) \widehat{W}(\ell) \ell k(t-\tau) \widehat{g}(\tau,k-\ell,kt-\ell \tau) d\tau - \int_{t_{in}}^t \widehat{\mathcal{E}}(\tau,k,kt)d\tau \nonumber \\ 
& = L_R + L_T + NL + E. \label{def:rho0}
\end{align}
From Corollary \ref{cor:linsolve}, the solution of the linearized Vlasov equations is given by the following, 
\begin{align}
\hat{\rho}(t,k) & =  \widehat{\rho_0}(t,k) + \int_{t_{in}}^t R(t-\tau,k)\widehat{\rho_0}(\tau,k) d\tau. \label{def:rhog}
\end{align}
By the estimate on $R$ in Corollary \ref{cor:linsolve} (see Lemma \ref{lem:basicvolt}), we have the following for any $c > 0$ and corresponding $C(c,\sigma)$ and $C'(c,\sigma)$ depending only $c$ and $\sigma$, 
\begin{align}
\norm{A_2\rho}_{L^2_t L^2} & \leq  \norm{A_2\rho_0}_{L^2_t L^2} + \norm{\int_{t_{in}}^t A_2(t,k,kt) R(t-\tau,k)\widehat{\rho_0}(\tau,k) d\tau}_{L^2_t L^2_k} \nonumber \\ 
& \leq \norm{A_2\rho_0}_{L^2_t L^2} + C(c,\sigma)\norm{\int_{t_{in}}^t e^{c\abs{k(t-\tau)}} R(t-\tau,k) A_2(\tau,k,k\tau)\widehat{\rho_0}(\tau,k) d\tau}_{L^2_t L^2_k} \nonumber \\ 
& \leq \left(1+ \sqrt{\delta}C'(c,\sigma)\right)\norm{A_2\rho_0}_{L^2_t L^2}. \label{ineq:ArhoLt}
\end{align}
Hence, for $\delta$ small depending on $\sigma$, it suffices to control $\rho_0$. 

\subsubsection{Linear reaction term $L_R$} \label{sec:LRg}
The primary difficulty is the ``reaction term'' $L_R$ which is the interaction of $g$ and $f^E$. 
This term is naturally divided into 
\begin{align}
L_R & = -\frac{1}{2\pi}\sum_{\ell \in \Integers_\ast} \int_{t_{in}}^t \hat{\rho}(\tau,\ell) \widehat{W}(\ell) \ell k(t-\tau) \widehat{f^L}(\tau,k-\ell,kt-\ell \tau) d\tau \\ & \quad - \frac{1}{2\pi}\sum_{\ell \in \Integers_\ast} \int_{t_{in}}^t \hat{\rho}(\tau,\ell) \widehat{W}(\ell) \ell k(t-\tau) \widehat{f^H}(\tau,k-\ell,kt-\ell \tau) d\tau  \nonumber \\ 
& = L_{R;L} + L_{R;H}. \label{def:LRrhoEst}
\end{align}

\textbf{Estimate of $L_{R;L}$}: \\
\noindent
We estimate $L_{R;L}$ first. We sub-divide further to remove the leading order contribution and the lower order terms (this is done to be precise in dependence on $\sigma$), 
\begin{align}
\norm{A_2(t)L_{R;L}}_{L^2_t L^2_k}^2 & \approx \sum_{k \in \Integers_\ast}\int_{t_{in}}^{T_\star} \left[ \epsilon \sum_{\ell = k\pm 1} \int_{t_{in}}^t A_2(t,k,kt) \hat{\rho}(\tau,\ell) \widehat{W}(\ell) \ell k(t-\tau) e^{-\abs{kt-\ell \tau}}  d\tau \right]^2 dt \nonumber \\ 
& \lesssim \sum_{k \in \Integers_\ast}\int_{t_{in}}^{T_\star} \left[ \epsilon \sum_{\ell = k\pm 1} \int_{t_{in}}^t \frac{G(t,k,kt)}{G(\tau,\ell,\ell \tau)}A_2(\tau,\ell,\ell \tau) \hat{\rho}(\tau,\ell) \widehat{W}(\ell) \ell k(t-\tau) e^{-\abs{kt-\ell \tau}}  d\tau \right]^2 dt \nonumber \\ 
& \quad + \sum_{k}\int_{t_{in}}^{T_\star} \left[ \epsilon \sum_{\ell = k\pm 1} \int_{t_{in}}^t \left[\jap{k, kt}^{\beta+2} e^{\mu(\tau)(K\epsilon)^{1/3}\jap{k, kt}^{1/3}} - \jap{\ell, \ell \tau}^{\beta+2} e^{\mu(\tau) (K\epsilon)^{1/3}\jap{\ell, \ell \tau}^{1/3}}\right] \right. \nonumber \\ 
& \left. \quad\quad\quad \times e^{(\mu(t) - \mu(\tau))(K\epsilon)^{1/3}\jap{k,kt}^{1/3}} G(t,k,kt) \hat{\rho}(\tau,\ell) \widehat{W}(\ell) \ell k(t-\tau) e^{-\abs{kt-\ell \tau}}  d\tau \right]^2 dt \nonumber \\ 
& = L_{R;L,1} + L_{R;L,2}, \label{def:RL1L2} 
\end{align}
where it is important to note that the implicit constant does not depend on $\mu$ or $\sigma$.

By Schur's test, $L_{R;L,1}$ in \eqref{def:RL1L2} is estimated via
\begin{align}
L_{R;L,1} & \lesssim \norm{A_2\rho}_{L^2_t L^2}^2 \left(\sup_{k \in \Integers_\ast} \sup_{t \in (t_{in},T_\star)} \sum_{\ell = k\pm 1} \int_{t_{in}}^{t} K_{w}(t,\tau,k,\ell) d\tau \right) \nonumber \\ & \quad\quad \times \left(\sup_{\ell \in \Integers_\ast} \sup_{\tau \in (t_{in},T_\star)} \sum_{k = \ell \pm 1} \int_{\tau}^{T_\star} K_{w}(t,\tau,k,\ell) d t \right), \label{ineq:LR1Schur}
\end{align}
 where, 
\begin{align}
K_{w}(t,\tau,k,\ell) = \epsilon \frac{G(t,k,kt)}{G(\tau,\ell,\ell \tau)} \abs{ \widehat{W}(\ell) \ell k(t-\tau)} e^{-\abs{kt-\ell \tau}}. \label{def:Kw}
\end{align}
Controlling this term then reduces to controlling the corresponding integral kernels. 
\begin{lemma} \label{lem:KernelLRL1} 
For some universal constant $C_{LR1}$ and constant $C'= C'(\sigma)$, there holds 
\begin{subequations} \label{ineq:kernelLRL1}
\begin{align}
\sup_{k \in \Integers_\ast} \sup_{t \in (t_{in},T_\star)} \sum_{\ell = k\pm 1} \int_{t_{in}}^{t} K_{w}(t,\tau,k,\ell) d\tau & \leq \frac{C_{LR1}}{K}  + C'\epsilon \label{ineq:kernelLRL1_a}\\ 
\sup_{\ell \in \Integers_\ast} \sup_{\tau \in (t_{in},T_\star)} \sum_{k = \ell \pm 1} \int_{\tau}^{T_\star} K_{w}(t,\tau,k,\ell) d t   & \leq \frac{C_{LR1}}{K}  + C'\epsilon. \label{ineq:kernelLRL1_b}
\end{align}
\end{subequations} 
\end{lemma} 
\begin{proof} 
We will just prove \eqref{ineq:kernelLRL1_a}; \eqref{ineq:kernelLRL1_b} is analogous and is omitted for brevity (see e.g. \cite{BMM13} for what kind of small modifications are necessary). 
For simplicity, we will restrict ourselves to positive $k$, the treatment for negative $k$ is completely analogous.  
Given a $k$ and $t$, we will divide the integral into two contributions, resonant and non-resonant:  
\begin{align*}
\sum_{\ell = k\pm 1} \int_{t_{in}}^{t} K_{w}(t,\tau,k,\ell) d\tau & = \sum_{\ell = k\pm 1} \int_{t_{in}}^{t}\left(\mathbf{1}_{\abs{kt-\ell \tau} < t/2} + \mathbf{1}_{\abs{kt-\ell \tau} \geq t/2}\right) K_{w}(t,\tau,k,\ell) d\tau  \\ 
& = \mathcal{I}_R + \mathcal{I}_{NR}. 
\end{align*}
Recalling \eqref{def:Kw} and \eqref{def:AG}, and applying Lemma \ref{lem:Gcomp} (using that $G$ is monotone decreasing in time so that $G(t,k,kt) \leq G(\tau,k,kt)$), followed by \eqref{ineq:IncExp} and the non-resonant assumption: 
\begin{align}
\mathcal{I}_{NR} & \lesssim  \epsilon \sum_{\ell = k\pm 1} \int_{t_{in}}^t \mathbf{1}_{\abs{kt-\ell \tau} \geq t/2} \frac{\jap{\tau}}{\abs{\ell}} e^{\tilde{r}(K\epsilon)^{1/3}\jap{k-\ell,kt-\ell \tau}^{1/3}} \jap{kt-\ell \tau} e^{-\abs{kt-\ell \tau}} d\tau \nonumber \\ 
& \lesssim \epsilon \sum_{\ell = k\pm 1} \int_{t_{in}}^t \mathbf{1}_{\abs{kt-\ell \tau} \geq t/2} \jap{\tau} e^{-\frac{3}{4}\abs{kt-\ell \tau}} d\tau \nonumber \\ 
& \lesssim \epsilon.\label{ineq:KernelLRL1_NonRes}
\end{align}
This is consistent with Lemma \ref{lem:KernelLRL1} (in fact, we could make this term much smaller). 

Turn next to the more subtle resonant contributions, $\mathcal{I}_R$. 
First, notice that only $\ell = k+1$ is present in this term. 
Hence, from \eqref{def:AG}, there holds 
\begin{align}
\frac{G(t,k,kt)}{G(\tau,\ell,\ell \tau)} & \leq \frac{w(\tau,(k+1)\tau)}{w(t,k)}e^{r(K\epsilon)^{1/3}\jap{kt}^{1/3} - r(K\epsilon)^{1/3}\jap{(k+1)\tau}^{1/3}} \nonumber \\ 
& \quad + w(\tau,\tau(k+1))e^{r(K\epsilon)^{1/3}\jap{k}^{1/3} - r(K\epsilon)^{1/3}\jap{\tau(k+1)}^{1/3}} \nonumber \\ 
& = T_1 + T_2, \label{eq:RatGEchoPf} 
\end{align}
which induces a corresponding decomposition of $\mathcal{I}_R$: 
\begin{align*}
\mathcal{I}_{R} & = \epsilon\sum_{\ell = k\pm 1} \int_{t_{in}}^{t}\mathbf{1}_{\abs{kt-\ell \tau} < t/2}\left(T_1 + T_2\right)\abs{\widehat{W}(\ell) \ell k(t-\tau)}e^{-\abs{kt - \ell \tau}} d\tau \\ 
& = \mathcal{I}_{R;1} + \mathcal{I}_{R;2}. 
\end{align*} 
For $T_2$, there holds on the support of the integrand in $\mathcal{I}_{R;2}$, for some $c > 0$, by Lemma \ref{lem:growthw}, 
\begin{align*}
T_2 &  \lesssim e^{r(K\epsilon)^{1/3}\jap{k}^{1/3} - \frac{1}{2}r(K\epsilon)^{1/3}\jap{\tau(k+1)}^{1/3}} \lesssim e^{-c \tau^{1/3}(K\epsilon)^{1/3} \abs{k}^{1/3}}, 
\end{align*}
It hence follows (using \eqref{ineq:SobExp}), 
\begin{align*}
\mathcal{I}_{R;2} & \lesssim \epsilon\sum_{\ell = k\pm 1} \int_{t_{in}}^{t}\mathbf{1}_{\abs{kt-\ell \tau} < t/2} \frac{\jap{\tau}}{\abs{k}} e^{-c\tau^{1/3} (K\epsilon)^{1/3}\abs{k}^{1/3}}e^{-\abs{kt - \ell \tau}} d\tau \\ 
& \lesssim \frac{1}{K}\sum_{\ell = k\pm 1} \int_{t_{in}}^{t}\mathbf{1}_{\abs{kt-\ell \tau} < t/2} \frac{1}{\abs{k}^2} e^{-c\tau^{1/3} (K\epsilon)^{1/3}\abs{k}^{1/3}}e^{-\abs{kt - \ell \tau}} d\tau \\  
& \lesssim \frac{1}{K}. 
\end{align*}
The constant is universal, and hence this is consistent with the first term in \eqref{ineq:kernelLRL1_a}. 

Consider next the first term, $T_1$,  in \eqref{eq:RatGEchoPf}. 
For $T_1$ we have, recalling $(K\epsilon)^{1/3} \leq 1$,  
\begin{align*}
T_1 & \leq \frac{w(\tau,(k+1)\tau)}{w(t,tk)}e^{r(K\epsilon)^{1/3}\jap{kt - (k+1)\tau}^{1/3}} \leq \frac{w(\tau,(k+1)\tau)}{w(t,kt)}e^{r\jap{kt - (k+1)\tau}^{1/3}}, 
\end{align*}
where the implicit constant does not depend on $K$ or $\epsilon$ (recall $r$ is a fixed, universal constant). 
Turn next to $\mathcal{I}_{R;1}$.
By \eqref{ineq:IncExp} and \eqref{ineq:SobExp}, we hence have the following, with an implicit constant that does not depend on $\epsilon$ or $K$: 
\begin{align*}
\mathcal{I}_{R;1} & \lesssim \epsilon\int_{t_{in}}^t \mathbf{1}_{\abs{kt-(k+1)\tau} < t/2}  \frac{w(\tau,(k+1)\tau)}{w(t,kt)} \frac{\jap{\tau}}{(k+1)} e^{-\frac{3}{4}\abs{kt-(k+1)\tau}} d\tau. 
\end{align*}
Next, we want to separate out the case that $\epsilon K t \leq k^2$ and vice-versa. In the former case we do not need (and cannot use) the presence of the ratio of $w$'s. 
Instead, using directly Lemmas \ref{lem:tildewcomp} and \ref{lem:wapproxtildw} as well as the restriction on $t$ and $k$, we have (with constants independent of $K$),  
\begin{align*}
\mathcal{I}_{R;1} \mathbf{1}_{\epsilon K t \leq k^2} & \lesssim \mathbf{1}_{\epsilon K t \leq k^2}  \epsilon\int_{t_{in}}^t \mathbf{1}_{\abs{kt-(k+1)\tau} < t/2} \frac{\jap{\tau}}{(k+1)} e^{\tilde{r}(K\epsilon)^{1/3}\jap{kt- (k+1)\tau}^{1/3}} \jap{kt - (k+1)\tau} e^{-\abs{kt-(k+1)\tau}} d\tau \\ 
& \lesssim \mathbf{1}_{\epsilon M t \leq k^2}  \frac{1}{K} \int_{t_{in}}^t \mathbf{1}_{\abs{kt-(k+1)\tau} < t/2}  (k+1) e^{\tilde{r}(K\epsilon)^{1/3}\jap{kt- (k+1)\tau}^{1/3}} e^{-\frac{3}{4}\abs{kt-(k+1)\tau}} d\tau. 
\end{align*}
Therefore, by \eqref{ineq:IncExp}, we have (still with constants independent of $K$), 
\begin{align*}
\mathcal{I}_{R;1}\mathbf{1}_{\epsilon K t \leq k^2} & \lesssim \frac{1}{K} \int_{t_{in}}^t \mathbf{1}_{\abs{kt-(k+1)\tau} < t/2}  (k+1)  e^{-\frac{1}{2}\abs{kt-(k+1)\tau}} d\tau  \lesssim \frac{1}{K}, 
\end{align*}
which is consistent with Lemma \ref{lem:KernelLRL1}. 
Turn next to the case that $\epsilon K t > k^2$. Here we begin by applying Lemmas \ref{lem:tildewcomp} and \ref{lem:wapproxtildw}, followed by \eqref{ineq:IncExp}, to give, with constants independent of $K$, 
\begin{align*}
\mathcal{I}_{R;1} \mathbf{1}_{\epsilon K t > k^2} & \lesssim \mathbf{1}_{\epsilon K t > k^2}  \epsilon\int_{t_{in}}^t \mathbf{1}_{\abs{kt-(k+1)\tau} < t/2} \frac{w(\tau,kt)}{w(t,kt)} \frac{\jap{\tau}}{k+1} e^{-\frac{1}{2}\abs{kt-(k+1)\tau}} d\tau.
\end{align*}
Here we use the key property of $w$ which motivates its design: Lemma \ref{lem:resw} implies
\begin{align}
\mathcal{I}_{R;1} \mathbf{1}_{\epsilon K t > k^2} & \lesssim \mathbf{1}_{\epsilon K t > k^2}  \int_{t_{in}}^t \mathbf{1}_{\abs{kt-(k+1)\tau} < t/2} \frac{(k+1)}{K} e^{-\frac{1}{4}\abs{kt-(k+1)\tau}} d\tau 
 \lesssim \frac{1}{K}, \label{ineq:usew} 
\end{align}
with implicit constant independent of $\epsilon$ and $K$.  
This completes the proof of \eqref{ineq:kernelLRL1_a}; as mentioned above \eqref{ineq:kernelLRL1_b} follows similarly. 
\end{proof} 

Using Lemma \ref{lem:KernelLRL1} in \eqref{ineq:LR1Schur} gives (with implicit constant independent of $K$), 
\begin{align*}
L_{R;L,1} \leq \norm{A_2\rho}_{L^2_t(I;L^2_k)}^2\left(\frac{C_{LR1}}{K} + C'\epsilon\right). 
\end{align*}
Therefore, for $K \gg C_{LR1}$ and $\epsilon$ sufficiently small, there holds 
\begin{align*}
L_{R;L,1} \leq \frac{1}{100}\norm{A_2\rho}_{L^2_t(I;L^2_k)}^2, 
\end{align*}
which is sufficient to control the contribution of $L_{R;L,1}$ in \eqref{ineq:ArhoLt} (after \eqref{def:LRrhoEst} and \eqref{def:RL1L2}).   

Turn next to $L_{R;L,2}$ in \eqref{def:RL1L2}, which is in some sense lower order due to the commutator. 
First, 
\begin{align}
\abs{\jap{k, k t}^{\beta+2} e^{\mu(\tau)(K\epsilon)^{1/3}\jap{k, kt}^{1/3}} - \jap{\ell, \ell \tau}^{\beta+2} e^{\mu(\tau) (K\epsilon)^{1/3}\jap{\ell, \ell \tau}^{1/3}}} & \nonumber \\ & \hspace{-6cm} \leq \abs{e^{\mu(\tau)(K\epsilon)^{1/3}\jap{k, kt}^{1/3}} - e^{\mu(\tau)(K\epsilon)^{1/3}\jap{\ell, \ell \tau}^{1/3}}}\jap{k,kt}^{\beta+2} \nonumber \\ & \hspace{-6cm}\quad + \abs{\jap{k,kt}^{\beta+2} - \jap{\ell,\ell \tau}^{\beta+2}}e^{\mu(\tau)(K\epsilon)^{1/3}\jap{\ell, \ell \tau}^{1/3}} \nonumber \\ 
& \hspace{-6cm} = T_G + T_S. \label{ineq:TGTS}
\end{align}
For $T_G$, by $\abs{e^x - 1} \leq xe^x$ and the mean-value theorem,  
\begin{align} 
T_G & \lesssim \mu(\tau) (K\epsilon)^{1/3} \frac{\jap{k-\ell, kt-\ell \tau}^{5/3}}{\jap{\ell,\ell \tau}^{2/3}} \jap{k,kt}^{\beta+2} e^{\mu(\tau)(K\epsilon)^{1/3}\jap{\ell, \ell \tau}^{1/3}} e^{\mu(\tau)(K\epsilon)^{1/3}\jap{k-\ell, kt-\ell \tau}^{1/3}} \nonumber \\ 
 & \lesssim_\beta  \frac{(K\epsilon)^{1/3}}{\jap{\ell,\ell \tau}^{2/3}} \jap{\ell,\ell \tau}^{\beta+2} e^{\mu(\tau)(K\epsilon)^{1/3}\jap{\ell, \ell \tau}^{1/3}} \jap{k-\ell, kt-\ell \tau}^{11/3+\beta} e^{\mu(\tau)(K\epsilon)^{1/3}\jap{k-\ell, kt-\ell \tau}^{1/3}}. \label{ineq:TGest} 
\end{align}
Similarly, we have 
\begin{align}
T_S \lesssim \jap{\ell, \ell \tau}^{\beta+1} e^{\mu(\tau)(K\epsilon)^{1/3}\jap{\ell, \ell \tau}^{1/3}} \jap{k-\ell, kt-\ell \tau}^{4+\beta}. \label{ineq:TSest} 
\end{align} 
Putting \eqref{ineq:TGTS}, \eqref{ineq:TGest}, and \eqref{ineq:TSest} together with \eqref{ineq:SobExp} and \eqref{ineq:IncExp} (recall \eqref{def:AG}), we get 
\begin{align}
L_{R;L,2} & \lesssim \sum_{k \in \Integers_\ast}\int_{t_{in}}^{T_\star} \left[ \epsilon \sum_{\ell = k\pm 1} \int_{t_{in}}^t \frac{1}{\jap{\ell, \ell \tau}} e^{\left(\mu(t) - \mu(\tau)\right) (K\epsilon)^{1/3} \jap{k,kt}^{1/3}} \frac{\jap{\tau}}{\abs{\ell}}\abs{A(\tau,\ell,\ell \tau) \hat{\rho}(\tau,\ell)} e^{-\frac{3}{4}\abs{kt-\ell \tau}} d\tau \right]^2 dt \nonumber \\ 
& \quad + \sum_{k \in \Integers_\ast}\int_{t_{in}}^{T_\star} \left[ \epsilon \sum_{\ell = k\pm 1} \int_{t_{in}}^t \frac{(K\epsilon)^{1/3}}{\jap{\ell,\ell \tau}^{2/3}} e^{\left(\mu(t) - \mu(\tau)\right) (K\epsilon)^{1/3} \jap{k,kt}^{1/3}} \frac{\jap{\tau}}{\abs{\ell}}\abs{A(\tau,\ell,\ell \tau) \hat{\rho}(\tau,\ell)} e^{-\frac{3}{4}\abs{kt-\ell \tau}} d\tau \right]^2 dt \nonumber \\
& \lesssim L_{R;L,2,S} + L_{R;L,2,G}.\label{ineq:LRL2}
\end{align}
By Schur's test, 
\begin{align}
L_{R;L,2,S} & \lesssim  \epsilon^2 \norm{A\rho}_{L^2_t(I;L^2)}^2 \left(\sup_{k \in \Integers_\ast} \sup_{t \in (t_{in},T_\star)} \sum_{\ell = k\pm 1} \int_{t_{in}}^{t} \frac{e^{-\frac{3}{4}\abs{kt-\ell \tau}}}{\abs{\ell}^2}  d\tau \right) \nonumber \\ & \quad\quad \times \left(\sup_{\ell \in \Integers_\ast} \sup_{\tau \in (t_{in},T_\star)} \sum_{k = \ell \pm 1} \int_{\tau}^{T_\star} \frac{e^{-\frac{3}{4}\abs{kt-\ell \tau}}}{\abs{\ell}^2} dt \right) \nonumber \\ 
& \lesssim \epsilon^2 \norm{A\rho}_{L^2_t(I;L^2)}^2. \label{ineq:LR2S}
\end{align}
This suffices to treat $L_{R;L,2,S}$ by choosing $\epsilon$ small (relative to $\sigma$ as the constant depends on $\beta$). 

The Gevrey term, $L_{R;L,2,G}$ in \eqref{ineq:LRL2}, is a little more complicated. 
By Schur's test, 
\begin{align}
L_{R;L,2,G} & \lesssim \epsilon^2 \norm{A\rho}_{L^2_t(I;L^2)}^2 \left(\sup_{t \in (t_{in},T_\ast)} \sup_{k \in \Integers_\ast }\int_{t_{in}}^t \sum_{\ell \in \Integers_\ast} K_G(t,\tau,\ell,k) d\tau \right) \nonumber \\ & \quad\quad \times \left(\sup_{\tau \in (t_{in},T_\ast)} \sup_{\ell \in \Integers_\ast}\int_{\tau}^{T_\ast} \sum_{k \in \Integers_\ast} K_G(t,\tau,k,\ell) dt \right), \label{ineq:KRL2G}
\end{align}
with the kernel 
\begin{align*}
K_G = \frac{(K\epsilon)^{1/3} \jap{\tau}^{1/3}}{\abs{\ell}^{5/3}} e^{\left(\mu(t) - \mu(\tau)\right) (K\epsilon)^{1/3} \jap{k, kt}^{1/3}} e^{-\frac{3}{4}\abs{kt-\ell \tau}}. 
\end{align*}
This term is then completed once we prove the following. 
\begin{lemma} \label{lem:KernelKGLRL2} 
For $b \in (0,\frac{1}{6})$ (recall \eqref{def:mu}), there exists some constant $C'= C'(\sigma)$, such that
\begin{subequations} \label{ineq:kernelKGLR2}
\begin{align}
\sup_{k \in \Integers_\ast} \sup_{t \in (t_{in},T_\star)} \sum_{\ell = k\pm 1} \int_{t_{in}}^{t} K_{G}(t,\tau,k,\ell) d\tau & \leq  C'(K\epsilon)^{-\frac{b}{1-3b}}\label{ineq:kernelKGLR2_a} \\  
\sup_{\ell \in \Integers_\ast} \sup_{\tau \in (t_{in},T_\star)} \sum_{k = \ell \pm 1} \int_{\tau}^{T_\star} K_{G}(t,\tau,k,\ell) dt   & \leq  C'(K\epsilon)^{-\frac{b}{1-3b}}. \label{ineq:kernelKGLR2_b}
\end{align}
\end{subequations} 
\end{lemma} 
\begin{proof} 
As in the proof of Lemma \ref{lem:KernelLRL1}, we will prove \eqref{ineq:kernelKGLR2_a}; \eqref{ineq:kernelKGLR2_b} is analogous and is omitted for brevity. 
Let $k,t$ be fixed. 
Divide the integral into resonant and non-resonant contributions (as in Lemma \ref{lem:KernelLRL1}), 
\begin{align*}
\sum_{\ell = k\pm 1} \int_{t_{in}}^{t} K_{G}(t,\tau,k,\ell) d\tau & = \sum_{\ell = k\pm 1} \int_{t_{in}}^{t}\left(\mathbf{1}_{\abs{kt-\ell \tau} < t/2} + \mathbf{1}_{\abs{kt-\ell \tau} \geq t/2}\right) K_{G}(t,\tau,k,\ell) d\tau  = \mathcal{I}_R + \mathcal{I}_{NR}. 
\end{align*}
The non-resonant term is treated as above in \eqref{ineq:KernelLRL1_NonRes} and is hence omitted for brevity, 
\begin{align*}
\mathcal{I}_{NR} & \lesssim \epsilon. 
\end{align*}
Hence, turn to the resonant term. 
As above, notice that the only term present is the $\ell = k+1$ contribution. 
Next, on the support of the integrand, the following holds, analogously to \eqref{ineq:rtrtau}: 
\begin{align*}
\mu(t) - \mu(\tau) 
\lesssim -\frac{t_{in}^b}{t^{b}(k+1)}, 
\end{align*}
which implies, for some constant $c > 0$,  
\begin{align*}
\mathcal{I}_R & \lesssim \int_{t_{in}}^{t} \mathbf{1}_{\abs{kt-\ell \tau} < t/2} \frac{(K\epsilon)^{1/3} \jap{\tau}^{1/3}}{\abs{k+1}^{5/3}} e^{-c(K\epsilon)^{1/3} t_{in}^b \abs{k+1}^{-2/3} t^{1/3-b}} e^{-\frac{3}{4}\abs{kt-\ell \tau}}. 
\end{align*}
Therefore, 
\begin{align*}
\mathcal{I}_R & \lesssim \int_{t_{in}}^{t} \mathbf{1}_{\abs{kt-\ell \tau} < t/2} \frac{(K\epsilon)^{1/3} \jap{\tau}^{1/3}}{\abs{k+1}^{5/3}} \frac{\abs{k+1}^{\frac{2}{3(1-3b)}} }{\tau^{1/3}(K\epsilon)^{\frac{1}{3(1-3b)}} t_{in}^{\frac{b}{1-3b}}} e^{-\frac{3}{4}\abs{kt-\ell \tau}} \lesssim (K\epsilon)^{-\frac{b}{(1-3b)}}. 
\end{align*}
This  completes the proof of Lemma \ref{lem:KernelKGLRL2}. 
\end{proof} 

Applying Lemma \ref{lem:KernelKGLRL2} to \eqref{ineq:KRL2G} and choosing $\epsilon$ sufficiently small implies,  
\begin{align*}
L_{R;L,2,G} & \leq \epsilon \norm{A_2\rho}_{L^2L^2}^2. 
\end{align*}
Together with \eqref{ineq:LRL2} and \eqref{ineq:LR2S}, this completes the treatment of the $L_{R;L,2}$ contributions in \eqref{ineq:ArhoLt} (from \eqref{def:LRrhoEst} and \eqref{def:RL1L2}). 
Accordingly, this completes the treatment of $L_{R;L}$ in \eqref{def:LRrhoEst}. 

\textbf{Estimate of $L_{R;H}$}: \\
\noindent
Turn next to $L_{R;H}$ in \eqref{def:LRrhoEst}. 
For this term, we apply Lemma \ref{lem:GenCEst}: 
\begin{align*}
\norm{A_2 L_{R;H}}_{L^2_t(I;L^2)} & \lesssim \norm{B \rho}_{L^2_tL^2} \sup_{\tau \in (t_{in},T_\star)}\left( \norm{\jap{v}\jap{\grad}^3 Af^H(\tau)}_2\right) \nonumber \\ &   \quad + \left(\sup_{t \in (t_{in},T_\star)} \sup_{k \in \Integers_\ast} \int_{t_{in}}^t \sum_{\ell \in \Integers_\ast} \bar{K}(t,\tau,k,\ell) d\tau\right)^{1/2} \\ & \quad\quad \times \left(\sup_{\tau \in (t_{in},T_\star)} \sup_{\ell \in \Integers_\ast} \int_{\tau}^{T_\star} \sum_{k \in \Integers_\ast} \bar{K}(t,\tau,k,\ell) d t\right)^{1/2}
 \nonumber \\ & \quad\quad \times \norm{A_2\rho}_{L^2_t L^2} \left(\sup_{\tau\in(t_{in},T_\star)} \norm{\jap{v}Bf^H(\tau)}_{2}\right),
\end{align*}
where 
\begin{align*}
\bar{K}(t,\tau,k,\ell) = \abs{\widehat{W}(\ell) \ell k (t-\tau)} e^{\left(\mu(t)-\mu(\tau)\right)(K\epsilon)^{1/3}\jap{k,kt}^{1/3}} e^{-c(K\epsilon)^{1/3}\jap{k-\ell,kt-\ell \tau}^{1/3}} \jap{k-\ell,kt-\ell \tau}^{-\gamma}. 
\end{align*}
Due to the small-ness coming from \eqref{ineq:BfH}, it will suffice to use: 
\begin{align}
\left(\sup_{t \in (t_{in},T_\star)} \sup_{k \in \Integers_\ast} \int_{t_{in}}^t \sum_{\ell \in \Integers_\ast} \bar{K}(t,\tau,k,\ell) d\tau\right)^{1/2}\left(\sup_{\tau \in (t_{in},T_\star)} \sup_{\ell \in \Integers_\ast} \int_{\tau}^{T_\star} \sum_{k \in \Integers_\ast} \bar{K}(t,\tau,k,\ell) d t\right)^{1/2} \lesssim \eta_0^2. \label{ineq:trivKernEst}
\end{align}
Using Lemma \ref{lem:AfHBfH} and the bootstrap hypothesis \eqref{ineq:LoRho}, we therefore have
\begin{align*}
\norm{A_2L_{R;H}}_{L^2_t(I;L^2)} & \lesssim \epsilon^{\sigma/5} \epsilon^{1-a} \eta_0^{R(a-1)+3} + \eta_0^2 \eta_0^{-\sigma/5} \norm{A\rho}_{L^2_t(I;L^2)}. 
\end{align*}
Therefore, recalling \eqref{def:eta0}, by choosing $\sigma$ such that 
\begin{align}
\sigma > 20 + 5(a-1) + 5\left(R(a-1)+3\right), \label{ineq:siglargeLRH}  
\end{align}
we have for $\epsilon$ sufficiently small, 
\begin{align*}
\norm{A_2 L_{R;H}}_{L^2_t(I;L^2)} & \lesssim \epsilon^4 + \epsilon^2 \norm{A\rho}_{L^2_t L^2}, 
\end{align*}
which (for $\epsilon$ sufficiently small), implies an estimate on \eqref{def:LRrhoEst} 
consistent with the improvement of \eqref{ineq:HiRho} desired in Proposition \ref{prop:boot}.
This moreover completes the treatment of $L_R$.  

\subsubsection{Linear transport term $L_T$} 
As in the treatment of $L_{R}$ above, we sub-divide into low and high frequency contributions: 
\begin{align*}
L_T & = -\frac{1}{2\pi} \sum_{\ell\in \Integers_\ast} \int_{t_{in}}^t \widehat{\rho^L}(\tau,\ell) \widehat{W}(\ell) \ell k(t-\tau) \widehat{g}(\tau,k-\ell,kt-\ell \tau) d\tau \\ & \quad  -\frac{1}{2\pi}\sum_{\ell \in \Integers_\ast} \int_{t_{in}}^t \widehat{\rho^H}(t,\ell) \widehat{W}(\ell) \ell k(t-\tau) \widehat{g}(\tau,k-\ell,kt-\ell \tau) d\tau  \\ 
& = L_{T;L} + L_{T;H}. 
\end{align*}
For the $L_{T;L}$ term, we may adapt in a straightforward manner the treatment of $\mathcal{C}_{LH}$ in the proof of Lemma \ref{lem:GenCEst} to deduce, using Lemma \ref{lem:AfLBfL} and \eqref{ineq:Hig},   
\begin{align*}
\norm{A_2 L_{T;L}}_{L^2_t(I;L^2)} & \lesssim \epsilon e^{-\frac{1}{2}\epsilon^{-q}} \left(\sup_{t \in (t_{in},T_\ast)} \jap{t}^{-10} \norm{\jap{\grad} A_2 g(t)}_{L^2}\right) 
 \lesssim e^{-\frac{1}{2}\epsilon^{-q}}\epsilon^2, 
\end{align*}
which for $\epsilon$ sufficiently small is consistent with Proposition \ref{prop:boot}. 

For $L_{T;H}$ we may apply Lemma \ref{lem:GenCEst} to yield, 
\begin{align*}
\norm{AL_{T;H}}_{L^2_t(I;L^2)} & \lesssim \norm{B \rho^H}_{L^2_t(I;L^2)} \sup_{\tau \in (t_{in},T_\star)}\left(\jap{\tau}^{-5} \norm{\jap{v}\jap{\grad} A_2 g(\tau)}_2\right) \nonumber \\ &   \quad + \left(\sup_{t \in (t_{in},T_\star)} \sup_{k \in \Integers_\ast} \int_{t_{in}}^t \sum_{\ell \in \Integers_\ast} \bar{K}(t,\tau,k,\ell) d\tau\right)^{1/2}\left(\sup_{\tau \in (t_{in},T_\star)} \sup_{\ell \in \Integers_\ast} \int_{\tau}^{T_\star} \sum_{k \in \Integers_\ast} \bar{K}(t,\tau,k,\ell) d t\right)^{1/2}
 \nonumber \\ & \quad\quad \times \norm{A_2\rho^H}_{L^2_t(I;L^2)} \left(\sup_{\tau\in(t_{in},T_\star)} \norm{\jap{v} Bg(\tau)}_{2}\right). 
\end{align*} 
Analogous to the treatment of $L_{R;H}$, by Corollary \ref{cor:RhoHctrls} and the bootstrap hypotheses \eqref{def:bootg}, this is estimated via
\begin{align*}
L_{T;H} & \lesssim \eta_0^{-\sigma/5} \epsilon^2 + \epsilon^{\sigma/5 - a +1} \eta_0^{R(a-1)+5}. 
\end{align*}
Analogous to \eqref{ineq:siglargeLRH}, for $\sigma$ chosen large relative only to $a$ and $r$, 
we have 
\begin{align*}
\norm{A_2 L_{T;H}}_{L^2_t(I;L^2)}  & \lesssim \epsilon^4, 
\end{align*}
which is consistent with Proposition \ref{prop:boot} for $\epsilon$ small. 
This completes the linear transport term $L_{T}$.  

\subsubsection{Nonlinear term $NL$}
By Lemma \ref{lem:GenCEst} and estimating as in $L_{R;H}$ or $L_{T;H}$ above, using \eqref{ineq:trivKernEst} and \eqref{def:bootg}: 
\begin{align*}
\norm{A_2 NL}_{L^2_t(I;L^2)} & \lesssim \norm{B\rho}_{L^2_t(I;L^2)} \sup_{\tau \in (t_{in},T_\star)}\left( \jap{\tau}^{-5} \norm{\jap{v}\jap{\grad}A_2g(\tau)}_2\right) \\ 
& \quad + \eta_0^2 \norm{A_2\rho}_{L^2_t(I;L^2)} \left(\sup_{\tau \in (t_{in},T_\ast)} \norm{\jap{v} Bg(\tau)}_{2}\right) \\ 
& \lesssim \left(1 + \eta_0^2\right) \epsilon^{\sigma/5+2}. 
\end{align*}
For $\epsilon$ sufficiently small, this is consistent with Proposition \ref{prop:boot}. 

\subsubsection{Consistency error} \label{sec:ArhoApproxError}
In this section we estimate $E$ in \eqref{def:rho0}. 
Recall that the consistency error is given by \eqref{def:cE}. 
Applying Lemmas \ref{lem:GenCEst} and \ref{lem:Abckgr} together with \eqref{ineq:trivKernEst} gives
\begin{align*}
\norm{A_2E}_{L^2_t(I;L^2)} & \lesssim \norm{B\rho^L}_{L^2_t(I;L^2)} \sup_{\tau \in (t_{in},T_\star)}\left(\norm{\jap{v}\jap{\grad}A_2 f^L(\tau)}_2 + \norm{\jap{v}\jap{\grad}A_2 f^H(\tau)}_2\right) \\
& \quad + \eta_0^2  \norm{A_2\rho^L}_{L^2_t(I;L^2)}  \sup_{\tau \in (t_{in},T_\star)}\left(\norm{\jap{v}Bf^L(\tau)}_2 + \norm{\jap{v}B f^H(\tau)}_2\right) \\ 
& \quad + \delta \norm{A_2\rho^L}_{L^2_t(I;L^2)} +  \norm{B\rho^H}_{L^2_t(I;L^2)} \sup_{\tau \in (t_{in},T_\star)} \norm{\jap{v}\jap{\grad}A_2 f^H(\tau)}_2 \\
& \quad + \eta_0^2 \norm{A_2\rho^H}_{L^2_t(I;L^2)}  \sup_{\tau \in (t_{in},T_\star)}\norm{\jap{v}Bf^H(\tau)}_2.
\end{align*}
Applying Lemmas \ref{lem:AfHBfH} and \ref{lem:AfLBfL}, together with Corollary \ref{cor:RhoHctrls}, then implies 
\begin{align*}
\norm{A_2E}_{L^2_t(I;L^2)} & \lesssim e^{-\frac{1}{2}\epsilon^{-q}}\left(\epsilon + \epsilon^{1-a} \eta_0^{R(a-1)+3}\right) +  \delta e^{-\frac{1}{2}\epsilon^{-q}} + \epsilon^{\sigma/5} \epsilon^{1-a} \eta_0^{R(a-1)+5}.
\end{align*}
Hence, by choosing $\sigma$ large depending only on $a$ and $R$, similar to e.g. \eqref{ineq:siglargeLRH}, we have for $\epsilon$ sufficiently small, 
\begin{align*}
\norm{A_2(t)E}_{L^2_t L^2} & \lesssim \epsilon^6, 
\end{align*}
which is consistent with Proposition \ref{prop:boot}. 

\subsection{Estimate on $A(t)g$} \label{sec:AgEst}
In this section, we improve the constant in the estimate \eqref{ineq:Midg}. 
Let $\alpha \in \set{0,1}$. 
Computing from \eqref{def:g}, we have (using that $-\partial_t w \leq 0$)
\begin{align}
\frac{1}{2}\frac{d}{dt}\norm{A(t)(v^\alpha g)}_{2}^2 & \leq \dot{\mu} (K\epsilon)^{1/3}\norm{\jap{\grad}^{1/6} A(t)(v^\alpha g)}_2^2 - \jap{A(t)(v^\alpha g), A(t)v^\alpha \left(E(z+tv) \partial_v f^0\right)}_2 \nonumber \\ & \quad - \jap{A(t)(v^\alpha g), A(t)v^\alpha \left(E(z+tv) (\partial_v - t\partial_z) f^E\right)}_2 \nonumber \\ & \quad - \jap{A(t)(v^\alpha g), A(t)v^\alpha \left(E^E(z+tv)  (\partial_v - t\partial_z) g\right)}_2  \nonumber \\ 
& \quad - \jap{A(t)(v^\alpha g), A(t)v^\alpha \left(E(z+tv)(\partial_v - t\partial_z) g \right)}_2 -\jap{A(t)(v^\alpha g), A(t)\left( v^\alpha \mathcal{E}\right)}_2 \nonumber \\ 
& = -CK_\mu - L_0 - L_{R} - L_T - NL - E. \label{ineq:gMidEnergEst}
\end{align}
Notice that since $t \leq \eta_0$, we have by $b < 1$ (not being very precise), 
\begin{align*}
\dot{\mu}(t) \lesssim -\epsilon^{-bq} \jap{t}^{-1-b} \lesssim -\epsilon^{-bq} \eta_0^{-1-b} \lesssim -\epsilon^{2}. 
\end{align*}
The linear term $L_0$ is estimated via Lemma \ref{lem:Abckgr}, which implies 
\begin{align}
\abs{L_0} & \lesssim \delta\norm{A(v^\alpha g)}_2 \norm{A\rho}_2 \lesssim \frac{\delta}{\jap{t}}\norm{A(v^\alpha g)}_2 \norm{A_1\rho}_2 \lesssim \frac{\delta}{\jap{t}^2}\norm{A(v^\alpha g)}_2^2 + \delta\norm{A_1 \rho}_2^2. \label{ineq:L0MidEst}    
\end{align}
By the bootstrap hypothesis \eqref{ineq:HiRho}, this is consistent with Proposition \ref{prop:boot} for $\delta$ sufficiently small. 

\subsubsection{Treatment of $L_R$} \label{sec:LRMidg}
First divide into low and high frequency contributions: 
\begin{align*}
L_R & = \jap{A(t)v^\alpha g, A(t)v^\alpha \left(E(z+tv)\cdot (\partial_v - t\partial_z) f^L\right)}_{2}  +  \jap{A(t) v^\alpha g, A(t)v^\alpha \left(E(z+tv)\cdot (\partial_v - t\partial_z) f^H\right)}_2 \\ 
& = L_{R;L} + L_{R;H}. 
\end{align*}
The more dangerous is the $L_{R;L}$ term. By the computations used to prove Lemma \ref{lem:AB}, we have the following (also applying Lemma \ref{lem:AfLBfL}), 
\begin{align*}
\abs{L_{R;L}}  & \lesssim \norm{A(t)(v^\alpha g)}_2 \norm{A \rho}_2 \norm{A(t)v^\alpha \left(\partial_v - t\partial_z\right)f^L}_2 \\ 
& \lesssim \epsilon \jap{t}\norm{A(t)(v^\alpha g)}_2 \norm{A \rho}_2 \\ 
& \lesssim \epsilon \jap{t}^{-1} \norm{A(t)(v^\alpha g)}_2 \norm{A_2 \rho}_2 \\
& \lesssim \frac{\epsilon}{\jap{t}^2}\norm{A(t)(v^\alpha g)}_2^2 + \epsilon \norm{A_2 \rho}_2^2, 
\end{align*}
which by \eqref{ineq:HiRho}, is consistent with Proposition \ref{prop:boot} for $\epsilon$ sufficiently small. 

For $L_{R;H}$ we apply \eqref{ineq:prod} (also Lemma \ref{lem:AMomentEquiv}),  
\begin{align*}
\abs{L_{R;H}} & \lesssim \norm{A(t)(v^\alpha g)}_2 \jap{t}\norm{A\rho}_2 \norm{B\left(\jap{v} f^H\right)}_2 + \jap{t}\norm{A(t)(v^\alpha g)}_2\norm{\jap{\partial_x, t\partial_x}^{-2} B\rho}_2\norm{\jap{\grad} A\left(\jap{v} f^H\right)}_2. 
\end{align*} 
Therefore, by \eqref{def:bootg} and Lemma \ref{lem:AfHBfH},  
\begin{align*}
\abs{L_{R;H}} & \lesssim \norm{A(t)(v^\alpha g)}_2\left(\eta_0^{-\sigma/5} \norm{A_1\rho}_2 + \epsilon^{\sigma/5} \epsilon^{1-a} \eta_0^{R(a-1)+2}\right). 
\end{align*}
Hence, for $\sigma$ sufficiently large depending only on $a$ and $R$ and for $\epsilon$ sufficiently small, this is consistent with Proposition \ref{prop:boot}. 

\subsubsection{Treatment of $L_T$} \label{sec:LRTMidg}
Turn next to $L_T$, which requires additional work to properly take advantage of the transport structure. 
As usual, separate the low and high contributions: 
\begin{align}
L_T & = \jap{A(t)(v^\alpha g), A(t)v^\alpha \left(E^L(z+tv)(\partial_v - t\partial_z) g\right)}_2 + \jap{A(t)(v^\alpha g), A(t)v^\alpha \left(E^H(z+tv)(\partial_v - t\partial_z) g\right)}_2 \nonumber \\ 
& = L_{T;L} + L_{T;H}. \label{def:LTLLTH} 
\end{align}

\textbf{Treatment of $L_{T;L}$:}\\
\noindent
Turn to the low frequency term first. 
Commuting the moment and derivatives gives, 
\begin{align*}
L_{T;L} &= \jap{A(t)(v^\alpha g), A(t)\left(E^L(z+tv)(\partial_v - t\partial_z)(v^\alpha g)\right)}_2 + \mathbf{1}_{\alpha = 1}\jap{A(t)(v^\alpha g), A(t)\left(E^L(z+tv) g\right)}_2 \\ 
& = L_{T;L,0} + L_{T;L,M}. 
\end{align*} 
By \eqref{ineq:prod} and \ref{lem:AfLBfL}, the lower order term is estimated via:
\begin{align*}
L_{T;L,M} & \lesssim \epsilon e^{-\frac{1}{2}\epsilon^{-q}}\norm{A(v^\alpha g)}_2^2, 
\end{align*}
which is consistent with Proposition \ref{prop:boot} for $\epsilon$ sufficiently small.
Turn now to the leading order term. In order to take advantage of the transport structure, we use integration by parts to introduce a commutator. 
This commutator trick is standard for dealing with transport equations in Gevrey regularity; see e.g. \cite{LevermoreOliver97}, however, due to $G$, things are more complicated here (as in \cite{BM13}),   
\begin{align}
L_{T;L} & = \jap{A(t) v^\alpha g, A(t)\left(E^L(z+tv)(\partial_v - t\partial_z) v^\alpha g\right) - \left(E^L(z+tv) (\partial_v - t\partial_z) A(t) v^\alpha g\right)}_2 \nonumber \\
& = \frac{1}{2\pi} \int A\overline{\partial_\eta^\alpha \hat{g}}(k,\eta)\left[A(t,k,\eta) - A(t,k-\ell,\eta-\ell \tau)\right] \nonumber \\ & \quad\quad \times \widehat{\rho^L}(t,\ell) \widehat{W}(\ell) (\eta-tk) \partial_\eta^\alpha \hat{g}(k-\ell,\eta-t \ell) d\eta.   \label{def:LTL}
\end{align}
Hence, consider the difference $A(t,k,\eta) - A(t,k-\ell,\eta-\ell \tau)$, which is divided into three contributions: 
\begin{align}
A(t,k,\eta) - A(t,k-\ell,\eta-t\ell) & = \jap{k,\eta}^\beta G(t,k,\eta)\left[e^{\mu(K\epsilon)^{1/3}\jap{k,\eta}^{1/3}} -  e^{\mu(K\epsilon)^{1/3}\jap{k-\ell,\eta-t\ell}^{1/3}}\right] \nonumber \\ 
 & \hspace{-2cm} \quad + G(t,k,\eta)e^{\mu(K\epsilon)^{1/3}\jap{k-\ell,\eta-t\ell}^{1/3}}\left[\jap{k,\eta}^{\beta} - \jap{k-\ell,\eta-t\ell}^\beta\right] \nonumber \\ 
 & \hspace{-2cm} \quad + \jap{k-\ell,\eta-t\ell}^{\beta}e^{\mu(K\epsilon)^{1/3}\jap{k-\ell,\eta-t\ell}^{1/3}}\left[G(t,k,\eta) - G(t,k-\ell,\eta-t\ell)\right]. \label{eq:AcommDiv} 
\end{align}
 This leads to three corresponding terms from \eqref{def:LTL}: 
\begin{align*}
L_{T;L} = L_{T;L,\mu} + L_{T;L,S} + L_{T;L,G}. 
\end{align*}
Analogous to \eqref{ineq:TGest} above, by the mean-value theorem, 
\begin{align*}
\abs{L_{T;L,\mu}} & \lesssim \jap{t}\int \abs{A\partial_\eta^\alpha \hat{g}(k,\eta)}\jap{k,\eta}^{\beta} G(t,k,\eta) \frac{(K\epsilon)^{1/3} \jap{\ell,\ell \tau}^{5/3}}{\jap{k-\ell,\eta-t\ell}^{2/3}} e^{\mu(K\epsilon)^{1/3}\jap{k-\ell,\eta - t\ell}^{1/3}} e^{\mu(K\epsilon)^{1/3}\jap{\ell,\ell t}^{1/3}} \\ & \quad\quad\quad \times \abs{\widehat{\rho^L}(t,\ell)} \jap{k-\ell,\eta-t\ell}\abs{\partial_\eta^\alpha \hat{g}(k-\ell,\eta-t \ell)} d\eta. 
\end{align*}  
Hence, by Lemma \ref{lem:Gcomp}, \eqref{ineq:L1L2}, \eqref{ineq:IncExp}, and a straightforward variant of Lemma \ref{lem:AfLBfL}, we have for $\epsilon$ sufficiently small,  
\begin{align*}
\abs{L_{T;L,\mu}} & \lesssim_{K,\sigma} \epsilon^{4/3} e^{-\frac{1}{8}\epsilon^{-q}} e^{-\frac{1}{8}t} \norm{\jap{\grad}^{1/6}A(v^\alpha g)}_2^2 \lesssim \epsilon CK_\mu.  
\end{align*}
Hence, for $\epsilon$ sufficiently small, this term is absorbed by the $CK_\mu$ term in \eqref{ineq:gMidEnergEst}. 
Analogous to \eqref{ineq:TSest} above, for the Sobolev term $L_{T;L,S}$, we may estimate as follows, using again Lemma \ref{lem:Gcomp}, \eqref{ineq:L1L2}, \eqref{ineq:IncExp}, and a straightforward variant of Lemma \ref{lem:AfLBfL}, 
\begin{align*}
\abs{L_{T;L,S}} & \lesssim_\sigma \jap{t}\int \abs{A\partial_\eta^\alpha \hat{g}(k,\eta)} G(t,k,\eta) \frac{\jap{\ell,\ell \tau}^{2}}{\jap{k-\ell,\eta-t\ell}} e^{\mu(K\epsilon)^{1/3}\jap{k-\ell,\eta-\ell t}^{1/3}} \\ & \quad\quad\quad \times \jap{\ell, \ell t}^{\beta} \abs{\widehat{\rho^L}(t,\ell)} \jap{k-\ell,\eta-t\ell}^{\beta+1}\abs{\partial_\eta^\alpha \hat{g}(k-\ell,\eta-t \ell)} d\eta \\ 
& \lesssim_{K,\sigma}  \epsilon e^{-\frac{1}{8}\epsilon^{-q}} e^{-\frac{1}{8}t} \norm{A(v^\alpha g)}_2^2, 
\end{align*}
which is consistent with Proposition \ref{prop:boot} for $\epsilon$ sufficiently small. 

The $G$ contribution, $L_{T;L,G}$, is significantly trickier. For this, we employ a variant of a trick used in \cite{BM13}. 
First, we divide based on the relationship between time and frequency, 
\begin{align*}
L_{T;L,G} & = \frac{1}{2\pi}\int A\overline{\partial_\eta^\alpha \widehat{g}}(t,k,\eta) \left(\mathbf{1}_{t < \frac{1}{2}(K\epsilon)^{-1/3}\min(\abs{\eta}^{2/3},\abs{\eta-t\ell}^{2/3})} + \mathbf{1}_{t \geq \frac{1}{2}(K\epsilon)^{-1/3}\min(\abs{\eta}^{2/3},\abs{\eta-t\ell}^{2/3})}\right) \\ & \quad\quad\quad \times \left[\frac{G(t,k,\eta)}{G(t,k-\ell,\eta-t\ell)} - 1\right]A(t,k-\ell,\eta-t\ell)  \\ 
& \quad\quad\quad \times \widehat{\rho^L}(\tau,\ell) \widehat{W}(\ell) \ell (\eta-tk)\partial_\eta^\alpha \hat{g}(t,k-\ell,\eta-t\ell) d\eta \\ 
& = L_{T;L,G}^{ST} + L_{T;L,G}^{LT};
\end{align*}
the `ST' stands for `short-time' and `LT' stands for `long-time'. 
For the short-time contribution, we apply Lemma \ref{lem:commG}, followed by \eqref{ineq:L1L2} to deduce (again using a variant of Lemma \ref{lem:AfLBfL}), 
\begin{align*}
\abs{L_{T;L,G}^{ST}} & \lesssim \frac{\epsilon t}{(\epsilon K)^{2/3}}e^{-\frac{1}{8}\epsilon^{-q}} e^{-\frac{1}{8}t}\norm{\jap{\grad}^{1/6}A(v^\alpha g)}_2^2 \lesssim \epsilon CK_\mu. 
\end{align*}
Hence, for $\epsilon$ sufficiently small, this term is absorbed by the $CK_\mu$ term in \eqref{ineq:gMidEnergEst}. 
For the $L_{T;L,G}^{LT}$ term, we subdivide into two more contributions:
\begin{align*}
\abs{L_{T;L,G}^{LT}} & = \frac{1}{2\pi}\int A\overline{\partial_\eta^\alpha \widehat{g}}(t,k,\eta) \mathbf{1}_{t \geq \frac{1}{2}(K\epsilon)^{-1/3}\min(\abs{\eta}^{2/3},\abs{\eta-t\ell}^{2/3})}\left(\mathbf{1}_{\abs{k,\ell} > 10\abs{\eta,\eta-t\ell}} + \mathbf{1}_{\abs{k,\ell} \leq 10\abs{\eta,\eta-t\ell}}\right) \\ & \quad\quad \times \left[\frac{G(t,k,\eta)}{G(t,k-\ell,\eta-t\ell)}-1\right] A(t,k-\ell,\eta-t\ell) \\ & \quad\quad\times \widehat{\rho^L}(\tau,\ell) \widehat{W}(\ell) \ell \cdot (\eta-tk)\partial_\eta^\alpha \hat{g}(t,k-\ell,\eta-t\ell) d\eta \\
& =  L_{T;L,G}^{LT,z} + L_{T;L,G}^{LT,v}. 
\end{align*}
On $L_{T;L,G}^{L,z}$, we can again apply Lemma \ref{lem:commG} and proceed as in $L_{T;L,G}^{ST}$ above to deduce 
\begin{align*}
\abs{L_{T;L,G}^{LT,z}} & \lesssim \frac{\epsilon t}{(\epsilon K)^{2/3}}e^{-\frac{1}{8}\epsilon^{-q}} e^{-\frac{1}{8}t}\norm{\jap{\grad}^{1/6}A(v^\alpha g)}_2^2, 
\end{align*}
which for $\epsilon$ small is absorbed by the $CK_\mu$ term in \eqref{ineq:gMidEnergEst}.  
Turn now to $L_{T;L,G}^{L,v}$, where Lemma \ref{lem:commG} does not apply.
However, we may use the restriction on time to gain powers of $\eta$ or $\eta-t\ell$, which due to the relative small-ness of $k$ and $\ell$, is sufficient. 
For any fixed, small $c > 0$, we have (using also \eqref{ineq:L1L2} and a variant of Lemma \ref{lem:AfLBfL}), 
\begin{align*}
\abs{L_{T;L,G}^{L,v}} & \lesssim \jap{t}\int \abs{A\partial_\eta^\alpha \widehat{g}(t,k,\eta)} \mathbf{1}_{t \geq \frac{1}{2}(K\epsilon)^{-1/3}\min(\abs{\eta}^{2/3},\abs{\eta-t\ell}^{2/3})}\mathbf{1}_{\abs{k,\ell} \leq 10\abs{\eta,\xi}} \frac{\jap{t}^2 (K\epsilon)^{2/3} }{\jap{\eta}^{4/3} + \jap{\eta-t\ell}^{4/3}} \\ & \quad\quad\quad \times \jap{k-\ell,\eta-t\ell} \abs{\widehat{\rho^L}(\tau,\ell) A\partial_\eta^\alpha \hat{g}(t,k-\ell,\eta-t\ell)} d\eta \\ 
& \lesssim \epsilon (K\epsilon)^{2/3}  e^{-\frac{1}{10}\epsilon^{-q}}e^{-\frac{1}{10}t} \norm{A(v^\alpha g)}_{2}^2,  
\end{align*} 
which for $\epsilon$ sufficiently small, is consistent with Proposition \ref{prop:boot}. 
This completes the $L_{T;L}$ contribution in \eqref{def:LTLLTH}. 

\textbf{Treatment of $L_{T;H}$:}\\
\noindent
Turn next to $L_{T;H}$. 
As in $L_{T;L}$, first commute the moment and derivatives,
\begin{align*}
L_{T;L} &= \jap{A(t)(v^\alpha g), A(t)\left(E^H(z+tv)(\partial_v - t\partial_z)(v^\alpha g)\right)}_2 + \jap{A(t)(v^\alpha g), A(t)\left(E^H(z+tv) g\right)}_2 \\ 
& = L_{T;H,0} + L_{T;H,M}. 
\end{align*} 
By \eqref{ineq:prod} followed by \eqref{def:bootg} and Lemma \ref{lem:AfHBfH} we have 
\begin{align*}
\abs{L_{T;H,M}} & \lesssim \norm{A(v^\alpha g)}_2\left(\norm{A\rho^H}_2\norm{Bg}_2 +  \norm{B\rho^H}_2\norm{Ag}_2\right) \\ 
& \lesssim \norm{A(v^\alpha g)}_2\left(\epsilon^{-a+1}\eta_0^{R(a-1)} \epsilon^{\sigma/5} +  \eta_0^{-\sigma/5}\norm{Ag}_2\right). 
\end{align*}
Hence, for $\sigma$ sufficiently large depending only on $a$ and $R$, and that $t \leq \eta_0$, this term is consistent with Proposition \ref{prop:boot} for $\epsilon$ (and hence $\eta_0^{-1}$) sufficiently small.   

Turn next to the leading order term. As in \eqref{def:LTL} above, we integrate by parts to take advantage of the transport structure: 
\begin{align}
L_{T;H} & = \jap{A(t) v^\alpha g, A(t)\left(E^H(z+tv)(\partial_v - t\partial_z) v^\alpha g\right) - \left(E^H(z+tv) (\partial_v - t\partial_z) A(t) v^\alpha g\right)}_2 \nonumber \\
& = \frac{1}{2\pi} \int A\overline{\partial_\eta^\alpha  \hat{g}}(k,\eta)\left[A(t,k,\eta) - A(t,k-\ell,\eta-\ell \tau)\right] \nonumber \\ & \quad\quad \times \widehat{\rho^H}(t,\ell) \ell \widehat{W}(\ell) (\eta-tk) \partial_\eta^\alpha \hat{g}(k-\ell,\eta-t \ell) d\eta.   \label{def:LTH}
\end{align} 
Next, we expand with a paraproduct decomposition: 
\begin{align*}
L_{T;H} & = \frac{1}{2\pi} \sum_{M \in 2^\Integers} \int A\overline{\partial_\eta^\alpha \hat{g}}(k,\eta)\left[A(t,k,\eta) - A(t,k-\ell,\eta-\ell \tau)\right] \widehat{\rho^H_M}(t,\ell) \ell \widehat{W}(\ell) (\eta-tk) \partial_\eta^\alpha \widehat{g_{<M/8}}(k-\ell,\eta-t \ell) d\eta \\ 
& \quad + \frac{1}{2\pi} \sum_{M \in 2^\Integers} \int A\overline{\partial_\eta^\alpha \hat{g}}(k,\eta)\left[A(t,k,\eta) - A(t,k-\ell,\eta-\ell \tau)\right] \widehat{\rho^H_{<M/8}}(t,\ell) \ell \widehat{W}(\ell) (\eta-tk) \partial_\eta^\alpha \widehat{g_{M}}(k-\ell,\eta-t \ell) d\eta \\ 
& \quad + \frac{1}{2\pi} \sum_{M \in 2^\Integers} \int A\overline{\partial_\eta^\alpha \hat{g}}(k,\eta)\left[A(t,k,\eta) - A(t,k-\ell,\eta-\ell \tau)\right] \widehat{\rho^H_M}(t,\ell) \ell \widehat{W}(\ell) (\eta-tk) \partial_\eta^\alpha \widehat{g_{\sim M}}(k-\ell,\eta-t \ell) d\eta \\ 
& = L_{T;H;HL} + L_{T;H;LH} + L_{T;H;\mathcal{R}}. 
\end{align*}
Consider first $L_{T;H;HL}$. The commutator is not useful here and we treat the two pieces separately, 
\begin{align*}
L_{T;H;HL} & = \frac{1}{2\pi} \sum_{M \in 2^\Integers} \int A\overline{\partial_\eta^\alpha \hat{g}}(k,\eta)A(t,k,\eta) \widehat{\rho^H_M}(t,\ell) \ell \widehat{W}(\ell) (\eta-tk) \partial_\eta^\alpha \widehat{g_{<M/8}}(k-\ell,\eta-t \ell) d\eta \\ 
& \quad + \frac{1}{2\pi} \sum_{M \in 2^\Integers} \int A\overline{\partial_\eta^\alpha \hat{g}}(k,\eta)A(t,k-\ell,\eta - t\ell) \widehat{\rho^H_M}(t,\ell) \ell \widehat{W}(\ell) (\eta-tk) \partial_\eta^\alpha \widehat{g_{<M/8}}(k-\ell,\eta-t \ell) d\eta \\ 
& =  L_{T;H;HL}^0 + L_{T;H;HL}^1. 
\end{align*}
Applying the same arguments as used in Lemma \ref{lem:AB}, followed  by \eqref{def:bootg} and Corollary \ref{cor:RhoHctrls}, we have
\begin{align*}
\abs{L_{T;H;HL}^0} & \lesssim \jap{t}\sum_{M \in 2^\Integers} \norm{A(v^\alpha g)_{\sim M}}_2 \norm{A\rho^H_M}_{2} \norm{B(v^\alpha g)_{<M/8}}_2 \\ 
& \lesssim \norm{A(v^\alpha g)}_2 \norm{A_1\rho^H}_{2} \norm{B(v^\alpha g)}_2 \\ 
& \lesssim \norm{A(v^\alpha g)}_2 \epsilon^{1-a} \eta_0^{R(a-1)+1}\epsilon^{\sigma/5}. 
\end{align*}
Using $t \leq T_\ast \leq \eta_0$, this is consistent with Proposition \ref{prop:boot} after choosing $\sigma$ large relative to $a$ and $R$ and then choosing $\epsilon$ sufficiently small.
The term $L_{T;H;HL}^1$ is straightforward using the frequency localizations; we omit the treatment for brevity and simply state the result: 
\begin{align*}
\abs{L_{T;H;HL}^1} & \lesssim \norm{A(v^\alpha g)}_2 \norm{\jap{\partial_x,t\partial_x}^{2}\rho^H}_{2} \norm{A(v^\alpha g)}_2  \lesssim \epsilon^{\sigma/5} \norm{A(v^\alpha g)}_2^2, 
\end{align*}
which is sufficient for Proposition \ref{prop:boot} for $\epsilon$ sufficiently small. 

Consider next $L_{T;H;LH}$, the term which requires the commutator. 
The term may be treated in a manner very similar to how we treated $L_{T;L}$ \eqref{def:LTL}, let us sketch the small differences.
We sub-divide analogously based on \eqref{eq:AcommDiv}
\begin{align}
L_{T;H;LH} &= L_{T;H,\mu} + L_{T;H,S} + L_{T;H,G}. \label{def:LTHLHdivs}
\end{align}
Using the frequency localizations on the support of the integrand, combining the argument used in \eqref{ineq:TGest} with \eqref{lem:scon}, there holds the following for a fixed constant $c \in (0,1)$ (depending only the details of our Littlewood-Paley localizations),   
\begin{align*}
\abs{L_{T;H,\mu}} & \lesssim_{\beta} \jap{t}\sum_{M \in 2^\Integers} \int A\abs{\partial_\eta^\alpha \hat{g}(k,\eta)} (K\epsilon)^{1/3} \jap{k-\ell,\eta-t\ell}^{1/3} G(t,k,\eta) e^{c\mu(K\epsilon)^{1/3}\jap{\ell,\ell \tau}^{1/3}} \abs{\widehat{\rho^H_{<M/8}}(t,\ell)} \\ & \quad\quad \times \jap{k-\ell, \eta-t\ell}^{\beta} e^{\mu(K\epsilon)^{1/3}\jap{k-\ell,\eta-\ell \tau}^{1/3}} \abs{\partial_\eta^\alpha \widehat{g_{<M/8}}(k-\ell,\eta-t \ell)} d\eta. 
\end{align*} 
Using Lemma \ref{lem:Gcomp}, \eqref{ineq:IncExp} and that $\nu(t) \geq c\mu(t) + \tilde{r}$ as in the proof of Lemma \ref{lem:AB}, we have 
\begin{align}
\abs{L_{T;H,\mu}} & \lesssim (K\epsilon)^{1/3}\jap{t} \sum_{M \in 2^\Integers} \norm{\jap{\grad}^{1/6}A(v^\alpha g)_{\sim M}}_2 \norm{\jap{\grad}^{1/6}A(v^\alpha g)_{M}}_2 \norm{B\rho^H}_2 \nonumber \\ 
& \lesssim  (K\epsilon)^{1/3}\jap{t} \epsilon^{\sigma/5}  \norm{\jap{\grad}^{1/6}A(v^\alpha g)}_2^2, \label{ineq:LTmu}
\end{align}
where we used Lemma \ref{lem:AfHBfH}. 
It follows from \eqref{ineq:LTmu} that, for $\epsilon$ sufficiently small, we have 
\begin{align}
\abs{L_{T;H,\mu}} & \lesssim \epsilon \dot{\mu}(t) (K\epsilon)^{1/3}\norm{\jap{\grad}^{1/6}A(v^\alpha g)}_2^2 \lesssim \epsilon CK_\mu. \label{ineq:LTHmu}
\end{align}
Hence, this term is absorbed by $CK_\mu$ in \eqref{ineq:gMidEnergEst}. 
The Sobolev term $L_{T;H,S}$ in \eqref{def:LTHLHdivs} is similar but significantly easier (and in fact the $CK_\mu$ term is not needed). The details are omitted for brevity.  
Turn next to the $L_{T;H,G}$ term in \eqref{def:LTHLHdivs}. 
This term is treated in a very analogous manner to the treatment of $L_{T;L,G}$ above. 
The details are omitted as they are repetitive; the resulting estimate is 
\begin{align*}
\abs{L_{T;L,G}} & \lesssim \frac{\jap{t}}{(\epsilon K)^{2/3}}\norm{B\rho^H}_2 \norm{\jap{\grad}^{1/6}A(v^\alpha g)}_2^2 + (K\epsilon)^{2/3}\jap{t}^3\norm{B\rho^H}_2 \norm{A(v^\alpha g)}_2^2 \\ 
& \lesssim_K \eta_0^{1-\sigma/5} \epsilon^{-2/3} \norm{\jap{\grad}^{1/6}A(v^\alpha g)}_2^2 + \epsilon^{2/3} \eta_0^{3-\sigma/5} \norm{A(v^\alpha g)}_2^2, 
\end{align*}  
which as in \eqref{ineq:LTHmu} above, is consistent with Proposition \ref{prop:boot} after choosing $\epsilon$ (hence $\eta_0^{-1}$) small. 

\subsubsection{Treatment of $NL$} \label{sec:NLMidgEst}
The nonlinear term $NL$ in \eqref{ineq:gMidEnergEst} can be treated in the same manner as $L_{T;H}$, using \eqref{def:bootg} instead of Lemma \ref{lem:AfHBfH} and Corollary \ref{cor:RhoHctrls}. 
The details are omitted for brevity as they are repetitive. 

\subsubsection{Treatment of $E$} \label{sec:EMidgEst}
As in \S\ref{sec:ArhoApproxError}, the consistency error contributions in \eqref{ineq:gMidEnergEst} are fairly straightforward. 
Indeed, from \eqref{ineq:prod} and Lemma \ref{lem:Abckgr}, we have (recall \eqref{def:cE}), 
\begin{align*}
\abs{E} & \lesssim \jap{t}\norm{A(v^\alpha g)}_2\left(\norm{A\rho^L}_2\left(\norm{B(\jap{v} f^L)}_2 +  \norm{B(\jap{v} f^H)}_2\right) + \norm{B\rho^L}_2\left(\norm{A_1(\jap{v}f^L)}_2 +  \norm{A_1(\jap{v}f^H)}_2\right)\right) \\ 
& \quad + \delta \norm{A(v^\alpha g)}_2\norm{A \rho^L}_2 + \jap{t}\norm{A(v^\alpha g)}_2\left(\norm{A\rho^H}_2 \norm{B(\jap{v} f^H)}_2 + \norm{B\rho^H}_2 \norm{A_1(\jap{v} f^H)}_2\right). 
\end{align*} 
Applying Lemmas \ref{lem:AfHBfH} and \ref{lem:AfLBfL} gives the following, 
\begin{align*}
\abs{E} & \lesssim \eta_0\epsilon \norm{A(v^\alpha g)}_2 e^{-\frac{1}{2}\epsilon^{-q}} \left(\eta_0^{-\sigma/5} + \epsilon^{-a+1}\eta_0^{R(a-1)+1}\right) + \delta \epsilon e^{-\frac{1}{2}\epsilon^{-q}} \norm{A(v^\alpha g)}_2 \\ 
& \quad + \eta_0^{1-\sigma/5} \epsilon^{-a+1}\eta_0^{R(a-1)+1} \norm{A(v^\alpha g)}_2.  
\end{align*}
Therefore, for $\sigma$ sufficiently large depending only on $a$ and $R$ and $\epsilon$ (and hence $\eta_0^{-1}$) sufficiently small, these contributions are consistent with Proposition \ref{prop:boot}. 

This completes the improvement to \eqref{ineq:Midg} claimed in Proposition \ref{prop:boot}. 

\subsection{Estimate on $\jap{\grad}^3 A(t) g$}
We next consider improving \eqref{ineq:Hig}. 
The estimate is very similar to the improvement to \eqref{ineq:Midg} carried out in \S\ref{sec:AgEst}; the difference is a trick introduced in \cite{BMM13} to get a reasonably controlled estimate despite the additional derivative. 
Hence, we will only provide a sketch. 
As we saw in \S\ref{sec:AgEst}, getting velocity moments in these estimates is a trivial extension of the estimate with no moments, and hence we ignore the moments in this section for clarity.    
Computing from \eqref{def:g} 
\begin{align}
\frac{1}{2}\frac{d}{dt}\norm{A_3(t) g}_{2}^2 & \leq \dot{\mu} (K\epsilon)^{1/3}\norm{\jap{\grad}^{1/6} A_3(t)g}_2^2 - \jap{A_3(t) g, A_3(t)\left(E(z+tv) \partial_v f^0\right)}_2 \nonumber \\ & \quad - \jap{A_3(t) g, A_3(t)\left(E(z+tv) (\partial_v - t\partial_z) f^E\right)}_2 \nonumber \\ & \quad - \jap{A_3(t)g, A_3(t)\left(E^E(z+tv)  (\partial_v - t\partial_z) g\right)}_2  \nonumber \\ 
& \quad - \jap{A_3(t)(v^\alpha g), A_3(t)\left(E(z+tv)(\partial_v - t\partial_z) g \right)}_2 -\jap{A_3(t)g, A_3(t)\left(\mathcal{E}\right)}_2 \nonumber \\ 
& = -CK_\mu - L_0 - L_{R} - L_T - NL - E. \label{ineq:gHiEnergEst}
\end{align} 
The linear term $L_0$ is estimated via Lemma \ref{lem:Abckgr}, however unlike in \eqref{ineq:L0MidEst}, here we will use the regularizing effect of $W$ (as in \cite{BMM13}): 
\begin{align*}
\abs{L_0} & \lesssim \delta \norm{A_3g}_2 \norm{\frac{\jap{\partial_x, t \partial_x }}{\jap{\partial_x}}A_2 \rho}_2 \lesssim \frac{\delta}{\jap{t}}\norm{A_3 g}_2^2 + \delta \jap{t}^3 \norm{A_2\rho}_2^2. 
\end{align*}
By \eqref{ineq:HiRho}, this is consistent with \eqref{ineq:Hig} for $\delta$ sufficiently small. 

The $L_R$ contribution is treated as in \S\ref{sec:LRMidg} however, we will use the  regularization  effect of $W$ again. We omit the details as they are repetitive. 
Following the same arguments as in \S\ref{sec:LRMidg}, we have the following, (the two terms correspond to the low and high frequencies of the of $f^E$ respectively), 
\begin{align}
\abs{L_R} & \lesssim \epsilon \jap{t} \norm{A_3(t)g}_2 \norm{ \frac{\jap{\partial_x, t \partial_x }}{\jap{\partial_x}} A_2 \rho}_2 \nonumber \\ 
& \quad + \jap{t}\norm{A_3(t)g}_2\left(\eta_0^{\sigma/5-3} \norm{\frac{\jap{\partial_x, t \partial_x }}{\jap{\partial_x}} A_2\rho}_2 + \epsilon^{\sigma/5} \epsilon^{1-a} \eta_0^{a(R+1)-R+4}\right)  \nonumber \\ 
& \lesssim \epsilon \jap{t}^2 \norm{A_3(t)g}_2 \norm{ A_2 \rho}_2 \nonumber \\ 
& \quad + \jap{t}^2\norm{A_3(t)g}_2\left(\eta_0^{-\sigma/5+1} \norm{A_2\rho}_2 + \epsilon^{\sigma/5} \epsilon^{1-a} \eta_0^{R(a-1)+4}\right) \nonumber \\  
& \lesssim \frac{\epsilon}{\jap{t}} \norm{A_3(t)g}_2^2 + \epsilon \jap{t}^5\norm{A_2\rho}_2^2 + \epsilon^{10}\norm{A_3(t)g}_2^2 + \epsilon^{10}, \label{ineq:LRHiG}
\end{align}
where in the last line we used that $\sigma$ is sufficiently large relative to $a$ and $R$ as in e.g. \eqref{ineq:siglargeLRH}. 
Therefore, for $\epsilon$ sufficiently, this is consistent with Proposition \ref{prop:boot}.  

The $L_{T}$ contribution is treated as in \S\ref{sec:LRTMidg}; we omit the treatment as it is similar. 

As in \S\ref{sec:NLMidgEst}, the $NL$ contribution is treated as a combination of the method used to treat the latter terms in $L_R$ above in \eqref{ineq:LRHiG} and the methods used to treat $L_T$ as in \S\ref{sec:LRTMidg} above. We omit the details for brevity. 

Finally, the contribution of the consistency errors, $E$ in \eqref{ineq:gHiEnergEst}, is treated the same as in \S\ref{sec:EMidgEst} above. 
Hence, this is also omitted for brevity.   

\subsection{Estimate on $B_2(t)\rho$}
In this section we improve the estimate \eqref{ineq:LoRho}.  
As in \eqref{ineq:ArhoLt} in \S\ref{sec:dense}, we have by Corollary \ref{cor:linsolve}, 
\begin{align}
\norm{B_2\rho}_{L^2_t(I;L^2)} & \leq \left(1+ \sqrt{\delta}C'(c,\sigma)\right)\norm{B_2 \rho_0}_{L^2_t(I;L^2)}, \label{ineq:BrhoLt}
\end{align}
where $\rho_0$ is defined in \eqref{def:rho0}. Hence, as in \S\ref{sec:dense}, it suffices to estimate $B_2\rho_0$ in $L^2_t(I;L^2)$.    

For simplicity, we are interested in estimating \eqref{ineq:BrhoLt} without including $w$ in our definition of $B$ (see \eqref{def:B}). 
Moreover, we want to avoid introducing unnecessary energy estimates into the scheme laid out in \eqref{def:bootg}.  
We use the estimates of $g$ in the norm $A$ in order to accomplish this. 

\subsubsection{Linear reaction term $L_R$}
As in \S\ref{sec:LRg}, this term is naturally divided into 
\begin{align*}
L_R & = -\frac{1}{2\pi}\sum_{\ell \in \Integers_\ast} \int_{t_{in}}^t \widehat{\rho}(\tau,\ell) \widehat{W}(\ell) \ell k(t-\tau) \widehat{f^L}(\tau,k-\ell,kt-\ell \tau) d\tau \\ & \quad -\frac{1}{2\pi} \sum_{\ell \in \Integers_\ast} \int_{t_{in}}^t \widehat{\rho}(t,\ell) \widehat{W}(\ell) \ell k(t-\tau) \widehat{f^H}(\tau,k-\ell,kt-\ell \tau) d\tau  \\ 
& = L_{R;L} + L_{R;H}. 
\end{align*}
The main difficulty lies in $L_{R;L}$. 
First, by \eqref{ineq:IncExp} followed by Schur's test we have
\begin{align*}
\norm{B_2(t)L_{R;L}}_{L^2_t(I;L^2_k)}^2 & \lesssim_{\gamma} \epsilon^2 \sum_{k \in \Integers_\ast}\int_{t_{in}}^{T_\star} \left[ \sum_{\ell = k\pm 1} \int_{t_{in}}^t B_2(\tau,\ell,\ell \tau) \abs{\rho(\tau,\ell) \widehat{W}(\ell) \ell k(t-\tau)} e^{-\frac{3}{4}\abs{kt-\ell \tau}}  d\tau \right]^2 dt \\ 
& \lesssim \epsilon^2\sum_{k \in \Integers_\ast}\int_{t_{in}}^{T_\star} \abs{B_2(t,k,kt)\frac{t}{\abs{k}} \rho(t,k)}^2 dt. 
\end{align*}
This estimate is essentially losing a derivative, however, we will interpolate against the high norm: 
\begin{align}
\norm{B_2(t)L_{R;L}}_{L^2_t(I;L^2)}^2 & \lesssim \epsilon^2 \sum_{k\in\Integers_\ast}\int_{t_{in}}^{T_\star} e^{2\nu(t) (K\epsilon)^{1/3} \jap{k,kt}^{1/3}} \jap{k,kt}^{2(\gamma+2)+2}  \abs{\rho(t,k)}^2 dt \nonumber \\ 
& \lesssim \epsilon^2 \left( \sum_{k \in \Integers_\ast} \int_{t_{in}}^{T_\star} e^{2\nu(t)(K\epsilon)^{1/3} \jap{k,kt}^{1/3}} \jap{k,kt}^{2(\gamma+2)}  \abs{\rho(t,k)}^2 dt \right)^{1- \frac{1}{(\beta-\gamma)}} \nonumber \\ 
& \quad\quad \times \left( \sum_{k \in \Integers_\ast}\int_{t_{in}}^{T_\star} e^{2\nu(t)(K\epsilon)^{1/3} \jap{k,kt}^{1/3}} \jap{k,kt}^{2(\beta+2)}  \abs{\rho(t,k)}^2 dt \right)^{\frac{1}{(\beta-\gamma)}} \nonumber \\ 
& \lesssim \epsilon^2 \norm{B_2\rho}_{L^2_t(I;L^2)}^{2 - \frac{2}{(\beta-\gamma)}} \norm{A_2\rho}_{L^2_t(I;L^2)}^{\frac{2}{(\beta-\gamma)}} \nonumber \\ 
& \lesssim \epsilon^{2 + \frac{2\sigma}{5}\left(1 - \frac{1}{(\beta-\gamma)}\right) + \frac{4}{(\beta-\gamma)}}. \label{ineq:B2LRL}
\end{align}
In order to be consistent with Proposition \ref{prop:boot}, we will need to ensure
\begin{align*}
2 + \frac{2\sigma}{5}\frac{(\beta-\gamma)-1}{(\beta-\gamma)} + \frac{4}{(\beta-\gamma)} > \frac{2\sigma}{5}. 
\end{align*}
As $\beta-\gamma \in (\beta/4-4, \beta/4-3)$ this holds for $\sigma = \beta+R$ chosen large relative to $R$ and universal constants (as $8/5 < 2$). 
Hence, for $\epsilon$ sufficiently small, \eqref{ineq:B2LRL} is consistent with Proposition \ref{prop:boot}. 

For $L_{R;H}$ we may use a very similar treatment. Indeed, using the algebra property of $B$, followed by Schur's test and Lemma \ref{lem:AfHBfH}, 
we have 
\begin{align*}
\norm{B_2(t)L_{R;H}}_{L^2_t L^2_k}^2 & \lesssim_{\gamma} \sum_{k \in \Integers_\ast}\int_{t_{in}}^{T_\star} \left[ \epsilon \sum_{\ell} \int_{t_{in}}^t B_2(\tau,\ell,\ell \tau) \abs{\widehat{\rho}(\tau,\ell)} \abs{\widehat{W}(\ell) \ell k(t-\tau)} \right. \\ & \quad \times \left. B_2(\tau,k-\ell,kt-\ell \tau) \abs{\widehat{f^H}(\tau,k-\ell,kt-\ell \tau)} d\tau \right]^2 dt \\ 
& \lesssim \epsilon^{2\sigma/5-6}\sum_{k \in \Integers_\ast}\int_{t_{in}}^{T_\star} \abs{B_2(t,k,kt)\frac{t}{\abs{k}} \rho(t,k)}^2 dt.  
\end{align*}
From here we may proceed as in \eqref{ineq:B2LRL} above, and hence for $\sigma$ large relative to $R$ and universal constants and $\epsilon$ chosen sufficiently small, this contribution is consistent with Proposition \ref{prop:boot}. 

\subsubsection{Estimate on $L_T$}
Turn next to $L_{T}$, which we divide as usual into low and high frequencies:  
\begin{align*}
L_T & = -\frac{1}{2\pi}\sum_{\ell \in \Integers_\ast} \int_{t_{in}}^t \widehat{\rho^L}(\tau,\ell) \widehat{W}(\ell) \ell k(t-\tau) \widehat{g}(\tau,k-\ell,kt-\ell \tau) d\tau \\ & \quad - \frac{1}{2\pi}\sum_{\ell \in \Integers_\ast} \int_{t_{in}}^t \widehat{\rho^H}(t,\ell) \widehat{W}(\ell) \ell k(t-\tau) \widehat{g}(\tau,k-\ell,kt-\ell \tau) d\tau  \\ 
& = L_{T;L} + L_{T;H}. 
\end{align*}
Indeed, applying Lemmas \ref{lem:GenCEst} and \ref{lem:AfLBfL} implies,  
\begin{align}
\norm{B_2 L_{T;L}}_{L^2_t(I;L^2)}^2 & \lesssim  \norm{B_2\rho^L}_{L^2_t(I;L^2)}^2 \left(\sup_{t \in (t_{in},T_\ast)} \norm{\jap{\grad}B_2(\jap{v}g(t))}_{L^2}^2\right) \nonumber \\ 
& \lesssim \epsilon^2 e^{-\frac{1}{2}\epsilon^{-q}} \left(\sup_{t \in (t_{in},T_\ast)} \norm{A(\jap{v}g(t))}_{L^2}^2\right) \nonumber \\ 
& \lesssim \epsilon^{6} e^{-\frac{1}{2}\epsilon^{-q}}. \label{ineq:B2LTLfinal}
\end{align}
Due to the exponential, we may choose $\epsilon$ sufficiently small (depending only $\sigma$, $q$, and the constants in the functional inequality) such that 
this is consistent Proposition \ref{prop:boot}. 

We may argue in a manner similar to $L_{T;L}$. 
Indeed, applying Lemma \ref{lem:GenCEst} and Lemma \ref{lem:AfHBfH}, 
\begin{align*}
\norm{B_2 L_{T;H}}_{L^2_t(I;L^2)}^2 & \lesssim  \norm{B_2 \rho^H}_{L^2_t(I;L^2)}^2\left(\sup_{t \in (t_{in},T_\ast)} \norm{\jap{\grad} B_2(\jap{v}g)}_{L^2}^2\right) \\ 
& \lesssim \epsilon^{2\sigma/5+4} \left(\sup_{t \in (t_{in},T_\ast)} \norm{B(\jap{v}g)}_{L^2}\right)^{2\left(1 - \frac{3}{\beta-\gamma}\right)} \left(\sup_{t \in (t_{in},T_\ast)} \norm{A(\jap{v}g)}_{L^2}\right)^{\frac{6}{\beta-\gamma}}  \\
& \lesssim \epsilon^{2\sigma/5 + 4} \epsilon^{\frac{2\sigma}{5}\left(1 - \frac{3}{\beta-\gamma}\right)} \epsilon^{\frac{6}{\beta-\gamma}}. 
\end{align*} 
As above in \eqref{ineq:B2LTLfinal} (although it is much easier here), this is consistent with Proposition \ref{prop:boot} for $\sigma$ large relative to $R$ and universal constants.  

\subsubsection{Treatment of $NL$}
Applying Lemma \ref{lem:GenCEst}, we have  
\begin{align*}
\norm{B_2 NL}_{L^2_t(I;L^2)}^2 & \lesssim \norm{B \rho}_{L^2_t(I;L^2)}^2\left(\sup_{t \in (t_{in},T_\ast)} \norm{\jap{\grad} B_2(\jap{v}g)}_{L^2}^2\right) + \norm{B_2 \rho}_{L^2_t(I;L^2)}^2\left(\sup_{t \in (t_{in},T_\ast)} \norm{B(\jap{v}g)}_{L^2}^2\right) \\ 
& \lesssim \epsilon^{2\sigma/5} \left(\sup_{t \in (t_{in},T_\ast)} \norm{B(\jap{v}g)}_{L^2}\right)^{2\left(1 - \frac{3}{\beta-\gamma}\right)} \left(\sup_{t \in (t_{in},T_\ast)} \norm{A(\jap{v}g)}_{L^2}\right)^{\frac{6}{\beta-\gamma}}  \\
& \quad + \epsilon^{2\sigma/5} \norm{B_2\rho}_{L^2_t(I;L^2)}^2, 
\end{align*}
which is consistent with Proposition \ref{prop:boot} for $\sigma$ large relative to $R$ and universal constants.  

\subsubsection{Treatment of $E$}
These terms are treated as in \S\ref{sec:ArhoApproxError}; we omit the details for brevity.  

\subsection{Estimate on $B g$}
For this estimate we use an easier variant of \S\ref{sec:AgEst}. 
Indeed, most terms are treated in the same manner.  
As the details are repetitive, we omit them for the sake of brevity. 

This completes the proof of Proposition \ref{prop:boot}, and hence of Proposition \ref{prop:stabg}. 
This, in turn, completes the proof of Theorem \ref{thm:main} in the case of gravitational interactions. 

\section{Extension to electrostatic interactions} \label{sec:electro}
The main difficulty in the electrostatic case is getting the lower bound on the approximate solution as in Proposition \ref{prop:blowup}. 
Specifically, if we could take $f^0 \equiv 0$, then the proof of Proposition \ref{prop:blowup} would easily adapt to the electrostatic case.
However, the linear problem dominates the evolution of the critical frequency in the proof of Proposition \ref{prop:blowup} as suggested by \eqref{def:resonantsub2}. 
In the gravitational case, the sign of the fundamental solution given by Lemma \ref{lem:basicvolt} is favorable for obtaining lower bounds, however, in the electrostatic case it is not. 
To overcome this, we will need our upper and lower bounds on the approximate solution to match almost precisely near the critical times.  
To this end,  we will need to choose a slightly more pathological configuration. 
For a small parameter $\alpha \in (0,1)$ to be fixed below depending only on $R$ and $p$ (essentially $\alpha \approx p(1+R)^{-1}$), define 
\begin{align}
f^H_{in}(z,v) & = 16\pi \frac{\epsilon}{\alpha\jap{k_0,\eta_0}^{\sigma}} \cos(k_0z) \frac{\cos(\eta_0 v)}{1 + \alpha^{-2} v^2}. \label{def:fHin_Elec}
\end{align}
The Fourier transform is given by
\begin{align*}
\widehat{f^H_{in}}(k_0,\eta) & = 2\pi \epsilon \jap{k_0,\eta_0}^{-\sigma} \left(e^{-\alpha\abs{\eta - \eta_0}} + e^{-\alpha\abs{\eta + \eta_0}}\right) \\ 
\widehat{f^H_{in}}(-k_0,\eta) & = 2\pi \epsilon\jap{k_0,\eta_0}^{-\sigma} \left(e^{-\alpha\abs{\eta - \eta_0}} + e^{-\alpha\abs{\eta + \eta_0}}\right) \\ 
\widehat{f^H_{in}}(k\neq k_0,\eta) & = 0. 
\end{align*} 
Notice that this increases the size of the initial condition in $H^\sigma$ (by approximately $\alpha^{-\sigma}$ via \eqref{ineq:SobExp}).  
We moreover define $f^{H}_{++}$, $f^{H}_{++}$, $f^{H}_{++}$, $f^{H}_{++}$, analogously \eqref{def:fHpm}. 

The first step, and the most difficult, is deriving the analogue of Proposition \ref{prop:blowup}. 

\subsection{Refined upper and lower bounds on approximate solution} \label{sec:RefinedElectroUpLow}
As above, for convenience, we define $\epsilon' = \epsilon \jap{k_0,\eta_0}^{-\sigma}$. 
By repeating the computations in Lemmas \ref{lem:AprxRhoUp} and \ref{lem:distup} being a little more precise with the radius of regularity, one deduces the following analogous lemma. 
The proof is omitted for brevity. 
\begin{lemma}[Upper bounds on approximate solution in electrostatic case] \label{lem:ElectroAprx_basic} 
Consider the solution to \eqref{eq:seciter} $f^H_{++}$ with initial data $f^H(t_{in}) = f^H_{++}(t_{in})$ and the associated density $\rho_{++}$. 
Then for all $0 < \kappa' < \kappa'' < \alpha$ there exists a constant $K_\alpha$ such that if we define the growth rate 
\begin{align*}
N_\alpha & = \textup{Floor}((K_\alpha\epsilon)^{1/3} \eta_0^{1/3}), \\ 
Y_k(\eta_0) & := 1 & \quad\quad k \geq N_\alpha \\ 
Y_k(\eta_0) & :=   \left(\frac{\eps K_\alpha \eta_0}{(k+1)^3}\right)\cdots \left(\frac{\eps K_\alpha \eta_0}{(N_\alpha-1)^3} \right)\left(\frac{\eps K_\alpha \eta_0}{N_\alpha^3} \right)  & \quad\quad 1 \leq k < N_\alpha, 
\end{align*}
then the following holds, 
\begin{align*}
\abs{f_{++}^H(t,k,\eta)} & \lesssim_{\alpha,\kappa'',\kappa'}     \epsilon' Y_k(\eta_0) e^{-\kappa'\abs{\eta-\eta_0}} e^{-k_0^{-1}\abs{k}} \\ 
\abs{\rho_{++}^H(t,k,\eta)} & \lesssim_{\alpha,\kappa'',\kappa'}  \epsilon' Y_k(\eta_0) e^{-\kappa''\abs{kt-\eta_0}} e^{-k_0^{-1}\abs{k}}; 
\end{align*}
analogous bounds hold also for $f^H_{--}$ and similar bounds hold without any $Y_k$ for $f^H_{-+}$ and $f^H_{+-}$. 
\end{lemma}

We next improve the above estimate near the critical times by replacing $Y_k(\eta_0)$ by a more precise growth factor. 
We record the following identity: let $\lambda \in (0,1)$ and $y \in \Real$ be arbitrary, 
\begin{align}
\int_{\Real} e^{-\lambda\abs{x} - \abs{x-y}} dx  = \frac{2}{1-\lambda^2}e^{-\lambda\abs{y}} - \frac{2\lambda}{1-\lambda^2} e^{-\abs{y}}. \label{eq:exactint}  
\end{align}

\begin{lemma}[Improved upper bounds near the critical times] \label{lem:ElecImpUp}
Let $\kappa < \kappa'$ be arbitrary. 
Define 
\begin{align}
K_f(\kappa) :=  \frac{2}{1-\kappa}, 
\end{align}
and the associated the growth factor
\begin{subequations} \label{def:Cketa0opt}
\begin{align}
N_f & = \textup{Floor}((K_f\epsilon)^{1/3} \eta_0^{1/3}), \\ 
X_k(\eta_0) & := 1 & \quad\quad k \geq N_f,  \\ 
X_k(\eta_0) & :=   \left(\frac{\eps K_f \eta_0}{(k+1)^3}\right)\cdots \left(\frac{\eps K_f \eta_0}{(N_f-1)^3} \right)\left(\frac{\eps K_f \eta_0}{N_f^3} \right)  & \quad\quad 1 \leq k < N_f.
\end{align}
\end{subequations}
There holds for $1 \leq k \leq k_0$, 
\begin{align*}
\abs{\widehat{\rho}^H_{++}(t,k)} \lesssim_{\alpha,\kappa,\kappa',\kappa''} \epsilon' X_k(\eta_0) e^{-\kappa\abs{\eta_0-kt}}, 
\end{align*}
and 
\begin{align*}
\abs{\widehat{f^H_{++}}(t,k,\eta)} \lesssim_{\alpha,\kappa,\kappa',\kappa''} \epsilon' X_k(\eta_0) e^{-\kappa\abs{\eta_0-\eta}}. 
\end{align*}
This implies similar upper bounds on $f^H_{--}$. 
\end{lemma} 
\begin{proof} 
The argument is a variant of Lemma \ref{lem:AprxRhoUp}.
For the duration of the proof, we use: 
\begin{align*}
\rho := \rho^H_{++}
\end{align*}
As in the proof of Lemma \ref{lem:AprxRhoUp}, we consider the solution first on the time interval $[t_{in},T]$, where $T$ is the largest time such that the following holds for all $t \in [t_{in},T]$, for some small $\delta'$ and large $C$ chosen below depending only on universal constants, 
\begin{align}
\abs{\hat{\rho}(t,k)} & \leq (1+\delta') C \eps' X_k(\eta_0) e^{-\kappa\abs{\eta_0-kt}}. \label{ineq:bootElectroRho}
\end{align}
We propagate \eqref{ineq:bootElectroRho} to $t_\star = \eta_0$ with a bootstrap argument. In particular, we next prove that for $t \leq T$, then for $\eps$ sufficiently small, we can proof \eqref{ineq:bootElectroRho} with the $\delta'$ replaced $\frac{1}{2}\delta'$. 
As above, by \eqref{eq:rhoseciter} and Corollary \ref{cor:linsolve},  
\begin{align}
\abs{\hat{\rho}(t,k)}\frac{e^{\kappa\abs{\eta_0-kt}}}{X_k(\eta_0)}  & \leq \frac{e^{\kappa\abs{\eta_0-kt}}}{X_k(\eta_0)} \abs{\widehat{f^H_{++}}(t_{in},k,kt)} \nonumber \\   
& \quad  + \epsilon\sum_{\ell = k \pm 1} \int_{t_{in}}^t \frac{e^{\kappa\abs{\eta_0-kt}}}{X_k(\eta_0)} \abs{\widehat{\rho}(\tau,\ell)} \abs{\ell \widehat{W}(\ell) k(t-\tau)} e^{-\abs{kt - \ell \tau}} d\tau \nonumber \\ 
& \quad + \sqrt{\delta} \int_{t_{in}}^t e^{-\abs{k(t-\tau)}} \frac{e^{\kappa\abs{\eta_0-kt}}}{X_k(\eta_0)} \abs{\widehat{f^H_{++}}(t_{in},k,k\tau)} d\tau \nonumber \\  
& \quad + \eps \sqrt{\delta} \sum_{\ell = k \pm 1} \int_{t_{in}}^t e^{-\abs{k(t-\tau)}} \int_{t_{in}}^\tau \frac{e^{\kappa\abs{\eta_0-kt}}}{X_k(\eta_0)} \abs{ \widehat{\rho}(s,\ell) \ell \widehat{W}(\ell) k(\tau-s)} e^{-\abs{k\tau - \ell s}} ds d\tau \nonumber \\  
& = \sum_{i=1}^4 R_i. \label{def:RiAprxRhoPrec} 
\end{align}
As in Lemma \ref{lem:AprxRhoUp}, upper bounds consistent with an improvement to \eqref{ineq:bootElectroRho} are deduced on $R_1$ and $R_3$ from the bounds on the initial data (for $C$ large). 
Turn to $R_2$, the leading order term. 
We will make the resonant region slightly more precise: 
\begin{align}
R_{2} & = \epsilon \int_{t_{k+1}}^{t_k} \mathbf{1}_{\abs{\eta_0 - kt} < \delta' \eta_0/k} \frac{e^{\kappa\abs{\eta_0-kt}}}{X_k(\eta_0)} \abs{\widehat{\rho}(\tau,k+1)} \abs{(k+1)\widehat{W}(k+1)k \left(t - \tau\right)} e^{-\abs{kt - (k+1)\tau}} d\tau + R_{2;NR} \nonumber \\ 
& = R_{2;R}+R_{2;NR}.\label{def:R2decomp} 
\end{align}
Consider first the non-resonant region, which we sub-divide further: 
\begin{align*}
R_{2;NR} & = \epsilon \int_{t_{k+1}}^{t_k} \mathbf{1}_{\abs{\eta_0 - kt} \geq \delta' \eta_0/k}\frac{e^{\kappa\abs{\eta_0-kt}}\abs{\widehat{\rho}(\tau,k+1)}}{X_k(\eta_0)}\abs{(k+1)\widehat{W}(k+1)k \left(t - \tau\right)} e^{-\abs{kt - (k+1)\tau}} d\tau \\ 
& \quad + \epsilon \int_{(t_{in},t) \setminus I_{k,\eta_0}} \frac{e^{\kappa\abs{\eta_0-kt}} \abs{\widehat{\rho}(\tau,k+1)}}{X_k(\eta_0)}\abs{(k+1)\widehat{W}(k+1)k \left(t - \tau\right)} e^{-\abs{kt - (k+1)\tau}} d\tau \\ 
& \quad + \epsilon \int_{t_{in}}^t \frac{e^{\kappa\abs{\eta_0-kt}}}{X_k(\eta_0)} \abs{\widehat{\rho}(\tau,k-1)} \abs{(k-1) \widehat{W}(k-1) k(t-\tau)} e^{-\abs{kt - (k-1)\tau}} d\tau \\ 
& = \sum_{i=1}^3 R_{2;NR}^i. 
\end{align*}
For $R_{2;NR}^1$ we use Lemma \ref{lem:ElectroAprx_basic} to deduce for $\epsilon$ sufficiently small, 
\begin{align*}
R_{2;NR}^1 & \lesssim \epsilon \epsilon' Y_k(\eta_0)e^{\kappa\abs{\eta_0-kt}} \mathbf{1}_{\abs{\eta_0 - kt} \geq \delta' \eta_0/k} \int_{t_{k+1}}^{t_k}  \frac{\jap{\tau}}{k+1}e^{-\kappa''\abs{\eta_0-(k+1)\tau}} e^{-\abs{kt - (k+1)\tau}} d\tau \\  
& \lesssim \epsilon \epsilon' Y_k(\eta_0) \frac{\eta_0}{k^2} e^{(\kappa-\kappa'')\abs{\eta_0-kt}} \mathbf{1}_{\abs{\eta_0 - kt} \geq \delta' \eta_0/k} \int_{t_{k+1}}^{t_k} e^{-(1-\kappa'')\abs{kt - (k+1)\tau}} d\tau \\  
& \lesssim_{\kappa,\kappa'',\delta'} \epsilon \epsilon', 
\end{align*}
which is consistent with an improvement \eqref{ineq:bootElectroRho} for $\epsilon$ sufficiently small (fixing $\kappa$ ,$\kappa''$, and $\delta'$ first, and then choosing $\epsilon$ small). 
The treatment of $R_{2;NR}^2$ is similar as on the support of the integrand, $\abs{\eta_0 - (k+1)\tau} \gtrsim \frac{\eta_0}{k+1}$. Hence this term is omitted for brevity. 
The $R_{2;NR}^3$ term is also similarly bounded, using that $\abs{kt - (k-1)\tau} \geq \tau \geq t_{in} = \epsilon^{-q}$ on the support of the integrand.
We omit the details for brevity.   

Turn next to the resonant contributions in \eqref{def:R2decomp}. Using \eqref{ineq:bootElectroRho}, we have for some $C'$, 
\begin{align}
e^{-\kappa\abs{kt-\eta_0}} R_{2;R} & \leq \epsilon \epsilon' \frac{(1+\delta')C}{k+1} \frac{X_{k+1}(\eta_0)}{X_k(\eta_0)} \int_{t_{k+1}}^{t_k} \mathbf{1}_{\abs{\eta_0 - kt} < \delta' \eta_0/k} \left(kt - k\tau\right) e^{-\kappa\abs{\eta_0 - (k+1)\tau}} e^{-\abs{kt - (k+1)\tau}} d\tau \nonumber \\ 
& \leq \epsilon \epsilon' \frac{(1+\delta')C}{k+1} \frac{X_{k+1}(\eta_0)}{X_k(\eta_0)} \mathbf{1}_{\abs{\eta_0 - kt} < \delta' \eta_0/k} \int_{t_{k+1}}^{t_k} \left(\eta_0 - k\tau\right) e^{-\kappa\abs{\eta_0 - (k+1)\tau}} e^{-\abs{kt - (k+1)\tau}} d\tau \nonumber \\ 
& \quad + \delta'\epsilon C' \epsilon' \frac{(1+\delta')C}{k+1} \frac{\eta_0}{(k+1)^2} \frac{X_{k+1}(\eta_0)}{X_k(\eta_0)}e^{-\kappa\abs{kt-\eta_0}}.  \label{ineq:R2R}
\end{align}
Consider the integral in the leading order term: 
\begin{align}
\int_{t_{k+1}}^{t_k} \left(\eta_0 - k\tau\right) e^{-\kappa\abs{\eta_0 - (k+1)\tau}} e^{-\abs{kt - (k+1)\tau}} d\tau & = \int_{\Real} \left(\eta_0 - k\tau\right) e^{-\kappa\abs{\eta_0 - (k+1)\tau}} e^{-\abs{kt - (k+1)\tau}} d\tau \nonumber \\ & \quad - \int_{\Real \setminus I_{k+1,\eta_0}} \left(\eta_0 - k\tau\right) e^{-\kappa\abs{\eta_0 - (k+1)\tau}} e^{-\abs{kt - (k+1)\tau}} d\tau. \label{eq:R2Rleadint}
\end{align}
The second integral is bounded using that, since $\abs{\eta_0 - kt} < \delta'\eta_0/k$, there holds on the support of the integrand
\begin{align*}
\abs{kt-(k+1)\tau} \geq \abs{\eta_0 - (k+1)\tau} - \abs{kt-\eta_0} \gtrsim \frac{\eta_0}{k}, 
\end{align*} 
and hence for some universal $c > 0$, 
\begin{align*}
 \int_{\Real \setminus I^C_{k+1}} \left(\eta_0 - k\tau\right) e^{-\kappa\abs{\eta_0 - (k+1)\tau}} e^{-\abs{kt - (k+1)\tau}} d\tau \lesssim e^{-c\frac{\eta_0}{k}}, 
\end{align*}
which can be made arbitrarily small by choosing $\epsilon$ small, since $k \leq N_f$. 
Turn to the first term in \eqref{eq:R2Rleadint}.  By the identity \eqref{eq:exactint}, 
\begin{align*}
\int_{\Real} \left(\eta_0 - k\tau\right) e^{-\kappa\abs{\eta_0 - (k+1)\tau}} e^{-\abs{kt - (k+1)\tau}} d\tau & = \frac{1}{k+1}\int_{\Real} \left(\frac{\eta_0}{k+1} - \frac{k}{k+1}s\right) e^{-\kappa\abs{s}} e^{-\abs{kt - \eta_0 - s}} ds \\ 
& = \frac{\eta_0}{(k+1)^2}\left(\frac{2}{1-\kappa^2}e^{-\kappa\abs{kt-\eta_0}} - \frac{2\kappa}{1-\kappa^2}e^{-\abs{kt-\eta_0}}\right) \\    
& \quad - \frac{k}{(k+1)^2}\int_{\Real} s e^{-\kappa\abs{s}} e^{-\abs{kt - \eta_0 - s}} ds. 
\end{align*}
Using that we are considering $\abs{kt-\eta_0} \leq \delta' \frac{\eta_0}{k}$, the latter integral is bounded via: 
\begin{align*}
\frac{k}{(k+1)^2}\int_{\Real} s e^{-\kappa\abs{s}} e^{-\abs{kt - \eta_0 - s}} ds & \leq \frac{k}{(k+1)^2}e^{-\kappa\abs{\eta_0 - kt}}\int_{\Real} s e^{-(1-\kappa)\abs{kt - \eta_0 - s}} ds \\ 
& \lesssim \delta'\frac{\eta_0}{(k+1)^2} e^{-\kappa\abs{\eta_0-kt}}. 
\end{align*}
Therefore, for $\delta'$ sufficiently small depending only on $\kappa$, we have 
\begin{align*}
\int_{\Real} \left(\eta_0 - k\tau\right) e^{-\kappa\abs{\eta_0 - k\tau}} e^{-\abs{kt - (k+1)\tau}} d\tau \leq \frac{\eta_0}{(k+1)^2}\left(\frac{2}{1-\kappa^{3/2}}\right)e^{-\kappa\abs{kt-\eta_0}}. 
\end{align*}
Putting everything together, for $\delta'$ sufficiently small depending only on $\kappa$, we have from \eqref{ineq:R2R},  
\begin{align*}
R_{2;R} & \leq C\epsilon' (1+\delta')\frac{\epsilon \eta_0}{(k+1)^3}\left(\frac{2}{1-\kappa^{5/4}}\right)\frac{X_{k+1}(\eta_0)}{X_k(\eta_0)}\\  
& = \frac{(1+\delta')(1-\kappa)}{(1+\frac{1}{2}\delta')(1-\kappa^{5/4})} C\epsilon'(1+\frac{1}{2}\delta').  
\end{align*}
Therefore, by choosing 
\begin{align*}
\delta' < \frac{\kappa - \kappa^{5/4}}{1 - \kappa^{5/4} - 2\kappa}, 
\end{align*} 
we can ensure that $R_{2;R}$ satisfies the desired improvement to \eqref{ineq:bootElectroRho} with some room to spare (depending only on $\kappa$) for the other error terms. 

Finally, turn to $R_4$ in \eqref{def:RiAprxRhoPrec}. 
We have 
\begin{align*}
R_4 \leq \eps \sqrt{\delta} \sum_{\ell = k \pm 1} \int_{t_{in}}^t e^{-(1-\kappa)\abs{k(t-\tau)}} \left( \int_{t_{in}}^\tau \frac{e^{\kappa\abs{\eta_0-k\tau}}}{X_k(\eta_0)} \abs{ \hat{\rho}(s,\ell) \ell \widehat{W}(\ell) k(\tau-s)} e^{-\abs{k\tau - \ell s}} ds \right) d\tau. 
\end{align*}
By repeating the same arguments we made for $R_{2;R}$, we have 
\begin{align*}
R_4 & \leq \sqrt{\delta} \int_{t_{in}}^t e^{-(1-\kappa)\abs{k(t-\tau)}} C\epsilon' (1+\delta')\frac{\epsilon \eta_0}{(k+1)^3}\left(\frac{2}{1-\kappa^{5/4}}\right)\frac{X_{k+1}(\eta_0)}{X_k(\eta_0)} e^{-\kappa\abs{kt-\eta_0}}  d\tau. 
\end{align*}
Hence, for $\delta = \epsilon^p$ small relative to $\delta'$, this is consistent with the stated improvement to \eqref{ineq:bootElectroRho}.  
\end{proof} 

With the enhanced upper bound near the critical times, we may now deduce a significantly more precise lower bound than we derived in the gravitational case.
\begin{proposition}[Instability of second iterate system in the electrostatic case] \label{prop:blowup_electro}
Define the constant
\begin{align}
K_f'(\eta_0) :=  \frac{2}{1+\alpha}. 
\end{align}
Define the growth factor
\begin{subequations} \label{def:Cketa0opt2}
\begin{align}
k_0 & := \textup{Floor}\left( (K_f'\epsilon)^{1/3} \eta_0^{1/3}\right), \\ 
X_k'(\eta_0) & := 1 & \quad\quad k \geq k_0,  \\ 
X_k'(\eta_0) & :=   \left(\frac{\eps K_f' \eta_0}{(k+1)^3}\right)\cdots \left(\frac{\eps K_f' \eta_0}{(k_0-1)^3} \right)\left(\frac{\eps K_f' \eta_0}{k_0^3} \right)  & \quad\quad 1 \leq k < k_0. 
\end{align}
\end{subequations}
For all $\kappa< \alpha < 1$, then for all $\eps$ sufficiently small (depending on $\kappa$ and $\alpha$), there holds 
\begin{align*}
\abs{\widehat{\rho^H_{++}}(t,k)} \geq \epsilon' X'_k(\eta_0) e^{-\alpha\abs{\eta_0-kt}} - \epsilon'\delta^{1/4}X_k(\eta_0)e^{-\kappa\abs{\eta_0-kt}},
\end{align*}
and 
\begin{align*}
\abs{\widehat{f^H_{++}}(t,k,\eta)} \geq \epsilon' X'_k(\eta_0) e^{-\alpha\abs{\eta_0-\eta}} - \epsilon'\delta^{1/4} X_k(\eta_0)e^{-\kappa\abs{\eta_0-\eta}},
\end{align*}
where $X_k$ is given in \eqref{def:Cketa0opt} above. An analogous lower bound holds also for $f_{--}^H$ by symmetry. 
\end{proposition} 
\begin{proof} 
The proof is a variation of that used in Proposition \ref{prop:blowup}. 
The short-time lower bound is the same as in Lemma \ref{lem:lowst}; we record the result for completeness but omit the proof for brevity. 
\begin{lemma}[Short time] \label{lem:lowstelc} 
For all $\gamma > 0$ and $t_{in} < t < t_{k_0}$, there holds for all $k,\eta$, 
\begin{align*}
\widehat{f^H_{++}}(t,k_0,\eta) \geq \epsilon' e^{-\alpha\abs{\eta-\eta_0}} - \epsilon^\gamma \epsilon' e^{-\kappa\abs{\eta-\eta_0}}. 
\end{align*}
\end{lemma} 
Next we will propagate lower bounds through the critical times. 
For notational simplicity, for the duration of the proof we simply write $f := f^H_{++}$ and $\rho := \rho^H_{++}$.  
Assume that the following lower bound holds for $3 \leq k \leq k_0$: 
\begin{align}
\abs{\widehat{f}(t_k,k,\eta)} \geq X_k'(\eta_0)\epsilon' e^{-\alpha\abs{\eta-\eta_0}} - X_k(\eta_0) \delta^{1/4} \epsilon' e^{-\kappa\abs{\eta-\eta_0}}. \label{ineq:blowlwbdinductelec}
\end{align}
Lemma \ref{lem:lowstelc} implies that this holds for $k = k_0$. 
Proposition \ref{prop:blowup_electro} then reduces to proving 
\begin{align}
\abs{\widehat{f}(t_{k-1},k-1,\eta)} \geq X_{k-1}'(\eta_0)\epsilon' e^{-\alpha\abs{\eta-\eta_0}} - X_{k-1}(\eta_0)\delta^{1/4} \epsilon' e^{-\kappa\abs{\eta-\eta_0}}. \label{ineq:ftkm1Low_elec}
\end{align}
The $k=2$ case is analogous and is omitted for brevity (though we are interested in $\widehat{f}(t_\star,1,\eta)$ in this case; see Proposition \ref{prop:blowup} for similar details).  
Unlike the gravitational case, through each critical time, the Fourier transform near $(k,\eta_0)$ flips sign. 
The first step is to prove that near the critical time, the critical density is large. 
By Lemma \ref{lem:basicvolt}, over $I_{k,\eta_0}$ the critical density mode satisfies 
\begin{align*}
\widehat{\rho}(t,k) & = \widehat{f}(t_k,k,kt) - \sqrt{\delta} \int_{t_k}^t  \sin(\sqrt{\delta}(t-\tau)) e^{-\abs{k}(t-\tau)} \widehat{f}(t_k,k,k\tau) d\tau \\ 
& \quad - \epsilon\sum_{\ell = k\pm 1}\int_{t_k}^t \widehat{\rho}(\tau,\ell) \ell \widehat{W}(\ell) k(t-\tau) e^{-\abs{kt - \ell \tau}} d\tau  \\ 
& \quad - \epsilon\sqrt{\delta}\sum_{\ell = k\pm 1}\int_{t_k}^t\sin(\sqrt{\delta}(t-\tau)) e^{-\abs{k}(t-\tau)}  \int_{t_k}^\tau \widehat{\rho}(s,\ell) \ell \widehat{W}(\ell) k(\tau-s) e^{-\abs{k\tau - \ell s}} ds d\tau \\ 
&=\sum_{j=1}^4 I_j. 
\end{align*} 
The main difficulty here is that the second term, $I_2$, which arises from the effect of $f^0$, has the opposite sign from the leading $I_1$ term.  
This removes the monotonicity available in the gravitational case, and hence we need to use the upper bound to control $I_2$, rather than the lower bound \eqref{ineq:blowlwbdinductelec}. 
Suppose that $\hat{f}(t_k,k,kt)$ is positive near $kt \approx \eta_0$; the negative case is treated analogously.
In the positive case, we are interested in getting a lower bound on the linear term $I_2$.  
Using the improved upper bound in Lemma \ref{lem:ElecImpUp}, 
\begin{align}
\abs{I_2} & \leq \sqrt{\delta} \int_{t_k}^t e^{-\abs{k}(t-\tau)} \abs{\widehat{f}(t_k,k,k\tau)} d\tau \nonumber \\ 
& \lesssim \sqrt{\delta} \epsilon' X_k(\eta_0) \int_{t_k}^t e^{-\abs{k}(t-\tau)} e^{-\kappa \abs{\eta_0 - k\tau}} d\tau \nonumber \\ 
& \lesssim  \sqrt{\delta} \epsilon' X_k(\eta_0) e^{-\kappa \abs{\eta_0 - kt}}. \label{ineq:I2rhoElec}
\end{align}
The $I_3$ and $I_4$ terms involve only non-critical frequencies, and are hence easily bounded. 
For example, in the integrand in $I_3$, we have $\abs{kt - \ell \tau} \gtrsim \tau \geq t_{in} = \epsilon^{-q}$. 
Hence, for $\epsilon$ sufficiently small, it is straightforward to obtain an estimate such as 
\begin{align}
\abs{I_3} \lesssim \epsilon  \epsilon' e^{-\kappa \abs{\eta_0 - kt}}. \label{ineq:I3rhoElec}
\end{align}  
To treat $I_4$ we note that on the support of the integrand, either $\abs{k\tau - \ell s} \gtrsim \tau \geq t_{in} = \epsilon^{-q}$ or 
$\abs{k(t-\tau)} \gtrsim t \geq t_{in} = \epsilon^{-q}$ and hence a similar estimate is valid. 
Putting together \eqref{ineq:blowlwbdinductelec}, \eqref{ineq:I2rhoElec}, and \eqref{ineq:I3rhoElec} (and the corresponding estimate $I_4$). 
 deduce the following lower bound on the density for some $C > 0$: 
\begin{align}
\rho(t,k) \geq  X_k'(\eta_0)\epsilon' e^{-\alpha\abs{kt-\eta_0}} - \left(\delta^{1/4}+ C\sqrt{\delta}\right) X_k(\eta_0) \epsilon' e^{-\kappa\abs{kt-\eta_0}}. \label{ineq:rhoClow_elec}
\end{align}
Turn next to the distribution function for $f(t_{k-1},k-1,\eta)$: 
\begin{align*}
\widehat{f}(t_{k-1},k-1,\eta) & = -\epsilon \int_{t_k}^{t_{k-1}} \widehat{\rho}(\tau,k) k \widehat{W}(k)(\eta-(k-1)\tau) e^{-\abs{\eta-k\tau}} d\tau \\ 
& \quad + \hat{f}(t_k,k-1,\eta) - \delta \int_{t_k}^{t_{k-1}} \widehat{\rho}(\tau,k-1)(k-1)\widehat{W}(k-1)(\eta-(k-1)\tau)e^{-\abs{\eta-(k-1)\tau}} d\tau \\ 
& \quad - \epsilon\int_{t_k}^{t_{k-1}} \hat{\rho}(\tau,k-2)(k-2) \widehat{W}(k-2)(\eta-(k-1)\tau) e^{-\abs{\eta-(k-2)\tau}} d\tau \\ 
& = \mathcal{I}_C  + \sum_{j=1}^3\mathcal{E}_j.
\end{align*}
As in the gravitational case, the growth comes from the leading term and the others are error.
The error terms $\mathcal{E}_j$ involve only non-critical contributions and are hence easily estimated in absolute value to be consistent with \eqref{ineq:ftkm1Low_elec} as in Proposition \ref{prop:blowup} above; we omit the proof for brevity as it is the same as in the gravitational case. 

We next divide the contribution from the critical mode more precisely into frequencies near $\eta_0$, and those far away which should not matter. 
Hence, for any $\delta' > 0$, set 
\begin{align*}
\mathcal{I}_C & = -\epsilon \int_{t_k}^{t_{k-1}}\left(\mathbf{1}_{\abs{\eta-\eta_0} < \delta'\frac{\eta_0}{k}} + \mathbf{1}_{\abs{\eta-\eta_0} \geq \delta'\frac{\eta_0}{k}}\right) \widehat{\rho}(\tau,k) k \widehat{W}(k)(\eta-(k-1)\tau) e^{-\abs{\eta-k\tau}} d\tau \\ 
& = \mathcal{I}_C^\delta + \mathcal{I}_C^{NR}. 
\end{align*}
Similar to arguments in Lemmas \ref{lem:AprxRhoUp} and \eqref{lem:ElecImpUp}, the $\mathcal{I}_C^{NR}$ terms can be made arbitrarily small, and hence this treatment is omitted for brevity (we will need to choose $\epsilon$ such that these are small relative to $\alpha$ and $\kappa$).  

In the former case, it follows that the kernel is strictly negative, and hence by the lower bound \eqref{ineq:rhoClow_elec}, 
\begin{align}
\mathcal{I}_C^\delta & \leq -\epsilon k \widehat{W}(k) \mathbf{1}_{\abs{\eta-\eta_0} < \delta'\frac{\eta_0}{k}} X_k'(\eta_0)\epsilon' \int_{t_k}^{t_{k-1}} (\eta-(k-1)\tau) e^{-\alpha\abs{\eta-k\tau}}  e^{-\abs{k\tau-\eta_0}} d\tau \nonumber  \\ & \quad + \epsilon k \widehat{W}(k) (\delta^{1/4}+C\sqrt{\delta}) X_k(\eta_0) \epsilon' \mathbf{1}_{\abs{\eta-\eta_0} < \delta'\frac{\eta_0}{k}} \int_{t_k}^{t_{k-1}} (\eta-(k-1)\tau) e^{-\kappa\abs{k\tau-\eta_0}}e^{-\abs{\eta-k\tau}} d\tau \nonumber \\ 
& = \mathcal{I}_{C}^{\delta,0} + \mathcal{I}_{C}^{\delta,1}. \label{def:Icdel0}
\end{align}
The last term is error, but it must be dealt with in a manner similar to the leading order term in the proof of Lemma \ref{lem:ElecImpUp}. 
We have the following for some $C' > 0$, analogous to methods in the proof of Lemma \ref{lem:ElecImpUp},  
\begin{align*}
\abs{\mathcal{I}_{C}^{\delta,1}} & \leq \frac{\epsilon}{k} (\delta^{1/4}+C\sqrt{\delta}) X_k(\eta_0) \epsilon' \mathbf{1}_{\abs{\eta-\eta_0} < \delta'\frac{\eta_0}{k}} \int_{\Real} (\eta_0 -(k-1)\tau) e^{-\kappa\abs{k\tau-\eta_0}}e^{-\abs{\eta-k\tau}} d\tau \nonumber \\ 
& \quad + C'\delta' \frac{\epsilon \eta_0}{k^3} (\delta^{1/4}+\sqrt{\delta}) X_k(\eta_0) \epsilon' e^{-\kappa\abs{\eta-\eta_0}}.  
\end{align*}
Applying the identity \eqref{eq:exactint} similar to above, 
\begin{align*}
\int_{\Real} (\eta_0 -(k-1)\tau) e^{-\kappa\abs{k\tau-\eta_0}}e^{-\abs{\eta-k\tau}} d\tau & = \frac{1}{k}\int \left(\frac{\eta_0}{k} - \frac{k-1}{k}s\right) e^{-\kappa\abs{s}} e^{-\abs{\eta - \eta_0 - s}} ds \\ 
& = \frac{\eta_0}{k^2}\left(\frac{2}{1-\kappa^2}e^{-\kappa\abs{\eta-\eta_0}} - \frac{2\kappa}{1-\kappa^2}e^{-\abs{\eta-\eta_0}}\right) \\ 
& \quad - \frac{k-1}{k^2}\int s e^{-\kappa\abs{s}} e^{-\abs{\eta - \eta_0 - s}} ds. 
\end{align*}
The latter integral is bounded by the following, using that $\abs{\eta-\eta_0} \leq \delta' \frac{\eta_0}{k}$,  
\begin{align*}
\frac{k-1}{k^2}\int s e^{-\kappa\abs{s}} e^{-\abs{\eta - \eta_0 - s}} ds &  \lesssim \delta'\frac{\eta_0}{(k+1)^2} e^{-\kappa\abs{\eta_0-\eta}}. 
\end{align*}
Therefore, for $\delta'$ sufficiently small, we have, for some $C'$, 
\begin{align*}
\abs{\mathcal{I}_{C}^{\delta,1}} & \leq \frac{2\epsilon \eta_0}{k^3(1-\kappa^{2})} (\delta^{1/4}+C\sqrt{\delta}) X_k(\eta_0) \epsilon' + C' \delta' \frac{\epsilon \eta_0}{k^3} (\delta^{1/4}+C\sqrt{\delta}) X_k(\eta_0) \epsilon' e^{-\kappa\abs{\eta-\eta_0}}.  
\end{align*}
It follows that for $\delta$ and $\delta'$ sufficiently small, we have (recall \eqref{def:Cketa0opt}), 
\begin{align*}
\abs{\mathcal{I}_{C}^{\delta,1}} & \leq \frac{1}{(1+\kappa)^{1/2}} X_{k-1}(\eta_0) \delta^{1/4} \epsilon' e^{-\kappa\abs{\eta-\eta_0}},
\end{align*}
which is consistent with \eqref{ineq:blowlwbdinductelec}. Note that the other error terms must fit within the gap of $1 - \frac{1}{(1+\kappa)^{1/2}}$, which will necessitate choosing $\delta'$ and $\epsilon$ small.  

Consider now the leading order $\mathcal{I}_C^{\delta,0}$ term (recall \eqref{def:Icdel0}). 
By \eqref{eq:exactint}, we deduce on the support of the integrand that there holds for some $C' > 0$, (using that $\abs{\eta-\eta_0} \leq \delta' \eta_0/k$), 
\begin{align*}
\int_{t_k}^{t_{k-1}} (\eta-(k-1)\tau) e^{-\alpha\abs{\eta-k\tau}}  e^{-\abs{k\tau-\eta_0}} d\tau & \geq \eta_0\int_{\Real} e^{-\alpha\abs{\eta-k\tau}}  e^{-\abs{k\tau-\eta_0}} d\tau - C'\delta'\frac{\eta_0}{k^2} \\ 
& = \frac{\eta_0}{k^2}\left(\frac{2}{1-\alpha^2}e^{-\alpha\abs{\eta-\eta_0}} - \frac{2\alpha}{1-\alpha^2}e^{-\abs{\eta-\eta_0}}\right) - C\delta'\frac{\eta_0}{k^2}e^{-\alpha\abs{\eta-\eta_0}} \\ 
& \geq \frac{\eta_0}{k^2}\frac{2}{1+\alpha}e^{-\alpha\abs{\eta-\eta_0}} - C'\delta'\frac{\eta_0}{k^2}e^{-\alpha\abs{\eta-\eta_0}}.     
\end{align*}
Hence, for $\delta'$ sufficiently small, there holds 
\begin{align*}
\mathcal{I}_{C}^{\delta,0} \leq -  \frac{\epsilon \eta_0}{k^3} \left(\frac{2}{1+\alpha}\right) X_k'(\eta_0)\epsilon' e^{-\alpha\abs{\eta-\eta_0}} + C'' \delta' \frac{\epsilon \eta_0}{k^3}  X_k'(\eta_0)\epsilon' e^{-\alpha\abs{\eta-\eta_0}}
\end{align*}
By the definition of $X_{k-1}'$, this is consistent with the the desired lower bound \eqref{ineq:ftkm1Low_elec} (for $\delta'$ sufficiently small).  
As all the terms have been dealt with, we may propagate \eqref{ineq:ftkm1Low_elec} and complete the lemma. 
\end{proof} 

\subsection{Proof sketch in electrostatic case} 
Let $R$ be fixed. 
The lower bound in Proposition \ref{prop:blowup_electro} implies the following from Lemma \ref{lem:GrowthFact}, 
\begin{align}
\abs{\widehat{f^H_{++}}(t,1,\eta)} & \gtrsim \frac{\epsilon}{\jap{k_0,\eta_0}^{\sigma} (K_f'\eps \eta_0)^{1/2}}e^{3(K_f' \epsilon)^{1/3}\eta_0^{1/3}} e^{-\alpha\abs{\eta_0-\eta}} \nonumber \\ & \quad  - \frac{\epsilon \delta^{1/4}}{\jap{k_0,\eta_0}^{\sigma} (K_f \eps \eta_0)^{1/2}}e^{3(K_f \epsilon)^{1/3}\eta_0^{1/3}} e^{-\kappa\abs{\eta_0-\eta}}. \label{ineq:lwbdelec}
\end{align}
We choose $\eta_0$ to satisfy
\begin{align*}
\frac{\epsilon}{(K_f'\eps \eta_0)^{1/2}}e^{3(K_f' \epsilon)^{1/3}\eta_0^{1/3}} = \jap{k_0,\eta_0}^{R}. 
\end{align*}
In order for \eqref{ineq:lwbdelec} to yield a useful lower bound for $\abs{\eta-\eta_0} < \alpha^{-1}$, we will need 
\begin{align*}
\frac{1}{(K_f'\eps \eta_0)^{1/2}}e^{3(K_f' \epsilon)^{1/3}\eta_0^{1/3}} \gg \frac{\delta^{1/4}}{(K_f \eps \eta_0)^{1/2}}e^{3(K_f \epsilon)^{1/3}\eta_0^{1/3}}. 
\end{align*}
Using the definition of $\eta_0$, this becomes the requirement: 
\begin{align*}
\frac{\jap{k_0,\eta_0}^{R}}{\epsilon} \gg \frac{\delta^{1/4}}{(K_f \eps \eta_0)^{1/2}} \left( \frac{(K_f'\eps \eta_0)^{1/2}\jap{k_0,\eta_0}^{R}}{\eps} \right)^{\left(\frac{K_f}{K_f'}\right)^{1/3}}. 
\end{align*}
Recall that $\delta^{1/4} = \eps^{p/4}$ for some $p \in (0,1)$. 
We can choose $\alpha$ and $\kappa$ in order to make $K_f (K_f')^{-1}$ arbitrarily close to one, and therefore for any fixed $p$ and $R$, we can subsequently guarantee this condition by choosing $1 \gg \alpha > \kappa > 0$ (depending only on $p$ and $R$) and then choosing $\eps$ small accordingly. 
It follows that we have the analogue of \eqref{ineq:fElowbd} in the electrostatic case. 
Theorem \ref{thm:main} will follow after the analogue of Proposition \ref{prop:stabg} is proved.  
However, the proof of Proposition \ref{prop:stabg} will not really be affected. 
We may first fix $p$ and $R$, then $\kappa$ and $\alpha$, then $K$, then $\sigma$ (depending on $R$, $\alpha$, and $K$ through $a$ via requirements such as \eqref{ineq:siglargeLRH}) and then finally $\epsilon$ small with respect to everything. 
Hence, after slightly more careful parameter tuning, Theorem \ref{thm:main} follows also in the electrostatic case.

\section{Proof sketch for Theorem \ref{thm:optimal}} \label{sec:thmopt}
In this section we briefly sketch the proof of Theorem \ref{thm:optimal}. 
We need only to apply the scheme of \cite{BMM13} with the norm $A$ in order to handle the echoes in an optimal way (although $G$ adds additional difficulties that must also be dealt with -- in particular the commutator estimates employed in e.g. \S\ref{sec:LRTMidg} above). 
We propagate the same estimates as in \cite{BMM13}, except with $A$: 
\begin{subequations}\label{def:bootg2} 
\begin{align}
\norm{\jap{v}\jap{\grad}^{3}A(t)f}_{L^2} & \leq 8\epsilon \jap{t}^{5}, \label{ineq:Hig2} \\ 
\norm{\jap{v}\jap{\grad}^2 A(t) \rho}_{L^2_t L^2} &  \leq 8\epsilon, \label{ineq:HiRho2} \\ 
\norm{\jap{v}A(t,\grad)f}_{L^2} & \leq 8\epsilon,   \label{ineq:Midg2}
\end{align} 
\end{subequations} 
where we set $\beta > 3/2$ fixed and arbitrary.   
To get intuition for why such a scheme follows easily, consider a paraproduct decomposition of the nonlinear term:
\begin{align*}
\partial_t f + E(t,z+tv)\partial_v f^0 + NL_{LH} + NL_{HL} + NL_{\mathcal{R}} = 0, 
\end{align*}
where 
\begin{align*}
NL_{LH} & =         \sum_{N \in 2^\Integers} E(t,z+tv)_{<N} (\partial_v-t\partial_z)f_N,  \\ 
NL_{HL} & =         \sum_{N \in 2^\Integers} E(t,z+tv)_{N} (\partial_v-t\partial_z)f_{<N}, \\ 
NL_{\mathcal{R}} & =  \sum_{N \in 2^\Integers} E(t,z+tv)_{N} (\partial_v-t\partial_z)f_{\sim N} = 0. 
\end{align*}
The techniques used to treat the term $E^L(t,z + tv) (\partial_v - t\partial_z) f$ in \S\ref{sec:g} will easily adapt apply to treat $NL_{LH}$, 
and techniques used to treat the term  $E(t,z + tv) (\partial_v - t\partial_z) f^L$ in \S\ref{sec:g} will easily adapt apply to treat $NL_{HL}$ (the remainder term is much easier). 
There are a few minor adjustments necessary. First, in the proof of \eqref{ineq:HiRho2}, we will need Remark \ref{rmk:genwres}, as now the ``low frequency reaction term'' will involve interactions with a variety of low spatial modes. 
Second, when treating $NL_{LH}$ in the proofs of \eqref{ineq:Midg2} and \eqref{ineq:Hig2}, the commutator estimate involving $G$ will give terms of the form
\begin{align*}
\frac{\epsilon^{1/3}}{K^{2/3}\jap{t}^2} \norm{\jap{\grad}^{1/6} A g}_2^2, 
\end{align*} 
which require choosing $K$ large in order to absorb with the corresponding $CK_\mu$ in the analogue of e.g. \eqref{ineq:gMidEnergEst}. 
We omit the details as they are an easy variant of the methods of \cite{BMM13} mixed with ideas from Proposition \ref{prop:stabg}.  

\appendix
\section{Properties of $w$ and $G$} \label{sec:ApxNorms}
This section is similar to an analogous section in \cite{BM13}, however, we need some refinements due to the sensitivity on $K$ and $\epsilon$. 
As the modifications are straightforward, we provide brief sketches. 

We first deduce properties on $\tilde w$ and extend them easily to $w$ from there. 
The first lemma confirms that the growth of $\tilde w$ is similar that of the growth factors in Lemma \ref{lem:GrowthFact}. 

\begin{lemma} \label{lem:growthtildw}
Lemma \ref{lem:growthw} holds with $w$ replaced by $\tilde{w}$. 
\end{lemma} 
\begin{proof} 
Without loss of generality, assume for simplicity that $\eta > 0$. 
Counting the growth over each interval implied by \eqref{def:tildw} gives the exact formula: 
\begin{align*} 
 \frac{1}{\tilde{w}(0,\eta)} & =\left(   \frac{K\epsilon\eta}{N^3}  \right)^2  \left(    \frac{K \epsilon \eta}{(N-1)^3}  \right)^2 ...
  \left(    \frac{K \epsilon \eta}{1^3}  \right)^2 = \left[\frac{(K\epsilon\eta)^{N}} { (N!)^3 }\right]^2.    
\end{align*}
The result then follows as in Lemma \ref{lem:GrowthFact}. 
\end{proof} 


Next we prove the following lemma, the analogue of [Lemma 3.5; \cite{BM13}].  
\begin{lemma} \label{lem:tildewcomp}
There exists a universal $\tilde{r} > 0$ such that the followings holds (with implicit constant independent of $K$ and $\epsilon$), 
\begin{align*}
\frac{\tilde{w}(t,\eta)}{\tilde{w}(t,\xi)} & \lesssim e^{\tilde{r}(K\epsilon)^{1/3}\abs{\eta-\xi}^{1/3}} \leq e^{\tilde{r}(K\epsilon)^{1/3}\jap{\eta-\xi}^{1/3}}. 
\end{align*}
\end{lemma}
\begin{proof} 
Switching the roles of $\xi$ and $\eta$, we may assume without loss of generality that $|\xi| \leq |\eta|$
and instead prove 
\begin{align}
e^{-\tilde{r}(K\epsilon)^{1/3}\abs{\eta - \xi}^{1/3}} \lesssim \frac{\tilde{w}(t,\xi)}{\tilde{w}(t,\eta)} \lesssim e^{\tilde{r} (K\epsilon)^{1/3} \abs{\eta - \xi}^{1/3}}. \label{ineq:twFreq} 
\end{align}
As in \cite{BM13}, we may reduce to the case $\max( t_{N(\xi),\xi} , t_{N(\eta),\eta})  \leq  t  \leq 2\xi \leq 2\eta$. 
Define $j$ and $n$ such that  $t \in I_{n,\eta}$ and $t \in I_{j,\xi}$. It follows  that $n \approx j \leq n$.
We proceed case-by-case depending on the relationship between $j$ and $n$.  

\noindent
\textbf{Case $j=n$:}\\ 
First assume that $t \in I_{n,\eta}^R \cap I_{n,\xi}^R$. 
In this case, from \eqref{def:tildw} we have  
\begin{align*}
\frac{\tilde{w}(t,\xi)}{\tilde{w}(t,\eta)} & = \left(\frac{\eta}{\xi}\right)^{2n}\frac{1 + a_{n,\xi}\frac{K\epsilon}{n}\abs{t-\frac{\xi}{n}}}{1 + a_{n,\eta}\frac{K\epsilon}{n}\abs{t-\frac{\eta}{n}}}.  
\end{align*}
First notice that for some $c$ which does not depend on $K$, $\epsilon$, $\xi$, or $\eta$, there holds, 
\begin{align}
1 \leq \left(\frac{\eta}{\xi}\right)^{2n} = \left(1 + \frac{\eta-\xi}{\xi}\right)^{2n} \leq \left(1 + \frac{\eta-\xi}{\xi}\right)^{2(K\epsilon \xi)^{1/3}} \leq e^{c(K\epsilon (\eta-\xi))^{1/3}}. \label{ineq:etxiratctrl}
\end{align}
Second, write 
\begin{align*}
\frac{1 + a_{n,\xi}\frac{K\epsilon}{n}\abs{t-\frac{\xi}{n}}}{1 + a_{n,\eta}\frac{K\epsilon}{n}\abs{t-\frac{\eta}{n}}} & \leq 1 + a_{k,\xi}\frac{K\epsilon}{n}\abs{\abs{t-\frac{\xi}{n}} - \abs{t-\frac{\eta}{n}}} + \frac{K\epsilon}{n}\abs{t-\frac{\xi}{n}}\abs{a_{n,\eta} - a_{n,\xi}} \\ 
& \lesssim 1 + \frac{K\epsilon}{n^2}\abs{\eta-\xi}.
\end{align*}
Note that the implicit constant is independent of $K$, $\epsilon$, and $n$. Similarly, 
\begin{align*}
\frac{1}{1 + \frac{K\epsilon}{n^2}\abs{\eta-\xi}} \lesssim \frac{1 + a_{k,\xi}\frac{K\epsilon}{n}\abs{t-\frac{\xi}{k}}}{1 + a_{k,\eta}\frac{K\epsilon}{k}\abs{t-\frac{\eta}{k}}}. 
\end{align*}
Therefore \eqref{ineq:twFreq} follows, completing the case $t \in I_{n,\eta}^R \cap I_{n,\xi}^R$.
The remaining cases follow analogously by adapting techniques from \cite{BM13} in a similar manner; the details are omitted for brevity.   
\end{proof} 
Lemma \ref{lem:tildewcomp} and the definition of $w$ in \eqref{def:wconv} immediately imply the following.  
\begin{lemma} \label{lem:wapproxtildw} 
The following holds for all $\eta,t$ with implicit constant independent of $K$ and $\epsilon$, 
\begin{align*}
w(t,\eta) \approx \tilde{w}(t,\eta). 
\end{align*}
\end{lemma} 

Lemma \ref{lem:wapproxtildw} then implies that Lemma \ref{lem:growthtildw} proves Lemma \ref{lem:growthw} and Lemma \ref{lem:tildewcomp} proves an analogous statement about $w$. 

Next, we prove Lemma \ref{lem:resw}. 
\begin{proof}[Lemma \ref{lem:resw}]
By Lemma \ref{lem:wapproxtildw} it suffices to prove Lemma \ref{lem:resw} with $w$ replaced by $\tilde{w}$. 
Suppose that $\tau \in I_{k+1,kt}^R$. Then, 
\begin{align*}
\frac{\tilde{w}(\tau,kt)}{\tilde{w}(t,kt)} & = \frac{(k+1)^3}{K\epsilon k t}\left(1 + a_{k+1,kt}\frac{K\epsilon}{\abs{k}+1}\abs{\tau-\frac{kt}{k+1}}\right)\frac{\tilde{w}(t_{k},kt)}{\tilde{w}(t_{k-1},kt)} \frac{K \eps t}{ k^2} \\ 
& \lesssim \left(\frac{k^2}{K \eps t}\right)^{2} \left(1 + \frac{K\eps}{k^2}\abs{(k+1)\tau - kt}\right). 
\end{align*}
Suppose that $\tau \in I_{k+1,kt}^L$. 
Then,  
\begin{align*}
\frac{\tilde{w}(\tau,kt)}{\tilde{w}(t,kt)} & = \left(1 + b_{k+1,kt}\frac{K\epsilon}{\abs{k}+1}\abs{\tau-\frac{kt}{k+1}}\right)^{-1}\frac{\tilde{w}(\frac{kt}{k+1},kt)}{\tilde{w}(t,kt)} \lesssim \left(\frac{\abs{k}^2}{K \eps t}\right)^2, 
\end{align*}
which completes the proof. 
\end{proof} 

Next, we prove Lemma \ref{lem:Gcomp}. 
\begin{proof}[Lemma \ref{lem:Gcomp}] 
First, note that 
\begin{align*} 
\frac{G(t,k,\eta)}{G(t,k,\xi)} \leq \frac{w(t,\xi)}{w(t,\eta)}e^{r(K\epsilon)^{1/3}\abs{\eta-\xi}^{1/3}} + e^{r(K\epsilon)^{1/3}\abs{k-\ell}^{1/3}}, 
\end{align*}
from which the result follows by Lemmas \ref{lem:wapproxtildw} and \ref{lem:tildewcomp}. 
\end{proof} 

Next, we prove Lemma \ref{lem:commG}; the proof is a variant of [Lemma 3.7; \cite{BM13}]. 
\begin{proof}[Lemma \ref{lem:commG}]
First, suppose we are in the case $t < 2(K\epsilon)^{-1/3}\min(\abs{\eta}^{2/3},\abs{\xi}^{2/3})$. 
Due to the assumption on $t$, $G(t,k,\eta) = G(0,k,\eta)$ and $G(t,\ell,\xi) = G(0,\ell,\xi)$. 
If $\abs{\eta}^{2/3} + \abs{\xi}^{2/3} + \abs{k}^{2/3} + \abs{\ell}^{2/3} \lesssim (K\epsilon)^{-2/3}\jap{k-\ell,\eta-\xi}$, then Lemma \ref{lem:commG} follows from Lemma \ref{lem:Gcomp} and so without loss of generality we may assume that 
\begin{align}
\jap{k-\ell,\eta-\xi} \leq \frac{(K\epsilon)^{2/3}}{1000}\left(\abs{\eta}^{2/3} + \abs{\xi}^{2/3} + \abs{k}^{2/3} + {\ell}^{2/3}\right). \label{ineq:100ctrl}
\end{align}
Begin by applying
\begin{align}
\abs{\frac{G(t,k,\eta)}{G(t,\ell,\xi)}- 1} & \leq \abs{\frac{e^{r(K\epsilon)^{1/3}\jap{\eta}^{1/3}} (w(t,\eta))^{-1} - e^{r(K\epsilon)^{1/3}\jap{\xi}^{1/3}} (w(t,\xi))^{-1}}{e^{r(K\epsilon)^{1/3}\jap{\xi}^{1/3}} (w(t,\xi))^{-1} + e^{r(K\epsilon)^{1/3}\jap{\ell}^{1/3}}}} \nonumber \\ & \quad + \abs{\frac{e^{r(K\epsilon)^{1/3}\jap{k}^{1/3}} - e^{r(K\epsilon)^{1/3}\jap{\ell}^{1/3}}}{e^{r(K\epsilon)^{1/3}\jap{\xi}^{1/3}} (w(t,\xi))^{-1} + e^{r(K\epsilon)^{1/3}\jap{\ell}^{1/3}}}} \nonumber \\ 
& = T_\eta + T_k.\label{def:TkTeta} 
\end{align}
\textbf{Case $\frac{1}{10}\abs{k,\ell} \leq \abs{\eta,\xi} \leq 10\abs{k,\ell}$:} \\ 
\noindent
For $T_k$ we have the following by $\abs{e^x - 1} \leq xe^x$, 
\begin{align}
T_k \leq \abs{e^{r(K\epsilon)^{1/3}(\jap{k}^{1/3} - \jap{\ell}^{1/3})} - 1}  
&  \lesssim \frac{(K\epsilon)^{1/3}\jap{k-\ell,\eta-\xi}}{\jap{k,\ell}^{2/3} + \jap{\eta,\xi}^{2/3}}e^{r(K\epsilon)^{1/3}\jap{k-\ell}^{1/3}}, \label{ineq:TkCase1}
\end{align}
which suffices.
Consider next $T_\eta$, 
\begin{align}
T_\eta 
& \leq \frac{w(0,\xi)}{w(0,\eta)}\abs{e^{r(K\epsilon)^{1/3}(\jap{\eta}^{1/3} - \jap{\xi}^{1/3})} - 1} + \abs{\frac{w(0,\xi)}{w(0,\eta)}- 1}. \label{ineq:Tetabrk}
\end{align}
The first term in \eqref{ineq:Tetabrk} is treated using $\abs{e^x - 1} \leq xe^x$ as above, and Lemma \ref{lem:tildewcomp} (and Lemma \ref{lem:wapproxtildw}), for some $c > 0$ we have,   
\begin{align*}
\frac{w(0,\eta)}{w(0,\xi)}\abs{e^{r(K\epsilon)^{1/3}(\jap{\eta}^{1/3} - \jap{\xi}^{1/3})} - 1} & \lesssim \frac{(K\epsilon)^{1/3} \jap{\eta-\xi}}{\jap{\eta,\xi}^{2/3}}e^{c(K\epsilon)^{1/3}\jap{\eta-\xi}^{1/3}} \\ &  \approx \frac{(K\epsilon)^{1/3} \jap{\eta-\xi}}{\jap{\eta,\xi}^{2/3} + \jap{k,\ell}^{2/3}} e^{c(K\epsilon)^{1/3}\jap{\eta-\xi}^{1/3}}. 
\end{align*}
The second term in \eqref{ineq:Tetabrk} is a little more subtle. 
Using Lemma \ref{lem:wapproxtildw}, 
\begin{align*}
\abs{\frac{w(0,\eta)}{w(0,\xi)}- 1} & = \frac{1}{w(0,\xi)}\abs{\int_{\abs{z} \leq 1} \varphi(z)\left(\tilde{w}(0,\eta+z)- \tilde{w}(0,\xi+z)\right) dz} 
 \lesssim \sup_{z\in(0,1)} \abs{\frac{\tilde{w}(0,\eta+z)}{\tilde{w}(0,\xi+z)} - 1}.     
\end{align*}
By \eqref{ineq:100ctrl}, we have $\abs{N(\eta+z) - N(\xi+z)} \leq 1$ for any $z \in (-1,1)$.  
Consider first the case $N(\eta+z) = N(\xi+z)$. 
By \eqref{ineq:100ctrl} and the definition of $N$ in \eqref{def:N},  for some universal $c > 0$,
\begin{align*}
\abs{\frac{\tilde{w}(0,\eta+z)}{\tilde{w}(0,\xi+z)} - 1} & = \abs{ \left(\frac{\abs{\xi+z}}{\abs{\eta+z}}\right)^{2N(\eta+z)} - 1} \\ 
& \leq 2N(\eta+z) \abs{\log \left(1+ \frac{\abs{\xi+z}-\abs{\eta+z}}{\abs{\eta+z}}\right)} e^{2N(\eta+z)\abs{\log\left(1 + \frac{\abs{\xi+z}-\abs{\eta+z}}{\abs{\eta+z}}\right)}} \\ 
& \leq \frac{2N(\eta+z)\abs{\eta-\xi}}{\abs{\eta+z}} e^{2N(\eta+z)\frac{\abs{\eta-\xi}}{\abs{\eta+z}}} \\ 
& \approx \frac{(K\epsilon)^{1/3} \abs{\eta-\xi}}{\jap{\eta,\xi}^{2/3} + \jap{k,\ell}^{2/3}} e^{c(K\epsilon)^{1/3}\jap{\eta-\xi}^{1/3}}, 
\end{align*}
which is sufficient.
The other cases are analogous and are omitted for the sake of brevity. 
This completes the case $\frac{1}{10}\abs{k,\ell} \leq \abs{\eta,\xi} \leq 10\abs{k,\ell}$. 

\textbf{Case $\abs{k,\ell} \geq 10\abs{\eta,\xi}$:} 
In this case, recall \eqref{def:TkTeta} and write 
\begin{align*}
\abs{\frac{G(t,k,\eta)}{G(t,\ell,\xi)}- 1} & \leq \abs{\frac{e^{r(K\epsilon)^{1/3}\jap{\eta}^{1/3}} (w(t,\eta))^{-1} - e^{r(K\epsilon)^{1/3}\jap{\xi}^{1/3}} (w(t,\xi))^{-1}}{e^{r(K\epsilon)^{1/3}\jap{\ell}^{1/3}}}} + \abs{\frac{e^{r(K\epsilon)^{1/3}\jap{k}^{1/3}} - e^{r(K\epsilon)^{1/3}\jap{\ell}^{1/3}}}{e^{r(K\epsilon)^{1/3}\jap{\ell}^{1/3}}}}. 
\end{align*}
The second term is treated as in \eqref{ineq:TkCase1}. For the first term, we use that \eqref{ineq:100ctrl} implies $\abs{k} \approx \abs{\ell}$ and that $\abs{\ell} \geq \frac{1}{5}\abs{\eta,\xi}$ to deduce from Lemmas \ref{lem:wapproxtildw}  and \eqref{lem:growthtildw} to deduce that there exists some $c > 0$ such that 
\begin{align*}
\abs{\frac{e^{r(K\epsilon)^{1/3}\jap{\eta}^{1/3}} (w(t,\eta))^{-1} - e^{r(K\epsilon)^{1/3}\jap{\xi}^{1/3}} (w(t,\xi))^{-1}}{e^{r(K\epsilon)^{1/3}\jap{\ell}^{1/3}}}} & \lesssim e^{-c(K\epsilon)^{1/3}\jap{\ell}^{1/3}} \lesssim \frac{1}{(K\epsilon)^{2/3}\abs{\ell}^{2/3}} e^{-\frac{c}{2}(K\epsilon)^{1/3}\jap{\ell}^{1/3}}, 
\end{align*}
which is sufficient. The case $10\abs{k,\ell} \leq \abs{\eta,\xi}$ proceeds in an analogous manner. 
Note that the case $\abs{k,\ell} \geq 10\abs{\eta,\xi}$ does not require the restriction on time, and hence Lemma \ref{lem:commG} holds under only this hypothesis as well.    
\end{proof}

We prove the moment estimate on $G$, Lemma \ref{lem:MomentG}. 
\begin{proof}[Lemma \ref{lem:MomentG}]
Differentiating $G$ we have 
\begin{align*}
\partial_\eta G = -r \frac{\eta}{3\jap{\eta}^{5/3}} \frac{e^{r\jap{\eta}^{1/3}}}{w(t,\eta)} - \frac{\partial_\eta w(t,\eta)}{w(t,\eta)} \frac{e^{r\jap{\eta}^{1/3}}}{w(t,\eta)}. 
\end{align*}
Lemma \ref{lem:MomentG} then follows from the the convolution in \eqref{def:wconv} and Lemma \ref{lem:wapproxtildw}. 
\end{proof}
Lemma \ref{lem:AMomentEquiv} follows quickly from Lemma \ref{lem:MomentG} and is hence omitted for the sake of brevity. 
  

\section{Fourier analysis and Gevrey spaces} \label{apx:Gev}
For $f(x,v)$ we define the Fourier transform $\hat{f}(k,\eta)$, where $(k,\eta) \in \Integer \times \Real$, and the inverse Fourier transform via
\begin{align*} 
\hat{f}(k,\eta)  = \frac{1}{(2\pi)}\int_{\Torus \times \Real} e^{-i x k - iy\eta} f(x,v) dx dv, \quad f(x,v) = \frac{1}{(2\pi)}\sum_{k \in \Integer} \int_{\Real} e^{i x k + iy\eta} \hat{f}(k,\eta) d\eta. 
\end{align*} 
Paraproducts are defined in \S\ref{sec:paranote} using the Littlewood-Paley dyadic decomposition.  
Let $\psi \in C_0^\infty(\Real_+;\Real_+)$ be such that $\psi(\xi) = 1$ for $\xi \leq 1/2$ and $\psi(\xi) = 0$ for $\xi \geq 3/4$ and define $\chi(\xi) = \psi(\xi/2) - \psi(\xi)$, supported in the range $\xi \in (1/2,3/2)$. 
For $f \in L^2(\Torus \times \Real)$, 
\begin{align} 
f_{M}  = \chi(M^{-1}\abs{\grad})f, \quad\quad\quad f_{< M}  = \sum_{K \in 2^{\Integers}: K < M} f_K, \label{def:fM}
\end{align}
where we mean that the sum runs over the dyadic integers $M = ...,2^{-j},...,1/4,1/2,1,2,4,...,2^{j},...$, which defines the decomposition (in the $L^2$ sense)  
\begin{align*} 
f = \sum_{M \in 2^\Integers} f_M.  
\end{align*}
We make use of the notation 
\begin{align*} 
f_{\sim M} = \sum_{K \in 2^{\Integer}: \frac{1}{C}M \leq K \leq CM} f_{K}. 
\end{align*}
for some constant $C$ which is independent of $M$; the exact value of $C$ is not important provided it is independent of $M$. 
There holds the almost orthogonality and the approximate projection property 
\begin{subequations} \label{ineq:LPOrthoProject}
\begin{align} 
\norm{f}^2_2 & \approx \sum_{M \in 2^{\Integers}} \norm{f_M}_2^2 \\
 \norm{f_M}_2 & \approx  \norm{(f_{M})_{\sim M}}_2, 
\end{align}
\end{subequations}
and more generally, if $f = \sum_{j} D_j$ for any $D_j$ with $\frac{1}{C}2^{j} \subset \textup{supp}\, D_j \subset C2^{j}$, 
\begin{align} 
\norm{f}^2_2 \approx_C \sum_{j \in \Integers} \norm{D_j}_2^2. \label{ineq:GeneralOrtho}
\end{align}

Next we give useful, basic inequalities. The first set are Young's inequality type; see e.g. \cite{BMM13}. 
\begin{lemma}  
\begin{itemize} 
\item[(a)] Let
  $,G(k,\eta) \in L^2(\Integer^d \times \Real^d)$ and
  $\jap{k}^\sigma H(t,k) \in L^2(\Integer^d)$ for $\sigma > d/2$. Then, for any $t \in \Real$,
\begin{align} 
  \norm{\sum_{\ell} \int H(t,\ell) G(k-\ell,\eta-t\ell)}_{L^2_{k,\eta}} \lesssim_{d,\sigma} \norm{G}_{L^2_{k,\eta} } \norm{\jap{k}^\sigma H(t)}_{L_k^2}. \label{ineq:L2L1}      
\end{align} 
If instead $\jap{k}^\sigma G(t,k,\eta) \in L^2(\Integer^d \times \Real^d)$ for $\sigma > d/2$, then for any $t \in \Real$,
\begin{align}
  \norm{\sum_{\ell} \int H(t,\ell) G(k-\ell,\eta-t\ell)}_{L^2_{k,\eta}} \lesssim_{d,\sigma} \norm{\jap{k}^{\sigma} G}_{L^2_{k,\eta} } \norm{H(t)}_{L_k^2}. \label{ineq:L1L2}      
\end{align} 
\end{itemize}
\end{lemma} 
\begin{lemma}[$L^2$ Trace] \label{lem:SobTrace}
Let $g \in H^s(\Real^d)$ with $s > (d-1)/2$ and $C \subset \Real^d$ be an arbitrary straight line. Then there holds, 
\begin{align*} 
\norm{g}_{L^2(C)} \lesssim_s \norm{g}_{H^s}. 
\end{align*} 
\end{lemma}  
The next set of inequalities show that one can often gain on the index of regularity when comparing frequencies which are not too far apart (provided $0 < s < 1$).
\begin{lemma}
Let $0 < s < 1$, $x,y>0$, and $K>1$. 
\begin{itemize} 
\item[(i)] There holds
\begin{align} 
\abs{x^s - y^s} \leq s \max(x^{s-1},y^{s-1})\abs{x-y}. \label{ineq:TrivDiff}
\end{align}
so that if $|x-y|<\frac{x}{K}$,
\begin{align} 
\abs{x^s - y^s} \leq \frac{s}{(K-1)^{1-s}}\abs{x-y}^s. \label{lem:scon}
\end{align} 
Note $\frac{s}{(K-1)^{1-s}} < 1$ as soon as $s^{\frac{1}{1-s}} + 1 < K$. 
\item[(ii)] There holds 
\begin{align} 
\abs{x + y}^s \leq \left(\frac{\max(x,y)}{x+y}\right)^{1-s}\left(x^s + y^s\right), \label{lem:smoretrivial}
\end{align}  
so that, if $\frac{1}{K}y \leq x \leq Ky$,
\begin{align} 
\abs{x + y}^s \leq \left(\frac{K}{1 + K}\right)^{1-s}\left(x^s + y^s\right). \label{lem:strivial}
\end{align} 
\end{itemize}
\end{lemma}
For all $x \geq 0$, $\alpha > \beta \geq 0$, $C,\delta > 0$, $\sigma > 0$, 
\begin{subequations}
\begin{align} 
e^{Cx^{\beta}} & \lesssim_{C,\delta,\beta,\alpha} e^{\delta x^{\alpha}},   \label{ineq:IncExp} \\ 
e^{-\delta x^{\alpha}} \lesssim_{\delta,\sigma,\alpha} \frac{1}{\jap{x}^{\sigma}}. \label{ineq:SobExp}
\end{align}
\end{subequations}

\section*{Acknowledgments}
The author would like to thank Bedros Afeyan, Pierre Germain, Zaher Hani, Nader Masmoudi, Clement Mouhot, and Zhiwu Lin for stimulating discussions.  
The author would also like to thank Siming He for correcting several typographical errors.

\bibliographystyle{abbrv} \bibliography{eulereqns}

\end{document}